\documentclass[11pt,reqno]{amsart}
\usepackage[margin=1.25in,marginparwidth=1in,marginparsep=0.15in,centering,letterpaper,dvips]{geometry}

\usepackage{amsmath, amssymb, amscd, bbm, mathrsfs, url, pinlabel, hyperref, verbatim, xargs}
\usepackage{color,dcpic,latexsym,graphicx,epstopdf,comment}
\usepackage[all]{xy}
\usepackage{enumitem}
\usepackage{tikz}
\usetikzlibrary{decorations.markings}

\usepackage{mathtools}
\DeclarePairedDelimiter\floor{\lfloor}{\rfloor}

\newtheorem {theorem}{Theorem}
\newtheorem {lemma}[theorem]{Lemma}
\newtheorem {proposition}[theorem]{Proposition}
\newtheorem {corollary}[theorem]{Corollary}
\newtheorem {conjecture}[theorem]{Conjecture}
\newtheorem {definition}[theorem]{Definition}

\theoremstyle{remark}

\newtheorem {remark}[theorem]{Remark}
\newtheorem {example}[theorem]{Example}

\numberwithin{equation}{section}
\numberwithin{theorem}{section}

\newlist{pcases}{enumerate}{1}
\setlist[pcases]{
  label=\bf{Case~\arabic*:}\protect\thiscase.~,
  ref=\arabic*,
  align=left,
  labelsep=0pt,
  leftmargin=0pt,
  labelwidth=0pt,
  parsep=0pt
}
\newcommand{\case}[1][]{%
  \if\relax\detokenize{#1}\relax
    \def\thiscase{}%
  \else
    \def\thiscase{~#1}%
  \fi
  \item
}

\newcommand{\rank}{\operatorname{rank}}
\newcommand{\alg}{\mathrm{alg}}
\newcommand{\ZZ}{\mathbb{Z}}
\newcommand{\Z}{\mathbb{Z}}
\newcommand{\N}{\mathbb{N}}
\newcommand{\R}{\mathbb{R}}
\newcommand{\Ralg}{\R^\alg}
\newcommand{\C}{\mathbb{C}}

\newcommand{\Q}{\mathbb{Q}}
\newcommand{\Qbar}{\overline{\Q}}
\newcommand{\img}{\operatorname{Im}}
\newcommand{\re}{\operatorname{Re}}

\newcommand{\cN}{\mathcal{N}}

\newcommand{\cR}{\mathcal{R}}

\newcommand{\Hom}{\operatorname{Hom}}
\newcommand{\ssm}{\smallsetminus}
\newcommand{\tr}{\operatorname{tr}}
\newcommand{\ab}{\operatorname{ab}}
\newcommand{\diag}{\operatorname{diag}}
\newcommand{\mirror}{\overline}
\newcommand{\Gal}{\mathrm{Gal}}
\newcommand{\twomatrix}[4]{{\left(\begin{matrix} #1 & #2 \\ #3 & #4 \end{matrix}\right)}}
\newcommand{\pt}{\mathrm{pt}}

\DeclareMathOperator{\lcm}{lcm}
\newcommand{\Wh}{\mathit{Wh}}

\newcommand{\KHI}{\mathit{KHI}}
\newcommand{\ad}{\operatorname{ad}}
\newcommand{\ch}{X}
\newcommand{\chsl}{{\ch}_{SL_2}}

\newcommand{\abs}[1]{\lvert#1\rvert}
\newcommand{\norm}[1]{\left\lVert #1 \right\rVert}

\renewcommand{\epsilon}{\varepsilon}
\newcommand{\pertdata}{{\{\iota_k,\chi_k\}}}

\newcommand{\Aut}{\operatorname{Aut}}
\newcommand{\id}{\operatorname{id}}
\newcommand{\Hol}{\operatorname{Hol}}

\title{$SU(2)$-cyclic surgeries and the pillowcase}
\date{}
\author{Steven Sivek}
\email{s.sivek@imperial.ac.uk}
\address{Imperial College London}

\author{Raphael Zentner}
\email{raphael.zentner@mathematik.uni-regensburg.de}
\address{Universit\"{a}t Regensburg}

\begin{document}

\begin{abstract}
We study knots in $S^3$ with infinitely many $SU(2)$-cyclic surgeries, which are Dehn surgeries such that every representation of the resulting fundamental group into $SU(2)$ has cyclic image.  We show that for every such nontrivial knot $K$, its set of $SU(2)$-cyclic slopes is bounded and has a unique limit point, which is both a rational number and a boundary slope for $K$.  We also show that such knots are prime and have infinitely many instanton L-space surgeries.  Our methods include the application of holonomy perturbation techniques to instanton knot homology, using a strengthening of recent work by the second author.
\end{abstract}

\maketitle

\section{Introduction}

The cyclic surgery theorem of Culler, Gordon, Luecke, and Shalen \cite{cgls} says that if a knot in $S^3$ is neither the unknot nor a torus knot, and if both $r$- and $s$-surgery on that knot give manifolds with cyclic fundamental group, then the distance $\Delta(r,s)$ between these slopes is at most 1, and hence there are at most two such nontrivial surgeries.  This is the strongest possible result: Moser \cite{moser} showed that torus knots have infinitely many lens space surgeries, and for example the pretzel knot $P(-2,3,7)$ has two lens space surgeries, of slopes 18 and 19.

In this paper we study a weaker property of surgeries on knots in $S^3$.  Following Lin \cite{lin}, we say that a 3-manifold $Y$ is \emph{$SU(2)$-cyclic} if all representations $\pi_1(Y) \to SU(2)$ have cyclic image.  For example, $\frac{37}{2}$-surgery on the pretzel knot $P(-2,3,7)$ is not a lens space, but it is known to be $SU(2)$-cyclic.  Kronheimer and Mrowka \cite{km-su2} proved that any $SU(2)$-cyclic surgery of slope $r$ on a nontrivial knot $K \subset S^3$ satisfies $|r|>2$, and Lin \cite{lin} proved that any two $SU(2)$-cyclic slopes $r_i = \frac{p_i}{q_i}$ ($i=1,2$) for $K$ satisfy the inequality $\Delta(r_1,r_2) \leq |p_1|+|p_2|$.  We will call a nontrivial knot \emph{$SU(2)$-averse} if it has infinitely many $SU(2)$-cyclic surgeries.

If $K$ is $SU(2)$-averse, then the slopes of its $SU(2)$-cyclic surgeries form an infinite subset of $\mathbb{RP}^1 = \R \cup \{\infty\}$, so they must have a limit point.  One can deduce from Lin's inequality or from Theorem~\ref{thm:open-pillowcase-image} below that this limit point is unique; we call this point the \emph{limit slope} of $K$.  For example, the torus knot $T_{p,q}$ has limit slope $pq$, since $(pq+\frac{1}{n})$-surgery on $T_{p,q}$ is a lens space for all integers $n \neq 0$.

By studying the $SU(2)$ character variety of the exterior $S^3 \ssm N(K)$, and its image under restriction in the character variety of the peripheral torus $\partial N(K)$, we prove the following.

\begin{theorem} \label{thm:main}
Let $K \subset S^3$ be an $SU(2)$-averse knot.  Then its limit slope $r(K)$ is a rational number, with $|r(K)| > 2$, and it is a boundary slope for $K$.  Moreover, $K$ is prime, and the manifold $S^3_{s}(K)$ constructed by $s$-surgery on $K$ is an instanton L-space for all rational $s \geq \lceil r(K)\rceil - 1$ if $r(K)>0$, or for all $s \leq \lfloor r(K)\rfloor + 1$ if $r(K)<0$.
\end{theorem}

\begin{corollary}
Let $K$ be a nontrivial knot in $S^3$.  Then there is a constant $N$ such that there is an irreducible representation
\[ \pi_1(S^3_r(K)) \to SU(2) \]
for all rational $r$ with $|r| > N$.
\end{corollary}

\begin{proof}
If the set of $SU(2)$-cyclic slopes for a knot $K$ is unbounded, then $K$ must be $SU(2)$-averse with limit slope $\infty$, contradicting $r(K) \in \Q$.
\end{proof}

After the initial version of this paper appeared, Baldwin and the first author \cite{bs-lspace} proved that if a nontrivial knot $K$ has an instanton L-space surgery of positive slope, then $K$ is fibered and strongly quasipositive, and the set of instanton L-space slopes for $K$ is $[2g(K)-1,\infty) \cap \Q$.  Thus we have the following.

\begin{corollary} \label{cor:fibered-sqp}
If $K$ is $SU(2)$-averse then $K$ is fibered, either $K$ or its mirror is strongly quasipositive, and the limit slope satisfies $|r(K)| > 2g(K)-1$.
\end{corollary}

We remark that the pretzel $P(-2,3,7)$ has infinitely many instanton L-space surgeries, as does any knot with a lens space surgery, but we will show in Example~\ref{ex:pretzel-non-averse} that it is not $SU(2)$-averse.  Thus being $SU(2)$-averse is a strictly stronger condition than having infinitely many instanton L-space surgeries.

In Theorem~\ref{thm:main}, the finiteness and rationality of $r(K)$ are proved as Theorems~\ref{thm:limit-slope-finite} and \ref{thm:limit-slope-rational} respectively, and the bound $|r(K)|>2$ comes from Theorem~\ref{thm:open-pillowcase-image} (though $r(K) \geq 2$ is already implied by \cite{km-su2}).  The fact that $r(K)$ is a boundary slope is Theorem~\ref{thm:su2-averse-rational}; this means that there is a properly embedded, essential surface in $S^3 \ssm N(K)$ whose boundary is a nonempty union of curves of slope $r(K)$.  The primeness of $K$ is proved as Theorem~\ref{thm:averse-connected-sum}.  The assertion about instanton L-spaces, namely that for each $s=\frac{a}{b}$ in the given range the singular instanton knot homology $I^\#(S^3_{a/b}(K),\emptyset)$ defined in \cite{km-khovanov} has rank $|a|$, is Theorem~\ref{thm:l-space-surgeries}.

We can also prove something about the set of $SU(2)$-cyclic surgery slopes of an $SU(2)$-averse knot.  Combining Theorem~\ref{thm:open-pillowcase-image} and Theorem~\ref{thm:finite-lines-converse}, we have:

\begin{theorem} \label{thm:main-slopes}
Let $K$ be an $SU(2)$-averse knot with limit slope $r(K)$.  Then all $SU(2)$-cyclic slopes $\frac{m}{n}$ satisfy
\[ \left| \frac{m}{n} - r(K) \right| \leq \frac{|r(K)|}{n}, \]
and there is an integer $N \geq 1$ such that $r(K) + \frac{1}{kN}$ is an $SU(2)$-cyclic slope for all but finitely many $k \in \Z$.
\end{theorem}

A knot $K \subset S^3$ is said to be \emph{small} if there are no closed, incompressible surfaces in its exterior other than boundary-parallel tori; examples include all two-bridge knots \cite{hatcher-thurston} and Montesinos knots with at most three rational tangles \cite{oertel}.  In this case we know even more, as proved in Theorems~\ref{thm:small-knot-integer} and \ref{thm:det-limit-slope} and Proposition~\ref{prop:small-r-6}.
\begin{theorem} \label{thm:main-small}
If $K$ is a small, $SU(2)$-averse knot with Alexander polynomial $\Delta_K(t) = \sum_{j=-d}^d a_jt^j$, then the limit slope $r(K)$ is an integer satisfying $|r(K)| \geq 6$ and
\[ |r(K)| \geq 1 + \sum_{j=-d}^d |a_j|. \]
In this case, all but finitely many conjugacy classes of irreducible representations
\[ \rho: \pi_1(S^3 \ssm K) \to SU(2) \]
satisfy $\rho(\mu^{r(K)}\lambda) = \pm I$, where $\mu$ and $\lambda$ are a meridian and longitude of $K$.
\end{theorem}

Finally, by studying how $SU(2)$-averseness behaves under satellite operations we can also say something about the potential existence of $SU(2)$-averse knots other than torus knots.  Notably, we show in Proposition~\ref{prop:cable-converse} that if $r(K) = \frac{p}{q}$ with $q \geq 2$, then the $(p,q)$-cable of $K$ is $SU(2)$-averse with limit slope $pq$.  However, by analogy with the cyclic surgery theorem, we conjecture that this never happens:

\begin{conjecture} \label{conj:main-conjecture}
If $K$ is a nontrivial $SU(2)$-averse knot, then $K$ is a torus knot.
\end{conjecture}

\noindent Indeed, in Section~\ref{sec:evidence}, we verify Conjecture~\ref{conj:main-conjecture} for all alternating Montesinos knots with at most three rational tangles (including all two-bridge knots) and for all knot types through 11 crossings.  The conjecture also holds for algebraic knots.

\subsection*{Outline}

The proof of Theorem~\ref{thm:main} comes from studying the image of the $SU(2)$ character variety of $\pi_1(S^3 \ssm N(K))$ in the \emph{pillowcase}, which is the $SU(2)$ character variety of $\pi_1(\partial N(K))$.  Up to conjugacy, every representation $\rho: \pi_1(S^3 \ssm K) \to SU(2)$ satisfies
\[ \rho(\mu) = \twomatrix{e^{i\alpha}}{0}{0}{e^{-i\alpha}}, \qquad \rho(\lambda) = \twomatrix{e^{i\beta}}{0}{0}{e^{-i\beta}} \]
for some constants $\alpha$ and $\beta$, and hence determines a unique point $(\alpha,\beta)$ of the pillowcase
\[ P \cong \frac{(\R/2\pi\Z) \times (\R/2\pi\Z)}{(\alpha,\beta) \sim (-\alpha,-\beta)}. \]
A slope $\frac{m}{n}$ is $SU(2)$-cyclic precisely when there are no images of irreducible characters of $\pi_1(S^3 \ssm K)$ along the line $m\alpha+n\beta = 0$ in $P$.

Our key observation, developed in Section~\ref{sec:pillowcase}, is that the line $m\alpha+n\beta=0$ can only avoid a given non-constant path $\gamma: [0,1] \to P$ through images of irreducibles if $-\frac{m}{n}$ is a reasonable Diophantine approximation to the slope of the line from $\gamma(0)$ to $\gamma(1)$.  We deduce that if $K$ is $SU(2)$-averse, then any such path must be a straight line of slope $-r(K)$.  (Conversely, we show in Section~\ref{sec:converse} that if every such path is a straight line of rational slope $-r$, then $K$ is $SU(2)$-averse with $r(K)=r$.)

In Section~\ref{sec:finiteness}, we show that $r(K)\neq\infty$ by finding (in nearly all cases) an arc of images of irreducibles which is not a vertical line.   We do this by making use of Kronheimer and Mrowka's  instanton knot homology \cite{km-excision}, or $\KHI$, which by construction is closely related to the space of irreducible representations $\pi_1(S^3 \ssm K) \to SU(2)$ which send a meridian to a traceless matrix.  Using a strengthened version of recent work of the second author \cite{zentner}, we can perturb the Chern-Simons functional which defines $\KHI(K)$ so that if $K$ has $SU(2)$-cyclic surgeries and no such arc of irreducibles, then $\KHI(K)$ is the same as $\KHI$ of the unknot, which means that $K$ must have been unknotted.  More precisely, we prove the following.
{
\renewcommand{\thetheorem}{\ref{thm:pillowcase-near-pi/2}}
\begin{theorem}
Let $K$ be a nontrivial knot in $S^3$.  Then at least one of the following is true:
\begin{enumerate}
\item For every $\beta \in \R/2\pi\Z$, there is an irreducible $\rho:\pi_1(S^3 \ssm K) \to SU(2)$ such that
\[ \rho(\mu) = \twomatrix{i}{0}{0}{-i} \qquad\mathrm{and}\qquad \rho(\lambda) = \twomatrix{e^{i\beta}}{0}{0}{e^{-i\beta}}. \]
\item There is a constant $\epsilon > 0$ and an arc of irreducible representations
\[ \rho_s: \pi_1(S^3 \ssm K) \to SU(2), \qquad s \in \left[\frac{\pi}{2}, \frac{\pi}{2}+\epsilon\right) \]
such that $\rho_s(\mu) = \diag(e^{is},e^{-is})$ for all $s$.
\end{enumerate}
\end{theorem}
\addtocounter{theorem}{-1}
}
\noindent This may be viewed as a strengthening of \cite[Corollary~7.17]{km-excision}; see also related work of Herald \cite{herald}, Collin--Steer \cite{collin-steer}, or Heusener--Kroll \cite{heusener-kroll}, in which similar results are proved under assumptions on the equivariant signatures of $K$.

\begin{remark}
If $K$ is small then the second case of Theorem~\ref{thm:pillowcase-near-pi/2} always applies, since the other case is ruled out by first arguing as in the proof of Theorem~\ref{thm:c-rational} that it would force $\infty$ to be a boundary slope and then applying \cite[Theorem~2.0.3]{cgls}.
\end{remark}

We mention here the needed strengthening of the approximation result in \cite{zentner}, which is of independent interest.  Recall that a shearing isotopy $(\zeta_t)$ on the 2-dimensional torus is an isotopy through area-preserving maps of the form
\begin{equation} \label{eq:shearing-isotopy}
\zeta_t\colon (x,y) \mapsto (x,y) + tf(w \cdot (x,y))\, v\, ,
\end{equation}
where $v,w \in \Z^2$ are orthogonal vectors with $w$ nonzero, $f: \R \to \R$ is a $2\pi$-periodic function, and $\cdot$ denotes the standard inner product on $\R^2$.  We say that an isotopy $(\phi_t)$ is a piecewise isotopy through shearings if there is a partition $0=t_0 < t_1 < \dots < t_n = 1$ such that for $t_i \leq t \leq t_{i+1}$, the map $\phi_t$ is given as
 \[
 	\phi_t = \zeta^{(i)}_{t-{t_i}} \circ \phi_{t_i} \, 
 \]
for some shearing isotopy $(\zeta^{(i)}_{s})_{s \in [0,1]}$.  For fixed $p\in T^2$, the path $t \mapsto \phi_t(p)$ of a piecewise isotopy through shearings is continuous but need not be smooth.  The case $r=0$ of the following was proved in \cite{zentner}.

{
\renewcommand{\thetheorem}{\ref{C1 approximation}}
\begin{theorem}
 Let $r \geq 0$. Let $(\psi_t)_{t \in [0,1]}$ be an isotopy through area-preserving maps $\psi_t: T^2 \to T^2$, and let $\epsilon > 0$. Then there is a piecewise isotopy through shearings $(\phi_t)_{t \in [0,1]}$ such that $\phi_t$ is $\epsilon$-close to $\phi_t$ in the $C^r$-topology for all $t \in [0,1]$. 
\end{theorem}
\addtocounter{theorem}{-1}
}

The rationality of $r(K)$ follows in Section~\ref{sec:rationality} from some foundational results in real algebraic geometry.  Using the fact that the pillowcase image of the $SU(2)$ character variety is a semi-algebraic set defined over the field of real algebraic numbers, and hence many of its points have algebraic coordinates, we manage to prove that a straight line in this image must have rational slope.  This allows us to further show that if $r(K) = \frac{p}{q}$, then along any such path of representations $\rho$, there is a real constant $c$ such that
\[ \rho(\mu^p\lambda^q) = \twomatrix{e^{ic}}{0}{0}{e^{-ic}} \]
up to conjugacy, where $e^{ic}$ is an algebraic number.

In Section~\ref{sec:a-polynomial}, we study $SU(2)$-averse knots via the $A$-polynomial, which describes the image of the $SL_2(\C)$ character variety of $\pi_1(S^3 \ssm K)$ in the $SL_2(\C)$ character variety of the peripheral torus.  The results of Section~\ref{sec:pillowcase} imply that if $r(K) = \frac{p}{q}$ then $A_K(M,L)$ is a multiple of $M^pL^q - e^{ic}$, and hence by applying deep results from \cite{ccgls} about the Newton polygon and edge polynomials of $A_K(M,L)$, we see that $r(K)$ is a boundary slope and that $e^{ic}$ is in fact a root of unity.  Building on this, Section~\ref{sec:small} uses further results of Boyer and Zhang \cite{boyer-zhang-seminorms} for small knots to prove Theorem~\ref{thm:main-small}.

In Section~\ref{sec:converse}, we use the fact that $e^{ic}$ is a root of unity to prove Theorem~\ref{thm:finite-lines-converse}, asserting that if the pillowcase image of the $SU(2)$ character variety consists mainly of line segments of slope $-r \in \Q$, then $K$ must be $SU(2)$-averse.  This allows us to prove in Section~\ref{sec:l-spaces} that $SU(2)$-averse knots have infinitely many instanton L-space surgeries, and our applications include Theorem~\ref{thm:smoothly-slice}, which says that smoothly slice knots are not $SU(2)$-averse.

Finally, in Section~\ref{sec:satellites} we investigate when a satellite of a given knot can be $SU(2)$-averse.  We prove the following, a combination of Propositions~\ref{prop:satellite-properties} and \ref{prop:satellites-nonzero-winding}:
\begin{theorem} \label{thm:main-satellites}
Let $K$ be a nontrivial knot, and suppose that some satellite $P(K)$ with winding number $w$ is $SU(2)$-averse.
\begin{itemize}
\item If $P(U)$ is not the unknot, then it is also $SU(2)$-averse, and $r(P(K)) = r(P(U))$.
\item If $w\neq 0$, then $K$ is $SU(2)$-averse, with $r(P(K)) = w^2r(K)$.
\end{itemize}
\end{theorem}
\noindent In particular, we prove that $SU(2)$-averse knots are prime (Theorem~\ref{thm:averse-connected-sum}) and we completely determine when a cable of a given knot is $SU(2)$-averse (Theorem~\ref{thm:cables}), leading to a proof of Conjecture~\ref{conj:main-conjecture} for algebraic knots (Corollary~\ref{cor:algebraic-knots}).  Similarly, in Section~\ref{sec:evidence} we apply many of the above results to prove that several other classes of knots (as described following Conjecture~\ref{conj:main-conjecture} above) are not $SU(2)$-averse.  In doing so, we make use of the following strengthening of a theorem of Crowell \cite[Theorem~6.5]{crowell}.
{
\renewcommand{\thetheorem}{\ref{prop:alternating-det-cr}}
\begin{proposition}
If $K$ is a prime alternating knot, then $\det(K) \geq 3c(K)-8$ unless $K$ is a $(2,2k+1)$ torus knot or a twist knot.
\end{proposition}
}

\subsection*{Acknowledgments}
We thank Hans Boden, Marc Culler, Cynthia Curtis, Nathan Dunfield, Paul Feehan, Michael Heusener, and Tom Mrowka for helpful conversations.  We thank the anonymous referee for their careful reading of this paper and many comments which improved the exposition.  We are also grateful to the Max Planck Institute for Mathematics for its hospitality during a significant portion of this work. The second author is also grateful for support by the SFB `Higher invariants' (funded by the Deutsche Forschungsgemeinschaft (DFG)) at the University of Regensburg, and for support by a Heisenberg fellowship of the DFG.

\section{Background} \label{sec:background}

To any space $Y$ we can associate its $SU(2)$ representation variety
\[ R(Y) = \Hom(\pi_1(Y), SU(2)). \]
We can realize $R(Y)$ as a real algebraic subset of some $\R^n$, given a finite presentation
\[ \pi_1(Y) = \langle g_1,g_2,\dots,g_k \mid w_1,w_2,\dots,w_l \rangle. \]
Namely, we assign to each $g_j$ four real variables $a_j,b_j,c_j,d_j$ satisfying the relation $a_j^2+b_j^2+c_j^2+d_j^2 = 1$, so that we can set
\[ \rho(g_j) = \twomatrix{a_j+ib_j}{c_j+id_j}{-c_j+id_j}{a_j-ib_j}, \qquad
\rho(g_j^{-1}) = \twomatrix{a_j-ib_j}{-c_j-id_j}{c_j-id_j}{a_j+ib_j}. \]
Then each $2\times 2$ matrix equation $\rho(w_i) = I$ gives four additional polynomial relations among the various $a,b,c,d$, and the collection of these relations defines $R(Y) \subset \R^{4k}$ as the zero locus of some collection of polynomials with coefficients in $\Z$.  (Since $\Z[a_1,\dots,d_k]$ is Noetherian, we only need finitely many of these relations even if $\pi_1(Y)$ was merely finitely generated.)  This presents $R(Y)$ as a closed subspace of $(S^3)^k$, hence it is compact.

A representation $\pi_1(Y) \to SU(2)$ is irreducible if and only if its image does not lie in some $U(1)$ subgroup, or equivalently if its image is not abelian.  The space
\[ R^*(Y) = \{ \rho: \pi_1(Y) \to SU(2) \mid \rho \mathrm{\ irreducible} \} \]
can therefore be cut out from $R(Y)$ by imposing the additional constraint
\[ \rho(g_ig_j) \neq \rho(g_jg_i) \mathrm{\ for\ some\ }i,j. \]
Equivalently, we can sum the squares of the magnitudes of the entries of the matrices $\rho(g_ig_j)-\rho(g_jg_i)$ over all $1 \leq i < j \leq k$, and this is positive if and only if $\rho$ is irreducible, so $R^*(Y)$ is a semi-algebraic set defined by the polynomial relations appearing in $R(Y)$ plus an additional polynomial inequality, again with coefficients in $\Z$.

Dividing out the action of $SU(2)$ by conjugation on the representation variety, we get the $SU(2)$ character variety and its irreducible part:
\[ \ch(Y) = R(Y) / SU(2), \qquad \ch^*(Y) = R^*(Y) / SU(2). \]
Since $R(Y)$ is compact, so is its quotient $\ch(Y)$.

In the case of a torus $T^2$, we can take a pair of elements $\mu,\lambda$ which generate $\pi_1(T^2) \cong \Z^2$, and since they commute their images under any representation $\rho: \pi_1(T^2) \to SU(2)$ can be simultaneously diagonalized as
\[ \rho(\mu) = \twomatrix{e^{i\alpha}}{0}{0}{e^{-i\alpha}}, \qquad
\rho(\lambda) = \twomatrix{e^{i\beta}}{0}{0}{e^{-i\beta}}. \]
The pairs $\alpha, \beta$ lie in a torus $T = (\R/2\pi\Z) \times (\R/2\pi\Z)$, but they are not uniquely determined, since conjugation by $\left(\begin{smallmatrix} 0&-1\\1&0 \end{smallmatrix}\right)$ induces an involution $\iota$ sending $(\alpha, \beta)$ to $(2\pi-\alpha, 2\pi-\beta)$.  Each conjugacy class in $SU(2)$ is completely determined by its trace, so in fact $\ch(T^2)$ is the quotient of $T$ by this involution:
\begin{equation} \label{eq:iota}
\ch(T^2) = \big((\R/2\pi\Z) \times (\R/2\pi\Z)\big) / \iota.
\end{equation}

The space $\ch(T^2)$ is often called the \emph{pillowcase}: it is homeomorphic to a sphere with four orbifold points of order 2, since $\iota$ has the four fixed points $(0,0)$, $(0,\pi)$, $(\pi,0)$, and $(\pi,\pi)$.  It can be constructed as the quotient of the fundamental domain $[0,\pi] \times [0,2\pi]$ by the identifications
\[ (0,\beta) \sim (0,2\pi-\beta), \qquad (\alpha,0) \sim (\alpha,2\pi), \qquad (\pi,\beta) \sim (\pi,2\pi-\beta) \]
where $0 \leq \alpha \leq \pi$ and $0 \leq \beta \leq 2\pi$, see Figure~\ref{pillowcase trefoil} below.  We can cut the pillowcase open along the lines $\alpha=0$ and $\alpha=\pi$ to get an annulus $[0,\pi] \times (\R/2\pi\Z)$, which we will call the \emph{cut-open pillowcase}.

Given a knot $K \subset S^3$, we let $R(K) = R(S^3 \ssm N(K))$ and likewise for $R^*(K)$, $\ch(K)$, and $\ch^*(K)$.  The inclusion $i: \partial N(K) \hookrightarrow S^3 \ssm N(K)$ of the boundary torus induces a restriction map on the $SU(2)$ character varieties
\[ i^*: \ch(K) \to \ch(T^2). \]
The image of $\ch(K)$ inside the cut-open pillowcase $[0,\pi] \times (\R/2\pi\Z)$ has canonical coordinates $(\alpha,\beta)$ given by taking $\mu$ and $\lambda$ to be a meridian and longitude of the knot.  (Explicitly, we have $\alpha = \arccos(\frac{1}{2}\tr(\rho(\mu)))$, where $0 \leq \arccos(\theta) \leq \pi$; and then the uniqueness of $\beta$ will follow from Proposition~\ref{prop:pillowcase-facts} below.)  It is also compact since $\ch(K)$ is.  Figure~\ref{pillowcase trefoil} shows the image of $\ch(K)$ in the pillowcase in the case where $K$ is the left-handed trefoil.  

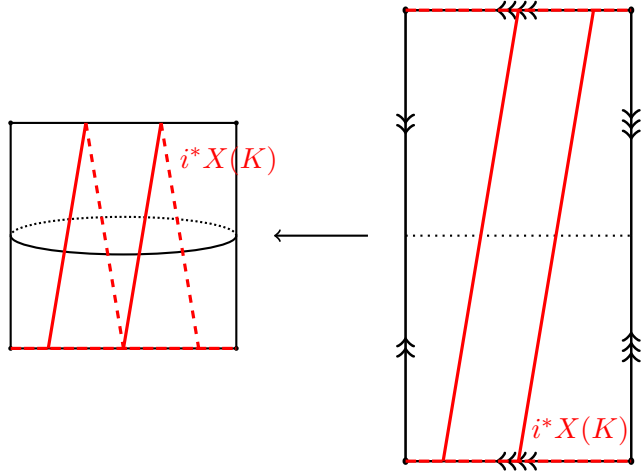
\begin{figure}
\begin{tikzpicture}[style=thick]
\begin{scope}
  \draw plot[mark=*,mark size = 0.5pt] coordinates {(0,0)(3,0)(3,3)(0,3)} -- cycle; 
  \draw (3,1.5) arc (0:-180:1.5 and 0.25);
  \draw[densely dotted] (3,1.5) arc (0:180:1.5 and 0.25);
  \begin{scope}[color=red,style=very thick]
    \draw[dash pattern=on 4pt off 1pt] (0,0) -- (3,0);
    \draw (0.5,0) -- (1,3) (1.5,0) -- (2,3);
    \draw[dashed] (1,3) -- (1.5,0) (2,3) -- (2.5,0);
    \node at (2.9,2.5) {$i^*\ch(K)$};
  \end{scope}
  \draw[->] (4.75,1.5) -- (3.5,1.5);
\end{scope}
\begin{scope}[xscale=1.5,yscale=2,shift={(3.5,-0.75)}]
  \draw plot[mark=*,mark size = 0.5pt] coordinates {(0,0)(2,0)(2,3)(0,3)} -- cycle; 
  \begin{scope}[decoration={markings,mark=at position 0.55 with {\arrow[scale=1.5]{>>}}}]
    \draw[postaction={decorate}] (0,0) -- (0,1.5);
    \draw[postaction={decorate}] (0,3) -- (0,1.5);
  \end{scope}
  \begin{scope}[decoration={markings,mark=at position 0.575 with {\arrow[scale=1.5]{>>>}}}]
    \draw[postaction={decorate}] (2,0) -- (2,1.5);
    \draw[postaction={decorate}] (2,3) -- (2,1.5);
  \end{scope}
  \begin{scope}[decoration={markings,mark=at position 0.6 with {\arrow[scale=1.5]{>>>>}}}]
    \draw[postaction={decorate}] (2,3) -- (0,3);
    \draw[postaction={decorate}] (2,0) -- (0,0);
  \end{scope}
  \draw[dotted] (0,1.5) -- (2,1.5);
  \begin{scope}[color=red,style=very thick]
    \draw (0.333,0) -- (1,3) (1,0) -- (1.666,3);
    \draw[dash pattern=on 4pt off 1pt] (0,0) -- (2,0) (0,3) -- (2,3);
    \node at (1.55,0.2) {$i^*\ch(K)$};
  \end{scope}
\end{scope}
\end{tikzpicture}
\caption{Identifications of the rectangle $[0,\pi] \times [0,2 \pi]$ yielding the pillowcase. The image of $\ch(K)$ is depicted for $K$ the left-handed trefoil.}
\label{pillowcase trefoil}
\end{figure}

The following facts are more or less standard; see e.g.\ \cite[Lemma~2.8]{lin}.

\begin{proposition} \label{prop:pillowcase-facts}
Let $\rho: \pi_1(S^3 \ssm K) \to SU(2)$ be a representation and $(\alpha,\beta)$ the coordinates of its image in the pillowcase.
\begin{enumerate}
\item Every point on the line $\beta \equiv 0 \pmod{2\pi}$ is the image of a reducible $\rho$.
\item If $\rho$ is reducible, then $\beta \equiv 0 \pmod{2\pi}$.
\item If $\rho$ is reducible and a limit of irreducible representations, then $e^{2i\alpha}$ is a root of the Alexander polynomial $\Delta_K(t)$.
\item There is a constant $\epsilon > 0$ depending only on $K$ such that if $\rho$ is irreducible, then $\epsilon < \alpha < \pi-\epsilon$.
\end{enumerate}
\end{proposition}

\begin{proof}
Composing the abelianization $\pi_1(S^3 \ssm K) \to H_1(S^3 \ssm K) \cong \Z$ with the map $k \mapsto \diag(e^{ik\alpha}, e^{-ik\alpha})$ gives a reducible $\rho$ with coordinates $(\alpha,0)$, since $\lambda$ is nullhomologous.  Conversely, if $\rho$ is reducible then its image is abelian, so $\rho$ factors through $H_1(S^3 \ssm K)$ and hence $\lambda$ lies in its kernel.  The claim about reducible limits of irreducible representations is a theorem of Klassen \cite[Theorem~19]{klassen}.

Finally, if $\alpha=0$ or $\alpha=\pi$ then $\rho(\mu) = \pm I$, and thus $\rho$ is reducible with image in $\{\pm I\}$ since $\mu$ normally generates the knot group.  In these cases $\Delta_K(e^{2i\alpha}) = \Delta_K(1) = 1$, so some neighborhood of $\rho$ in $R(K)$ consists only of reducibles.  Now if there is a sequence of irreducibles $\rho_n$ with coordinates $(\alpha_n, \beta_n)$ such that $\alpha_n \to 0$ (resp.\ $\alpha_n \to \pi$), then by compactness some subsequence converges to a representation with $\alpha$-coordinate $0$ (resp.\ $\pi$), which must then have a neighborhood with no irreducibles, giving a contradiction.
\end{proof}

\begin{remark}
The coordinates on the pillowcase itself are not quite canonical, since $(0,\beta)$ is identified with $(0,2\pi-\beta)$ and likewise for $(\pi,\beta)$ and $(\pi,2\pi-\beta)$.  However, since the image of $\ch(K)$ avoids all such points except for $(0,0)$ and $(\pi,0)$, the coordinates on the image of $\ch(K)$ are indeed well-defined, both on the pillowcase and on its cut-open variant, and so we can safely ignore this ambiguity.
\end{remark}

The pillowcase images of character varieties exhibit symmetries, as in the following two propositions.

\begin{proposition} \label{prop:pillowcase-symmetry}
The image of $\ch(K)$ inside the pillowcase is invariant under the involution $(\alpha,\beta) \mapsto (\pi-\alpha, 2\pi-\beta)$.
\end{proposition}

\begin{proof}
Suppose that the image of $\rho$ has coordinates $(\alpha,\beta)$, and hence $\rho(\mu) = \diag(e^{i\alpha},e^{-i\alpha})$ and $\rho(\lambda) = \diag(e^{i\beta},e^{-i\beta})$ up to conjugacy.  We take the central character
\[ \chi: \pi_1(S^3 \ssm K) \xrightarrow{\ab} \Z \to \{\pm I\} \subset SU(2) \]
in which the latter map sends $k \in \Z$ to $(-I)^k$, and note that $\chi(\mu) = -I$ and $\chi(\lambda) = I$.  The representation $\tilde\rho = \chi\rho$ satisfies $\tilde\rho(\mu) = \diag(e^{i(\alpha+\pi)}, e^{-i(\alpha+\pi)})$ and $\tilde\rho(\lambda) = \diag(e^{i\beta},e^{-i\beta})$, and since $\iota(\alpha+\pi, \beta) = (\pi-\alpha, 2\pi-\beta)$ lies in $[0,\pi] \times (\R/2\pi\Z)$, it follows that the latter are the coordinates of $\tilde\rho$ in the pillowcase.
\end{proof}

\begin{proposition} \label{prop:mirror-image}
Let $\mirror{K}$ denote the mirror of $K$.  Then the pillowcase image of $\ch(\mirror{K})$ is the reflection of the pillowcase image of $\ch(K)$ across the line $\beta = \pi$, obtained by the transformation $(\alpha,\beta) \mapsto (\alpha,2\pi-\beta)$.
\end{proposition}

\begin{proof}
We observe that $\mirror{K}$ has the same knot group as $K$, but with peripheral data
\[ \mu_{\mirror{K}} = \mu_K^{-1}, \qquad \lambda_{\mirror{K}} = \lambda_K. \]
Any representation $\rho \in R(K)$ with image $(\alpha,\beta)$ thus gives rise to an identical $\mirror{\rho} \in R(\mirror{K})$ with coordinates $\iota(2\pi-\alpha,\beta) = (\alpha,2\pi-\beta)$.
\end{proof}

Finally, we have the following general facts about the components of the image of $\ch(K)$ in the pillowcase.

\begin{proposition} \label{prop:path-components}
The image of $\ch(K)$ inside the cut-open pillowcase $C = [0,\pi] \times (\R/2\pi\Z)$ has finitely many path components.  If $K$ is not the unknot, then at least one of these components is homologically nontrivial in $H_1(C;\Z) \cong \Z$.
\end{proposition}

\begin{proof}
The representation variety $R(K)$ is a semi-algebraic subset of some $\R^n$, so by \cite[Theorem~2.4.5]{bcr} it has finitely many connected components, each of them semi-algebraic and semi-algebraically connected and therefore path-connected by \cite[Proposition~2.5.13]{bcr}.  Its quotient $\ch(K)$ then has finitely many path components, and so does its image in the pillowcase.  (In fact, the image is a finite embedded graph, as argued in \cite{zentner}.)  Cutting the pillowcase open along $\alpha=0$ and $\alpha=\pi$ does not change the number of components, since the image only meets these lines in the points $(0,0)$ and $(\pi,0)$.

The existence of a homologically nontrivial component when $K$ is a nontrivial knot is a theorem of the second author \cite[Theorem~7.1]{zentner}, which relies in turn on the fact that the result $S^3_0(K)$ of zero-surgery on $K$ separates a symplectic 4-manifold with nonvanishing Donaldson invariants \cite{km-p,km-excision}.
\end{proof}

\begin{corollary} \label{cor:connecting-path}
If $K$ is nontrivial, then the image of $\ch(K)$ inside the fundamental domain $[0,\pi] \times [0,2\pi]$ of the cut-open pillowcase contains a path from the line $\beta=0$ to the line $\beta=2\pi$.
\end{corollary}

\begin{proof}
Let $Z = i^*(X(K)) \subset X(T^2)$ denote the image of $X(K)$, and let $z_0 \in Z$ be the image of some reducible representation.  If $j: Z \hookrightarrow [0,\pi]\times (\R/2\pi\Z)$ denotes inclusion into the cut-open pillowcase and $\pi: [0,\pi] \times (\R/2\pi\Z) \to \R/2\pi\Z$ is projection, then the composition
\[ \pi_1(Z, z_0) \xrightarrow{j_*} \pi_1([0,\pi] \times (\R/2\pi\Z), z_0) \xrightarrow{\pi_*} \pi_1(\R/2\pi\Z, 0) \]
is nonzero, because the abelianization $H_1(Z) \to H_1([0,\pi] \times (\R/2\pi\Z))$ of the first map is nonzero and the second map is an isomorphism.  But $\pi\circ j$ sends a point to its $\beta$-coordinate, so if $\gamma \in \pi_1(Z,z_0)$ is a path with image $2\pi n > 0$ in $\pi_1(\R/2\pi\Z,0) \cong 2\pi\Z$ then we can lift its $\beta$-coordinate to $\R$ to get a path from $0$ to $2\pi n$, and a segment of this path from $\beta = 0$ to $\beta=2\pi$ corresponds to the desired path in the fundamental domain.
\end{proof}

\section{The pillowcase and $SU(2)$-cyclic surgeries} \label{sec:pillowcase}

In this section we use the image of $\ch(K)$ inside the pillowcase to understand which surgeries on $K$ do not admit irreducible $SU(2)$ representations.  The fundamental group of $\frac{m}{n}$-surgery on $K$ is the quotient $\pi_1(S^3 \ssm K) / \langle\mu^m\lambda^n\rangle$, so irreducible representations of the surgered manifold correspond bijectively to irreducible representations $\rho$ of the knot group satisfying $\rho(\mu^m\lambda^n) = I$.

In this section we will let $\pi: \R^2 \to \ch(T^2)$ denote the map
\[ \R^2 \to (\R/2\pi\Z)^2 \to \ch(T^2), \]
defined by as the composition of reduction mod $2\pi$ with the quotient by the involution $\iota$ of equation~\eqref{eq:iota}.  It is clear that paths in $\ch(T^2)$ lift via $\pi$ to paths in $\R^2$.

\begin{proposition} \label{prop:slope-inequality}
Let $\gamma: [0,1] \to \ch(T^2)$ be a path consisting of images of irreducible characters of $\pi_1(S^3 \ssm K)$ in the pillowcase.  Choose a lift $\tilde\gamma: [0,1] \to \R^2$ of $\gamma$ to the plane, so that $\pi\circ\tilde\gamma = \gamma$.  Suppose that $\tilde\gamma(0)$ and $\tilde\gamma(1)$ are distinct points in $\R^2$ with coordinates $(\alpha_0, \beta_0)$ and $(\alpha_1,\beta_1)$, and that some $\frac{m}{n}$-surgery on $K$ is $SU(2)$-cyclic with $n \geq 1$.
\begin{itemize}
\item If $\alpha_0 \neq \alpha_1$ and if $r = \frac{\beta_1 - \beta_0}{\alpha_1 - \alpha_0}$ is the slope of the line segment from $\gamma(0)$ to $\gamma(1)$, then \[ \left|\frac{m}{n} - (-r)\right| < \frac{c_\gamma}{n} \]
where $c_\gamma = 2\pi / |\alpha_1-\alpha_0|$.
\item If $\alpha_0 = \alpha_1$, then $n < \frac{2\pi}{|\beta_1-\beta_0|}$.
\end{itemize}
\end{proposition}

\begin{proof}
Suppose that each $\tilde\gamma(t)$ has coordinates $(\alpha_t,\beta_t) \in \R^2$.  If $\frac{m}{n}$-surgery on $K$ is $SU(2)$-cyclic, then any representation $\rho: \pi_1(S^3 \ssm K) \to SU(2)$ such that $\rho(\mu^m\lambda^n) = I$ must be reducible.  In particular, if $\rho$ is irreducible with image $\tilde\gamma(t)$, then $\rho(\mu^m\lambda^n) \neq I$.  But up to conjugacy we have $\rho(\mu) = \diag(e^{i\alpha_t},e^{-i\alpha_t})$ and $\rho(\lambda) = \diag(e^{i\beta_t},e^{-i\beta_t})$, so since $S^3_{m/n}(K)$ is $SU(2)$-cyclic we must have
\[ m\alpha_t + n\beta_t \not\equiv 0 \pmod{2\pi} \]
for all $t$, $0 \leq t \leq 1$.

At $t=0$, the fact that $S^3_{m/n}(K)$ is $SU(2)$-cyclic tells us that
\[ 2\pi k < m\alpha_0 + n\beta_0 < 2\pi(k+1) \]
for some integer $k$.  Now $m\alpha_t+n\beta_t$ varies continuously with $t$, and it can never equal $2\pi k$ or $2\pi(k+1)$ since these are $0 \pmod{2\pi}$, so at time $t=1$ we must have
\[ 2\pi k < m\alpha_1 + n\beta_1 < 2\pi(k+1) \]
as well.  Combining these inequalities, we see that
\[ \left| (m\alpha_1 + n\beta_1) - (m\alpha_0 + n\beta_0) \right| < 2\pi. \]
If $\alpha_0 \neq \alpha_1$, then dividing both sides by $n|\alpha_1 - \alpha_0|$ yields
\[ \left| \frac{m}{n} + \left(\frac{\beta_1 - \beta_0}{\alpha_1 - \alpha_0}\right) \right| < \frac{2\pi}{n|\alpha_1 - \alpha_0|}. \]
Otherwise $\alpha_0 = \alpha_1$ but $\beta_0 \neq \beta_1$ since $\gamma(0) \neq \gamma(1)$, and we have $n < 2\pi/ |\beta_1-\beta_0|$ as claimed.
\end{proof}

\begin{theorem} \label{thm:open-pillowcase-image}
Suppose that $K$ admits infinitely many $SU(2)$-cyclic surgeries.  Define
\[ I(K) \subset \R^2 \]
as the preimage under $\pi: \R^2 \to \ch(T^2)$ of the pillowcase image of $\ch^*(K)$.  Then every path component of $I(K)$ is either a point or a line segment, and the line segments all have the same slope.  If in addition $K$ is not the unknot, then at least one path component is a line segment, and its slope $r \in \R \cup \{\infty\}$ satisfies $|r| > 2$.  If $r \neq \infty$, then the $SU(2)$-cyclic slopes $\frac{m}{n}$ all satisfy
\[ \left| \frac{m}{n} - (-r) \right| \leq \frac{|r|}{n}, \]
so only finitely many values of $m$ are possible for any given $n$, and $\frac{m}{n} \to -r$ as $n\to \infty$.
\end{theorem}

\begin{proof}
Let $I_0 \subset I(K)$ be a path component.  If $I_0$ is not contained in some straight line, then it contains three points $p_i = (\alpha_i, \beta_i)$, $i=0,1,2$, which are not collinear.  The $\alpha_i$ cannot all be the same, so at most two of them are equal; we label the $p_i$ so that $\alpha_0$ is different from both $\alpha_1$ and $\alpha_2$.  Letting $c_i = \frac{2\pi}{|\alpha_i-\alpha_0|}$ and $r_i = \frac{\beta_i - \beta_0}{\alpha_i - \alpha_0}$ for $i=1,2$, we know that $r_1 \neq r_2$ or else the $p_i$ would be collinear, and that $r_1,r_2 \neq \infty$.

If $\frac{m}{n}$ is an $SU(2)$-cyclic surgery slope for $K$, say with $n>0$, then we can apply Proposition~\ref{prop:slope-inequality} to paths from $p_0$ to $p_1$ and from $p_0$ to $p_2$ within $I_0$ to see that
\[ \left| \frac{m}{n} - (-r_1) \right| < \frac{c_1}{n}, \qquad \left| \frac{m}{n} - (-r_2) \right| < \frac{c_2}{n}. \]
Given a sequence of such slopes $\frac{m_k}{n_k}$ with $n_k \to \infty$, it would follow that $\frac{m_k}{n_k}$ converges to both $-r_1$ and $-r_2$, which contradicts $r_1 \neq r_2$.  The $SU(2)$-cyclic slopes $\frac{m}{n}$ therefore satisfy some uniform upper bound $n \leq C$.  But from either of these inequalities it follows that for fixed $n$ there can only be finitely many $m$ such that $\frac{m}{n}$ is an $SU(2)$-cyclic surgery slope, so the total number of such slopes is finite, which is a contradiction.  We conclude that $I_0$ must be contained in a straight line.

In fact, the above argument shows that the slope $r$ of the straight line is uniquely determined by the set $\{\frac{m_k}{n_k}\}$ of $SU(2)$-cyclic slopes.  Indeed, Proposition~\ref{prop:slope-inequality} shows that $r = \infty$ if and only if the $n_k$ are bounded, and if they are unbounded then $n_k\to\infty$ and $-\displaystyle r = \lim_{n_k\to\infty} \frac{m_k}{n_k}$.  Thus all nontrivial path components of $I(K)$ have the same slope.

If $K$ is not the unknot, then Corollary~\ref{cor:connecting-path} says that there is a path in
\[ I(K) \cup \big(\R \times 2\pi\Z\big) \subset \R^2, \]
which is the lift to $\R^2$ of the image of all of $\ch(K)$ (including the reducible characters), between two points $(\alpha_0,0)$ and $(\alpha_1,2\pi)$ with $0 < \alpha_0,\alpha_1 < \pi$.  The path must be a straight line in the region $0 < \beta < 2\pi$, where every point is a lift of the image of an irreducible character, and for any small $\delta > 0$ it connects points $(\alpha'_0,\delta)$ and $(\alpha'_1, 2\pi-\delta)$ which project to the image of two irreducibles, with slope $r = \frac{2\pi-2\delta}{\alpha'_1-\alpha'_0}$.

By Proposition~\ref{prop:pillowcase-facts}, there is some $\epsilon>0$ depending only on $K$ such that every irreducible character has $\alpha$-coordinate in the open interval $(\epsilon,\pi-\epsilon)$ in the pillowcase.  The line segment in $\R^2$ between $(\alpha'_0,\delta)$ and $(\alpha'_1,2\pi-\delta)$ consists entirely of lifts of images of characters which have $\beta\not\in 2\pi\Z$ and are thus irreducible, so it follows that $k\pi + \epsilon < \alpha'_0,\alpha'_1 < (k+1)\pi-\epsilon$ for some integer $k$.  This gives us $|\alpha'_1 - \alpha'_0| < \pi-2\epsilon$, so the slope $r$ of this segment satisfies $|r| > \frac{2\pi-2\delta}{\pi-2\epsilon}$.  Taking limits as $\delta\to 0$, we see that $|r| \geq \frac{2\pi}{\pi-2\epsilon} > 2$.  We remark that if $r \neq \infty$ then we also have $\left\lvert\frac{m}{n} - (-r)\right\rvert < \frac{2\pi/|\alpha'_1-\alpha'_0|}{n}$ for each such $(\alpha'_0,\delta)$ and $(\alpha'_1,2\pi-\delta)$, and as $\delta \to 0$ the ratio $\frac{2\pi}{\alpha'_1 - \alpha'_0}$ approaches the slope $-r$, so that $\left\lvert\frac{m}{n}-(-r)\right\rvert \leq \frac{|r|}{n}$.
\end{proof}

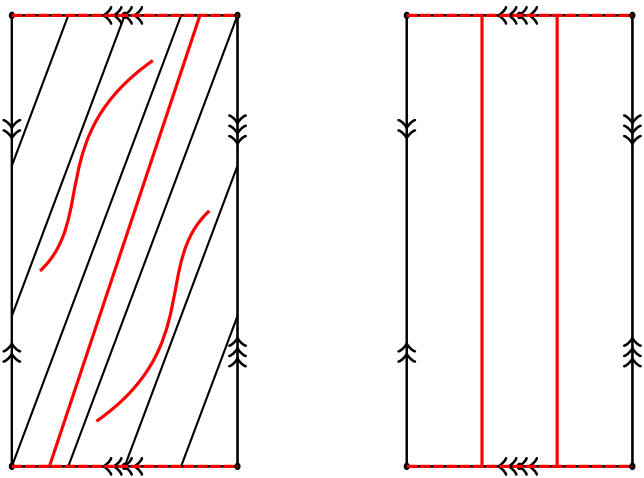
\begin{figure}
\begin{tikzpicture}[style=thick]
\begin{scope}[xscale=1.5,yscale=2]
  \draw plot[mark=*,mark size = 0.5pt] coordinates {(0,0)(1,0)(2,0)(2,3)(1,3)(0,3)} -- cycle; 
  \begin{scope}[decoration={markings,mark=at position 0.55 with {\arrow[scale=1.5]{>>}}}]
    \draw[postaction={decorate}] (0,0) -- (0,1.5);
    \draw[postaction={decorate}] (0,3) -- (0,1.5);
  \end{scope}
  \begin{scope}[decoration={markings,mark=at position 0.575 with {\arrow[scale=1.5]{>>>}}}]
    \draw[postaction={decorate}] (2,0) -- (2,1.5);
    \draw[postaction={decorate}] (2,3) -- (2,1.5);
  \end{scope}
  \begin{scope}[decoration={markings,mark=at position 0.6 with {\arrow[scale=1.5]{>>>>}}}]
    \draw[postaction={decorate}] (2,3) -- (0,3);
    \draw[postaction={decorate}] (2,0) -- (0,0);
  \end{scope}
  \draw (0,2) -- (0.5,3) (0,1) -- (1,3) (0,0) -- (1.5,3) (0.5,0) -- (2,3) (1,0) -- (2,2) (1.5,0) -- (2,1);
  \begin{scope}[color=red,style=very thick]
    \draw (0.333,0) -- (1.667,3);
    \draw (0.25,1.3) .. controls (0.75,1.65) and (0.3,2.2) .. (1.25,2.7);
    \draw (0.75,0.3) .. controls (1.7,0.8) and (1.25,1.35) .. (1.75,1.7);
    \draw[dash pattern=on 4pt off 1pt] (0,0) -- (2,0) (0,3) -- (2,3);
  \end{scope}
\end{scope}
\begin{scope}[xscale=1.5,yscale=2,shift={(3.5,0)}]
  \draw plot[mark=*,mark size = 0.5pt] coordinates {(0,0)(1,0)(2,0)(2,3)(1,3)(0,3)} -- cycle; 
  \begin{scope}[decoration={markings,mark=at position 0.55 with {\arrow[scale=1.5]{>>}}}]
    \draw[postaction={decorate}] (0,0) -- (0,1.5);
    \draw[postaction={decorate}] (0,3) -- (0,1.5);
  \end{scope}
  \begin{scope}[decoration={markings,mark=at position 0.575 with {\arrow[scale=1.5]{>>>}}}]
    \draw[postaction={decorate}] (2,0) -- (2,1.5);
    \draw[postaction={decorate}] (2,3) -- (2,1.5);
  \end{scope}
  \begin{scope}[decoration={markings,mark=at position 0.6 with {\arrow[scale=1.5]{>>>>}}}]
    \draw[postaction={decorate}] (2,3) -- (0,3);
    \draw[postaction={decorate}] (2,0) -- (0,0);
  \end{scope}
  \begin{scope}[color=red,style=very thick]
    \draw (0.667,0) -- (0.667,3) (1.333,0) -- (1.333,3);
    \draw[dash pattern=on 4pt off 1pt] (0,0) -- (2,0) (0,3) -- (2,3);
  \end{scope}
\end{scope}
\end{tikzpicture}
\caption{If the image $i^*(\ch(K))$ contains a component which is curved, as at left, then there are at most finitely many $SU(2)$-cyclic surgeries by Theorem~\ref{thm:open-pillowcase-image}.  If there are only vertical segments, as at right, then there could be an unbounded set of $SU(2)$-cyclic surgeries, but we will see in Section~\ref{sec:finiteness} that this cannot occur.}
\label{pillowcase curved and infinity}
\end{figure}

\begin{definition} \label{def:limit-slope}
Suppose that $K \subset S^3$ is a nontrivial knot with infinitely many $SU(2)$-cyclic surgeries.  Then we will say that $K$ is \emph{$SU(2)$-averse}, and we define the \emph{limit slope} $r(K) \in \R \cup\{\infty\}$ to be $-r$, where $r$ is the common slope of the line segments in $I(K)$ whose existence is guaranteed by Theorem~\ref{thm:open-pillowcase-image}.
\end{definition}

Rephrasing Theorem~\ref{thm:open-pillowcase-image} in terms of Definition~\ref{def:limit-slope}, we see that an $SU(2)$-averse knot with limit slope $r(K)$ satisfies $|r(K)| > 2$, and if $r(K) \neq \infty$ then every $SU(2)$-cyclic slope $\frac{m}{n}$ for $K$ satisfies
\[ \left|\frac{m}{n} - r(K)\right| \leq \frac{|r(K)|}{n}. \]
In particular, there are only finitely many possible $m$ for a given $n$, so the denominators $n$ are unbounded and we have $\frac{m}{n} \to r(K)$ as $n\to\infty$.  (We will prove in Section~\ref{sec:finiteness} that $r(K) \neq \infty$ always holds.)

\begin{example}
The torus knots $T_{p,q}$ are $SU(2)$-averse with limit slope $r=pq$, since Moser \cite{moser} showed that $(pq + \frac{1}{n})$-surgery on $T_{p,q}$ is a lens space, and thus $SU(2)$-cyclic, for all integers $n\neq 0$.
\end{example}

\begin{theorem} \label{thm:finitely-many-lines}
If $K$ is a nontrivial $SU(2)$-averse knot with limit slope $r = r(K)$, and if $r<\infty$ (resp.\ $r=\infty$), then there are finitely many constants $c_1,\dots,c_n$ such that every irreducible $\rho: \pi_1(S^3\ssm K) \to SU(2)$ has image on one of the lines $\beta = -r\alpha + c_i$ (resp.\ $\alpha = c_i$) in the cut-open pillowcase.
\end{theorem}

\begin{proof}
We lift the pillowcase image of $\ch^*(K)$ to $I(K) \subset \R^2$ as in Theorem~\ref{thm:open-pillowcase-image}, which guarantees that $I(K)$ is a union of isolated points and line segments of slope $-r$.  Then its image $I'$ in the cut-open pillowcase $C \cong [0,\pi] \times (\R/2\pi\Z)$, obtained from $I(K) \cap \big([0,\pi]\times \R\big)$ by reducing the second coordinate mod $2\pi$, is also a union of such points and line segments.

The set $I' \subset C$ can be identified with the pillowcase image of $\ch^*(K)$, since the latter avoids a neighborhood of the lines $\alpha=0$ and $\alpha=\pi$ in $\ch(T^2)$.  Since $R^*(K)$ is a semi-algebraic set, it has only finitely many path components, exactly as argued in the proof of Proposition~\ref{prop:path-components}, and hence so do its quotient $\ch^*(K)$ and the image $I' \subset C$ of this quotient.  Since each of these components of $I'$ is contained in a line of the form $\beta = -r\alpha+c_i$ (resp.\ $\alpha=c_i$), we conclude that finitely many $c_i$ suffice.
\end{proof}

\section{Finiteness of limit slopes} \label{sec:finiteness}

In this section we will show that $SU(2)$-averse knots have limit slope $r(K) \neq \infty$.  The result is a straightforward consequence of the following theorem about the $SU(2)$ character variety of a nontrivial knot.

\begin{theorem} \label{thm:pillowcase-near-pi/2}
Let $K$ be a nontrivial knot in $S^3$.  Then at least one of the following holds:
\begin{enumerate}
\item \label{i:whole-curve-pi/2} The pillowcase image of $\ch^*(K)$ contains the entire curve $\alpha=\frac{\pi}{2}$, or equivalently for every $\beta \in \R/2\pi\Z$ there is an irreducible $\rho: \pi_1(S^3 \ssm K) \to SU(2)$ such that
\[ \rho(\mu) = \twomatrix{i}{0}{0}{-i} \qquad\mathrm{and}\qquad \rho(\lambda) = \twomatrix{e^{i\beta}}{0}{0}{e^{-i\beta}}; \]
\item \label{i:arc-near-pi/2} There is a constant $\epsilon > 0$ and an arc of irreducible representations
\[ \rho_s: \pi_1(S^3 \ssm K) \to SU(2), \qquad s \in \left[\frac{\pi}{2},\frac{\pi}{2}+\epsilon\right) \]
with $\alpha$-coordinate $s$, i.e.\ such that $\rho_s(\mu) = \diag(e^{is},e^{-is})$.
\end{enumerate}
If case~\eqref{i:whole-curve-pi/2} applies, then $K$ has no nontrivial $SU(2)$-cyclic surgeries.
\end{theorem}

We do not know of any examples where case~\eqref{i:arc-near-pi/2} does not occur, and we expect (but are unable to prove) that it always happens.  When case~\eqref{i:whole-curve-pi/2} holds, however, the last claim of Theorem~\ref{thm:pillowcase-near-pi/2} is immediate: if $\frac{p}{q} \neq \infty$ is an $SU(2)$-cyclic slope, then the line $p\alpha+q\beta \equiv 0 \pmod{2\pi}$ in the pillowcase intersects the line $\alpha=\frac{\pi}{2}$ at the point
\[ \left(\frac{\pi}{2}, -\frac{p\pi}{2q}\pmod{2\pi}\right), \]
so there is an irreducible representation $\pi_1(S^3\ssm K) \to SU(2)$ with image at this point.  Before proving Theorem~\ref{thm:pillowcase-near-pi/2}, we explain how the finiteness of $r(K)$ follows.

\begin{theorem} \label{thm:limit-slope-finite}
If $K \subset S^3$ is a nontrivial $SU(2)$-averse knot, then the limit slope $r(K)$ is finite.
\end{theorem}

\begin{proof}
Since $K$ has $SU(2)$-cyclic surgeries, it must satisfy case~\eqref{i:arc-near-pi/2} of Theorem~\ref{thm:pillowcase-near-pi/2}.  Thus we have an arc of irreducible representations
\[ \rho_s: \pi_1(S^3 \ssm K) \to SU(2), \qquad s \in \left[\frac{\pi}{2},\frac{\pi}{2}+\epsilon\right) \]
where each $\rho_s$ has $\alpha$-coordinate $s$ in the pillowcase.  If $r(K)=\infty$, then by Theorem~\ref{thm:finitely-many-lines} we can only achieve finitely many $\alpha$-coordinates along this arc, and this is a contradiction.
\end{proof}

The proof of Theorem~\ref{thm:pillowcase-near-pi/2} follows the same lines as the following theorem of Kronheimer and Mrowka, so we briefly recall its proof here.
\begin{theorem}[{\cite[Corollary~7.17]{km-excision}}] \label{thm:km-traceless}
Every nontrivial knot $K \subset S^3$ admits a nonabelian representation
\[ \rho: \pi_1(S^3 \ssm K) \to SU(2) \]
such that $\rho(\mu)$ is traceless.
\end{theorem}
\begin{proof}[Proof (sketch)]
Kronheimer and Mrowka construct for each knot $K \subset S^3$ a closed 3-manifold
\begin{equation} \label{eq:y1-khi}
Y_1(K) = \big(S^3 \ssm N(K)\big) \cup \big(F\times S^1\big),
\end{equation}
where $F$ is a genus-1 surface with one boundary component, and the gluing identifies
\[ \mu \sim \{p\} \times S^1, \qquad \lambda \sim \partial F \times \{q\}. \]
Letting $\alpha,\beta \subset F$ be a pair of curves which intersect transversely in a point, they define a Hermitian line bundle $w \to Y_1(K)$ such that $c_1(w)$ is Poincar\'e dual to $\beta \times \{p\} \subset F \times S^1$, and a $U(2)$-bundle $E \to Y_1(K)$ equipped with an isomorphism $\bigwedge^2E \to w$. We denote by 
\begin{equation} \label{rep variety unperturbed}
	X^{w}(Y_1(K)) = \{ \rho \in \Hom(\pi_1(Y_1(K)),SO(3)) \mid w_2(\rho) \equiv c_1(w) \pmod{2} \}/SO(3)
\end{equation}
the variety of characters of representations $\pi_1(Y_1(K)) \to SO(3)$ whose second Stiefel-Whitney class agrees with $c_1(w)$ modulo $2$. In \cite[Lemma 7.15]{km-excision}, they show that $\cR^{w}(Y_1(K))$ can be identified with a double cover of the variety
\begin{equation} \label{eq:RKi}
R(K,i) := \left\{ \rho:\pi_1(S^3\ssm K) \to SU(2) \,\middle|\, \rho(\mu) = \twomatrix{i}{0}{0}{-i} \right\}.
\end{equation}
The centralizer of $\diag(i,-i)$ is a $U(1)$ acting on $R(K,i)$ by conjugation; the orbits are a point for the unique reducible $\rho \in R(K,i)$ and $S^1$ for all irreducible $\rho$.

Kronheimer and Mrowka define $\KHI(K)$ to be a group whose dimension is half that of the instanton homology $I_*(Y_1(K))_w$, which is defined as in \cite{donaldson-book}.  If $R(K,i)$ contains only the reducible representation, then the Chern-Simons functional used to construct $I_*(Y_1(K))_w$ has only two critical points, both nondegenerate, and it follows that $\dim(\KHI(K)) \leq 1$.  But they prove that $\dim(\KHI(K)) \geq 2$ whenever $K$ is knotted, and thus $R(K,i)$ must contain at least one irreducible.
\end{proof}

We will use $\KHI(K)$ in a similar way to prove Theorem~\ref{thm:pillowcase-near-pi/2}.  Assuming that the theorem is false for some knot $K$, we will describe the pillowcase image of $\ch(K)$ near the line $\alpha=\frac{\pi}{2}$ in Lemma~\ref{lem:space-near-pi/2}.  Using a generalization of the results of \cite{zentner}, we then perturb the Chern-Simons functional defining $I_*(Y_1(K))_w$ so that its critical point set corresponds not to the curve $\alpha=\frac{\pi}{2}$ (i.e., where $\rho(\mu)$ is conjugate to $\left(\begin{smallmatrix}i&0\\0&-i\end{smallmatrix}\right)$), which was used to prove Theorem~\ref{thm:km-traceless}, but to some nearby curve in the pillowcase which avoids the image of $\ch^*(K)$ completely.  This would imply that $\KHI(K)$ has rank 1 by the same argument (which we generalize in Theorem~\ref{thm:perturbed-khi-bound}), and so $K$ must be the unknot.  We will now begin the proof of Theorem~\ref{thm:pillowcase-near-pi/2} in earnest.

We denote by $\mathscr{A}$ the space of $SO(3)$ connections on the adjoint bundle $\ad(E)$, or equivalently, the space of $U(2)$-connections on $E$ which induce some fixed connection in the determinant line bundle $w$ which we suppress from our notation.  The character variety $X^w(Y_1(K))$ can be identified with the space of flat $SO(3)$ connections on $\ad(E)$ modulo determinant-1 gauge transformations,
\[
	X^w(Y_1(K)) \cong \{ A \in \mathscr{A} \mid F_A = 0 \}/{\Aut(E)} \, . 
\]
Equivalently, it can be identified with projectively flat $U(2)$-connections in $E$; this is independent of the choice of fixed connection in the determinant line bundle.

The proof of Theorem~\ref{thm:pillowcase-near-pi/2} will make use of a variation of Kronheimer and Mrowka's argument sketched above by applying holonomy perturbations. In order to use these, we set up the notation and recall the main technical results about holonomy perturbations of \cite{zentner}, and we refer to the latter reference for more details on these perturbations. We define a manifold diffeomorphic to the original $Y_1(K)$ from equation \eqref{eq:y1-khi} above by
\[
	Y_1(K) = \big(S^3 \ssm N(K)\big) \cup \big([0,1] \times T^2\big) \cup \big(F\times S^1\big) \, ,
\]
where we have introduced a thickened cylinder $M = [0,1] \times T^2$ in between the knot complement and $(F\times S^1)$.  We declare that the boundary component $\{0\} \times T^2$ is the one glued to $S^3 \ssm N(K)$, mapping the circles $\{0\} \times S^1 \times \{\pt\}$ to meridians of $K$ and the circles $\{0\} \times \{\pt\} \times S^1$ to longitudes.  The boundary component $\{1\} \times T^2$ is glued to $F \times S^1$ with the circle $\{1\} \times S^1 \times \{\pt\}$ mapping to the $S^1$-factors, and with the circles $\{1\} \times \{\pt\} \times S^1$ mapping to the boundary of $F$. 

Let $\Sigma = [0,1] \times S^1$ be the 2-dimensional annulus. For $k = 0, \dots, n-1$, let 
\[
	A_k = \begin{pmatrix} a_k & c_k \\ b_k & d_k \end{pmatrix} \in \text{\em SL}(2,\Z)
\]
be a sequence of matrices. We define a sequence of embeddings 
$
	\iota_k \colon \Sigma \times S^1 \to M
$
by the formula 
\begin{equation} \label{holonomy perturbation directions}
\left(t,\begin{pmatrix} z \\ w \end{pmatrix}\right) \mapsto \left(\frac{k + t}{n} , \begin{pmatrix} a_k & c_k \\ b_k & d_k \end{pmatrix} \begin{pmatrix} z \\ w \end{pmatrix} \right) \, , 
\end{equation}
where we understand $S^1 = \R / \Z$ and write $\Sigma \times S^1$ as $[0,1] \times T^2$, so that the matrix multiplication is understood in the usual sense. Notice that $\img(\iota_k) = [\frac{k}{n},\frac{k+1}{n}] \times T^2 \subseteq M$. 

Let $\mu$ be a 2-form on $\Sigma$ with support in the interior of $\Sigma$ and integral $1$. The 2-form $\pi_\Sigma^*\mu$ on $\Sigma \times S^1$, defined by pulling back $\mu$ by the projection onto the first factor, can be pushed forward via $\iota_k$ to obtain a 2-form $\mu_k$ on $M$.  Notice that $\mu_k$ has compact support in the interior of $M$, so that $\mu_k$ extends to $Y_1(K)$.

For $k = 0, \dots, n-1$, let $\chi_k\colon SU(2) \to \R$ be a sequence of class functions. These give rise to $2\pi$-periodic even functions $g_k: \R \to \R$ via the formula
\begin{equation}\label{eq:even functions}
\chi_k \left(\begin{pmatrix} e^{it} & 0 \\ 0 & e^{-it} \end{pmatrix} \right) = g_k(t) \, .
\end{equation}
To each class function $\chi \colon SU(2) \to \R$ we associate a Lie-algebra valued, $SU(2)$-equivariant function $\chi': SU(2) \to \mathfrak{su}(2)$ which is dual, with respect to the trace, to the differential $d \chi \colon TSU(2) \to \R$.

In the sequel, we suppose we have chosen a trivialization of our bundle $E \to Y_1(K)$ over $M$. In fact, by our choice of the line bundle $w$, there is no obstruction to such a trivialization over $M$. For any point $z \in \Sigma$ there is a loop $\iota_k( \{ z \} \times S^1)$ passing through it, and there is precisely one such loop through any point in the image of $\iota_k$. For such a point $x$ in the image, we denote by $\Hol_{\iota_k}(A)_x$ the holonomy of $A$ along the loop through this point. This way $\Hol_{\iota_k}(A)$ can be seen as a section of the automorphism bundle $\Aut(E)$ over the image of $\iota_k$. For a class function $\chi_k$ we can therefore define 
\[
	\chi_k'(\Hol_{\iota_k}(A)) \,  \mu_k 
\]
which is a 2-form with values in the bundle $\ad(E)$ over $Y_1(K)$, and with compact support in the interior of the image of $\iota_k$.

We now define a family, parametrized by $t \in [0,1]$, of perturbations to the flatness equation on $Y_1(K)$.

\begin{definition}\label{one-parameter perturbed flat thickened torus}
Let the function $\theta^t\colon  \mathscr{A} \to \Omega^2(Y_1(K);\ad(E))$ be defined in the following way. For each $k=0,\dots,n-1$ and $\frac{k}{n} \leq t \leq \frac{k+1}{n}$, we let
\begin{equation*}
	\theta^t_{\pertdata}(A) := n\left(t-\frac{k}{n}\right) \cdot \chi'_k(\operatorname{Hol}_{\iota_k}(A)) \, \mu_k + \sum_{l=0}^{k-1} \chi'_l(\operatorname{Hol}_{\iota_l}(A)) \,  \mu_l.
\end{equation*}
This function interpolates between $\theta^0 = 0$ and $\theta^1 (A)  = \sum_{l=0}^{n-1} \chi'_l(\operatorname{Hol}_{\iota_l}(A)) \,  \mu_l$. 

The perturbed flatness equation for a connection $A \in \mathscr{A}$ now is 
\begin{equation}\label{perturbed_flatness_equation}
 F_A = \theta^t_{\pertdata}(A) \, ,
\end{equation}
and we denote by $X^w_{\pertdata(t)}(Y_1(K))$ the space of connections $A \in \mathscr{A}$ which solve this equation, modulo the group of determinant-1 automorphisms of $E$. 
\end{definition}

The parameter $t$ is not to be confused with the coordinate $t$ on the thickened torus $M=[0,1]\times T^2$, although there is the accidental coincidence that $\theta^{t}$ has support inside $[0,\frac{1}{n}\lceil nt\rceil] \times T^2 \subseteq M$.   At each time $t$, the set $X^w_{\pertdata(t)}(Y_1(K))$ arises as the critical set modulo gauge transformations of some perturbation of the Chern-Simons functional, which for $t=1$ we write as
\begin{equation} \label{eq:perturbed-cs}
\mathit{CS} + \Phi: \mathscr{A} \to \R
\end{equation}
where $\Phi(A) = \sum_{k=0}^{n-1} \int_{\Sigma_k} \chi_k(\Hol_{\iota_k}(A))\mu_k(z)$.

Note that equation \eqref{perturbed_flatness_equation} also makes sense for $SU(2)$-connections defined over the thickened cylinder $M$. We denote by $X_{\pertdata(t)}(M)$ the connections which solve this equation modulo bundle automorphisms. The two inclusions of the boundary components $\{0\} \times T^2$ and $\{1\} \times T^2$ into $M$ induce two restriction maps $r_0$ and $r_1$ to the pillowcase $\ch(T^2)$, to which we give coordinates $\alpha, \beta$ just as before. 

The main result (Theorem 4.2) of \cite{zentner} states that given any isotopy through area-preserving maps $\psi_t$ of the pillowcase $\ch(T^2)$, fixing the four singular points, and given some $\epsilon > 0$, there is a finite sequence of embeddings $\{\iota_k\}$ and there are class functions $\{\chi_k\}$ as above, such that the isotopy $\psi_t$ is $\epsilon$-close in the $C^0$-topology to an isotopy $\overline{\phi}_t$ which also fixes the four singular points.  The isotopy $\overline{\phi}_t$ is described in the following way.

The lift $\phi_t$ of $\overline\phi_t$ to the branched double cover $\R^2/2\pi \Z^2$ is an isotopy which for $\frac{k}{n} \leq t \leq \frac{k+1}{n}$ is given by
\begin{equation*}
\begin{split}
	\phi_t =  \zeta_{n(t-\frac{k}{n})}^{(k)} \circ \zeta_1^{(k-1)} \circ \dots  \circ \zeta_1^{(0)} 
\end{split}
\end{equation*}
for any $k= 0, \dots, n-1$. 
Here the maps $\zeta_{s}^{(k)}$ are defined for $s \in [0,1]$, and for $k= 0, \dots, n-1$  by the equation
\begin{equation*}
	\zeta_{s}^{(k)} = A_k \circ \chi^{s}_{f_k} \,\circ A_k^{-1},
\end{equation*}
where the $A_k$ are elements of $SL(2,\Z)$ and
\begin{equation}\label{eq:shearing map perturbation}
	\chi^{s}_{f_k}\colon  \begin{pmatrix} \alpha \\ \beta \end{pmatrix} \mapsto 
\begin{pmatrix} \alpha + s\cdot f_k(\beta) \\ \beta \end{pmatrix} 
\end{equation} 
are shearing isotopies. The functions $f_k$ are the derivatives of the functions $g_k$ defined in equation \eqref{eq:even functions}. These isotopies $\zeta_{s}^{(k)}$ are in fact shearing isotopies, as defined in equation~\eqref{eq:shearing-isotopy} in the introduction. 

Associated to the perturbation data $\pertdata$, there is the {\em perturbed character variety} $X_{\pertdata (t)}(M)$ for every $t \in [0,1]$, for which the two restriction maps $r_0$ and $r_1$ are related by the commutative diagram
\begin{equation}\label{commutative diagram}
\vcenter{\xymatrix{
& X_{\pertdata(t)}(M) \ar[dl]_{r_0} \ar[dr]^{r_1} & \\
\ch(T^2) \ar[rr]^{\overline{\phi}_t} && \ch(T^2).
}}
\end{equation}
We will need a slight enhancement of the approximation result \cite[Theorem 3.3]{zentner} which is a part of the statement of \cite[Theorem 4.2]{zentner}. In fact, the approximation holds in any $C^k$-topology for $k \geq 0$. However, in this article we will only need the fact that we can approximate area-preserving isotopies by shearing isotopies in the $C^1$-topology.

We defer the slightly technical proof of the following result to Appendix~\ref{sec:c1-approximation}. We state it in the $\Z/2$-equivariant case which is needed for our application, but we remark that it also holds in the non-equivariant version. 

\begin{theorem}\label{C1 approximation}
Let 
\[ \psi: [0,1] \times T^2 \to T^2, \qquad (t,(x,y)) \mapsto \psi_t(x,y) \]
be a $\Z/2$-equivariant smooth isotopy through area-preserving maps, which necessarily fixes the four fixed points of the hyperelliptic involution. Let $r \geq 0$ be given. Then for any $\epsilon > 0$, there is a $\Z/2$-equivariant map
\[ \phi: [0,1] \times T^2 \to T^2, \qquad (t,(x,y)) \mapsto \phi_t(x,y) \]
which is continuous in $t$ and smooth in $(x,y)$, such that:
	\begin{enumerate}[label=(\roman*)]
	\item If $d$ denotes the Euclidean metric on $T^2$, then we have
	\begin{equation*}
		\sup_{(x,y) \in T^2} d(\psi_t(x,y),\phi_t(x,y)) + \sum_{l=1}^{r} \norm{D^{(l)}\psi_t- D^{(l)} \phi_t}_\infty < \epsilon
	\end{equation*}
	for all $t \in [0,1]$ and all $(x,y) \in T^2$.
	\item There is a partition $0=t_0 < t_1 \dots < t_{n+1}=1$ and a $\Z/2$-equivariant shearing isotopy $(\zeta^{(i)}_s)_{s\in[0,1]}$ for each $i=0,\dots,n$ such that for any $t_i \leq t \leq t_{i+1}$, we have
	\begin{equation*}
	   \phi_t = \zeta^{(i)}_{t-t_i} \circ \phi_{t_i}
	\end{equation*}
for all $(x,y) \in T^2$. 
	\end{enumerate}
In other words, the smooth isotopy $(\psi_t)$ can be $C^r$-approximated
by isotopies $(\phi_t)$ which are piecewise (in the coordinate $t$) isotopies through shearings.
\end{theorem}

The following lemma is the analogue of \cite[Lemma 7.15]{km-excision} for the perturbed character variety $X^w_{\pertdata(t)}(Y_1(K))$; we will mostly be interested in the case $t=1$. We let $C$ be the circle in $\ch(T^2)$ defined by $\alpha = \pi/2$, and then we define the circle
\[ \tilde{C} \subset (\R/2\pi\Z)_\alpha \times (\R/2\pi\Z)_\beta \]
as one of two possible lifts of $C$ to the branched double cover $\tilde\ch(T^2) \cong T^2$ of $\ch(T^2)$: for concreteness, we will take $\tilde{C}$ to be the circle $\alpha = \frac{\pi}{2}$, rather than $\alpha = \frac{3\pi}{2}$.  This is analogous to a choice implicitly made in \cite[Lemma~7.15]{km-excision}: if $\rho(\mu)$ is diagonal and traceless then it could be either $\diag(i,-i)$ or $\diag(-i,i)$, and the two choices are equivalent since they are related by conjugation but Kronheimer and Mrowka choose the former.
\begin{lemma}\label{le:repvariety_Y1_perturbed}
The character variety $X^w_{\pertdata(t)}(Y_1(K))$ is a double cover of the perturbed representation variety
\[ R(K \mid (\phi_t)^{-1}(\tilde{C})) := \left\{ \rho \in R(K) \,\middle\vert\,
\begin{array}{l} \rho(\mu) =  \twomatrix{e^{i\alpha}}{0}{0}{e^{-i\alpha}}, \rho(\lambda) = \twomatrix{e^{i\beta}}{0}{0}{e^{-i\beta}} \\ \quad\mathrm{for\ some\ } (\alpha,\beta) \in (\phi_t)^{-1}(\tilde{C}). \end{array} \right\} \, .
\]

\end{lemma} 

\begin{proof}
Lemma 7.14 of \cite{km-excision} describes the $SO(3)$-character variety of $F\times S^1$ with the second Stiefel-Whitney class $w$ as specified above, where the restriction map on characters from $F\times S^1$ to its boundary $T^2$ identifies these conjugacy classes of representations with the circle which doubly covers the circle $\tilde{C}\subset \tilde\ch(T^2)$.  The representations of $T^2$ from this circle which extend to equivalence classes of perturbed representations of $X^w_{\pertdata(t)}(Y_1(K))$ on the other side of $F \times S^1$ are precisely those which correspond to points of $(\phi_t)^{-1}(\tilde{C})$ because of the diagram \eqref{commutative diagram} from above. 
\end{proof}

Having explained how the above perturbations affect the generators of the $\KHI$ chain complex, we now find a convenient choice of curve in $\ch(T^2)$ to make the generating set as small as possible.  The following lemma, which we will use to prove Theorem~\ref{thm:pillowcase-near-pi/2}, describes the pillowcase image of $\ch^*(K)$ in a neighborhood of the curve $\alpha=\frac{\pi}{2}$.

\begin{lemma} \label{lem:space-near-pi/2}
Suppose that $K\subset S^3$ is a nontrivial knot which does not satisfy Theorem~\ref{thm:pillowcase-near-pi/2}.  Then there are constants $b$ and $\epsilon$, with $0 < \epsilon < b < \pi-\epsilon$, such that if $(\alpha,\beta)$ lies in the pillowcase image of $\ch^*(K)$ and
\[ \frac{\pi}{2} - \epsilon < \alpha < \frac{\pi}{2} + \epsilon, \]
then $\alpha = \frac{\pi}{2}$ and $\beta \not\in (b-\epsilon,b+\epsilon) \cup (2\pi-b-\epsilon,2\pi-b+\epsilon)$.
\end{lemma}

\begin{proof}
Since case~\eqref{i:whole-curve-pi/2} of Theorem~\ref{thm:pillowcase-near-pi/2} does not apply, there is some point $(\frac{\pi}{2},b_0)$ which is not in the pillowcase image of $\ch^*(K)$.  The variety $R(K,i)$ of \eqref{eq:RKi} is compact, as it is the closed subset of the compact variety $R(K)$ cut out by the matrix equation $\rho(\mu) = \diag(i,-i)$.  Moreover, since $\Delta_K(-1) = \pm \det(K) \neq 0$, Proposition~\ref{prop:pillowcase-facts} says that the unique reducible representation in $R(K,i)$ is not a limit of irreducibles, and so the subset
\[ R^*(K,i) \subset R(K,i) \]
consisting of irreducibles is also compact.

The image of the compact set $R^*(K,i)$ in the pillowcase is closed, so it avoids an open neighborhood $U \subset \{\frac{\pi}{2}\}\times \R/2\pi\Z$ of $(\frac{\pi}{2},b_0)$.  Letting $b \not\in \{0,\pi\} $ be the $\beta$-coordinate of some point in $U$, we can assume that $0<b<\pi$ by Proposition~\ref{prop:pillowcase-symmetry}, and we pick $\delta > 0$ so that $0 < b-\delta < b+\delta < \pi$ and the interval
\[ \left\{\frac{\pi}{2}\right\} \times (b-\delta,b+\delta) \]
in the pillowcase is disjoint from the image of $\ch^*(K)$.  The interval
\[ \left\{\frac{\pi}{2}\right\} \times (2\pi - b - \delta, 2\pi - b + \delta) \]
also avoids the image of $\ch^*(K)$, again by Proposition~\ref{prop:pillowcase-symmetry}.

Next, since case~\eqref{i:arc-near-pi/2} of Theorem~\ref{thm:pillowcase-near-pi/2} is not satisfied, we claim that $\frac{\pi}{2}$ is not a limit point of the set
\[ \left\{ \alpha \neq \frac{\pi}{2} \,\middle|\, \exists \rho \in R^*(K) \mathrm{\ with\ pillowcase\ coordinates\ }(\alpha,\beta) \mathrm{\ for\ some\ }\beta\right\}. \]
This will guarantee for some $\epsilon > 0$ that any irreducible representation with $\alpha$-coordinate between $\frac{\pi}{2}-\epsilon$ and $\frac{\pi}{2}+\epsilon$ actually has $\alpha$-coordinate $\frac{\pi}{2}$, and taking $\epsilon < \delta$ without loss of generality, this will complete the proof.

To prove the claim, suppose that there is a sequence of irreducible representations $\rho_n: \pi_1(S^3 \ssm K) \to SU(2)$ with pillowcase coordinates $(\alpha_n,\beta_n)$, such that $\alpha_n \neq \frac{\pi}{2}$ and
\[ \lim_{n\to\infty} \alpha_n = \frac{\pi}{2}. \]
Using Proposition~\ref{prop:pillowcase-symmetry}, we can assume that $\alpha_n > \frac{\pi}{2}$ for all $n$; since $R(K)$ is compact and has finitely many path components, we can pass to a subsequence to ensure that they all lie in the same path component $C$ and converge to a limit $\rho$, with pillowcase coordinates $(\frac{\pi}{2},\beta)$ for some $\beta$.  Then $\rho$ is irreducible as argued above, and since $C$ is closed it must also contain $\rho$.  A continuous semi-algebraic path $\gamma=\gamma_t: [0,1] \to C$ from $\rho$ to $\rho_1$ (which exists by \cite[Proposition~2.5.13]{bcr}) must eventually leave the line $\alpha=\frac{\pi}{2}$ in the pillowcase, since $\alpha_1 > \frac{\pi}{2}$.  If $t_0$ is the last time at which $\gamma_{t_0}$ meets this line, then for some $\tau > 0$ the arc $\gamma|_{[t_0,t_0+\tau)}$ consists of points with distinct $\alpha$-coordinates; note that here we use the fact that $\gamma$ is not just continuous but semi-algebraic, so that $\cos(\alpha_t) = \frac{1}{2}\tr(\gamma_t(\mu))$ defines a semi-algebraic map $[0,1] \to \R$.  A reparametrization of this arc shows that $K$ must have satisfied case~\eqref{i:arc-near-pi/2} of Theorem~\ref{thm:pillowcase-near-pi/2} after all, which is a contradiction.
\end{proof}

Lemmas~\ref{le:repvariety_Y1_perturbed} and \ref{lem:space-near-pi/2} now point us toward the proof of Theorem~\ref{thm:pillowcase-near-pi/2}: we find an isotopy which takes $C$ to some curve in the pillowcase which avoids the image $i^*(\ch^*(K))$ of the irreducible characters of $K$, so that $X^w_{\pertdata(t)}(Y_1(K))$ consists of only finitely many points.  If these are irreducible and nondegenerate, then they will generate a complex whose homology is $\KHI(K)^{\oplus 2}$.  The following statement is more general than we need for Theorem~\ref{thm:pillowcase-near-pi/2}, but we will eventually use it again in a different setting to prove Theorem~\ref{thm:det-limit-slope}.

\begin{theorem} \label{thm:perturbed-khi-bound}
Let $K \subset S^3$ be a knot.  Let $C'$ be a smooth, simple closed curve in the pillowcase, and suppose that there is an area-preserving isotopy
\[ \overline{\psi}_t: \ch(T^2) \to \ch(T^2), \qquad 0 \leq t \leq 1 \]
which takes $C'$ to $C = \left\{\alpha=\frac{\pi}{2}\right\}$ and fixes the four orbifold points $(0,0)$, $(0,\pi)$, $(\pi, 0)$, and $(\pi,\pi)$.  Suppose that $C'$ intersects the line $\{\beta\equiv 0\pmod{2\pi}\}$ transversely in exactly $n \geq 1$ points $(\alpha_1,0),\dots,(\alpha_n,0)$, and that $e^{2i\alpha_j}$ is not a root of $\Delta_K(t)$ for any $j$.  If the pillowcase image $i^*(\ch^*(K))$ of the irreducible character variety of $K$ is disjoint from $C'$, then $\KHI(K)$ has rank at most $n$.
\end{theorem}

\begin{proof}
We first claim that $C'$ has an open neighborhood $U$ which is disjoint from $i^*(\ch^*(K))$.  Supposing otherwise, we can find conjugacy classes $[\rho_n] \in \ch^*(K)$ such that the distance from $i^*([\rho_n])$ to $C'$ goes to zero as $n \to \infty$, and since $R(K)$ is compact, some subsequence of the $\rho_n$ converges to a representation $\rho$; since $i^*([\rho])$ has distance zero from the compact set $C'$, it must lie in $C'$.  By hypothesis $\rho$ must then be reducible, but it is a limit of irreducibles and so it must have pillowcase coordinates $(\alpha,0)$ with $\Delta_K(e^{2i\alpha}) = 0$.  We have assumed that $C'$ avoids all such points, though, giving a contradiction.

Fixing $\epsilon > 0$, we now lift $\overline{\psi}_t$ to an isotopy $\psi_t: T^2 \to T^2$ of the branched double cover, apply Theorem~\ref{C1 approximation} to find a composition of $\Z/2$-equivariant shearing isotopies $\phi_t: T^2 \to T^2$ such that $\norm{\phi_t-\psi_t}_{C^1} < \epsilon$ for all $t \in [0,1]$, and project this back down to get an isotopy $\overline{\phi}_t: \ch(T^2) \to \ch(T^2)$ which is $C^1$-close to $\overline{\psi}_t$.  If we take $\epsilon$ sufficiently small, then this will guarantee first that $(\overline{\phi}_1)^{-1}(C)$ lies in the neighborhood $U$ of $C' = (\overline{\psi}_1)^{-1}(C)$, hence remains disjoint from $i^*(\ch^*(K))$; second, that it still avoids the points $(\alpha,0)$ with $\Delta_K(e^{2i\alpha})=0$; and third, that it still intersects the line $\{\beta \equiv 0 \pmod{2\pi}\}$ transversely in $n$ points.  Thus we may replace $C'$ with $(\overline{\phi}_1)^{-1}(C)$ and replace $\overline{\psi}_t$ with the isotopy $\overline{\phi}_t$ which lifts to $\phi_t$.

We now use the shearing isotopies which make up $\phi_t$ to perturb the flatness equation on
\[ Y_1(K) = \big(S^3 \ssm N(K)\big) \cup \big([0,1] \times T^2\big) \cup \big(F\times S^1\big) \]
to $F_A = \theta^t_{\pertdata}(A)$, exactly as described in Definition~\ref{one-parameter perturbed flat thickened torus}.  At time $t=1$, we recall that these perturbed flat connections are equivalently the critical points of a perturbed Chern-Simons functional $\mathit{CS}+\Phi$, as in \eqref{eq:perturbed-cs}.  By Lemma~\ref{le:repvariety_Y1_perturbed}, the space $X^w_{\pertdata(1)}(Y_1(K))$ of solutions to this equation modulo gauge transformations is a double cover of
\[ R(K\mid (\phi_1)^{-1}(\tilde{C})) = R(K \mid \tilde{C'}), \]
where $\tilde{C}$ is the lift of $C$ to $T^2$ satisfying $\alpha=\frac{\pi}{2}$ and $\tilde{C'} = (\phi_1)^{-1}(\tilde{C})$ is the corresponding lift of $C'$.  Thus by hypothesis $X^w_{\pertdata(1)}(Y_1(K))$ consists of $2n$ points $A_1, \dots, A_{2n}$.

Our goal is to associate an instanton Floer homology group $I_*^{\pertdata(1)}(Y_1(K))_w$ to this perturbation data and show that it is isomorphic to the usual group $I_*(Y_1(K))_w$; then the rank of the latter will be at most $2n$, and hence by definition $\KHI(K)$ will have rank at most $n$.  In order to carry this out, we must first show that all points of $X^w_{\pertdata(1)}(Y_1(K))$ are irreducible and nondegenerate, and then an additional small perturbation which we suppress from the notation will fix the set of generators while ensuring that the moduli spaces of trajectories which define the differential are transversely cut out as well.  (See \cite[\S5.5.1]{donaldson-book} for discussion of this last perturbation, or \cite[Proposition~3.18]{km-yaft} for a concise statement.)

For irreducibility, we note that the 2-forms $\mu_k$ appearing in the definition of $\theta^t_{\pertdata}(A)$ are supported entirely on $[0,1] \times T^2$.  In particular, any connection $A$ with equivalence class
\[ [A] \in X^w_{\pertdata(t)}(Y_1(K)) \]
restricts to a flat connection on both $S^3 \ssm N(K)$ and $F \times S^1$.  But flat connections on the bundle $\ad(E)|_{F \times S^1}$ are irreducible by construction, since $c_1(w)$ has odd pairing with the surface $\alpha \times S^1 \subset F\times S^1$ (where $E$, $w$, and $\alpha$ are as described in the proof of Theorem~\ref{thm:km-traceless}), and so $X^w_{\pertdata(t)}(Y_1(K))$ does not contain any reducible connections for any $t$, $0 \leq t \leq 1$.

Now we claim that each $A=A_j \in X^w_{\pertdata(1)}(Y_1(K))$, $j=1,\dots,2n$, is nondegenerate.  Writing $Y=Y_1(K)$, this is equivalent to the vanishing of the kernel of the operator
\[ K_{A,\Phi} = \twomatrix{0}{d_A^*}{d_A}{-{*}d_{A,\Phi}} \]
on $(\Omega^0\oplus\Omega^1)(Y;\ad(E))$, where $d_A$ and $d_A^*$ are the exterior derivative twisted by $A$ and its Hodge dual, and $d_{A,\Phi} = d_A - 4\pi^2\,{*}\operatorname{Hess}_A(\Phi)$; see \cite{floer,taubes} or \cite[Section~6]{herald-legendrian}.
By Hodge theory we can identify
\[ \ker(K_{A,\Phi}) \cong H^0_A(Y;\ad(E)) \oplus H^1_{A,\Phi}(Y;\ad(E)), \]
where the groups on the right are cohomology groups of the complex
\[ 0 \to \Omega^0(Y;\ad(E)) \xrightarrow{d_A} \Omega^1(Y;\ad(E)) \xrightarrow{d_{A,\Phi}} \Omega^2(Y;\ad(E)) \xrightarrow{d_A} \Omega^3(Y;\ad(E)) \to 0. \]
Since $A$ is irreducible we have $H^0_A(Y;\ad(E)) = 0$, so we wish to prove that $H^1_{A,\Phi}(Y;\ad(E)) = 0$ as well.  In the sequel we will drop the $\ad(E)$ from our notation.

This cohomology can be computed by a Mayer-Vietoris argument, in which we write
\[ Y = (S^3 \ssm N(K)) \cup_{T^2} V, \]
where $V$ is a neighborhood of $([0,1] \times T^2) \cup (F\times S^1)$, so that the perturbation $\Phi$ is supported in $V$ away from a collar neighborhood of $\partial V \cong T^2$.  Since $A$ is a reducible flat connection on both $S^3 \ssm N(K)$ and $T^2$ with the same $U(1)$ stabilizer, the restriction map
\[ H^0_A(S^3 \ssm N(K)) \to H^0_A(T^2) \]
is surjective.  Thus we have an exact sequence
\begin{equation} \label{eq:mayer-vietoris-Y}
0 \to H^1_{A,\Phi}(Y) \to H^1_A(S^3 \ssm N(K)) \oplus H^1_{A,\Phi}(V) \xrightarrow{h} H^1_A(T^2),
\end{equation}
and it suffices to show that the map $h$ is injective.  (We may drop $\Phi$ from two of the above subscripts because it is zero on $S^3 \ssm N(K)$ and on $T^2$.)

Denoting by $\rho = \Hol(A)$ the holonomy representation of $A$, with restrictions $\rho_T$ and $\rho_K$ to $T^2$ and $S^3 \ssm N(K)$ having centralizers $Z(\rho_T)$ and $Z(\rho_K)$ respectively, we know that
\[ \dim T_{\rho_T}R(T^2) = 3+\dim Z(\rho_T) = 4, \qquad \dim T_{\rho_K}R(K) = 3, \]
by \cite[Section~1.2]{goldman} and \cite[Theorem~19]{klassen} respectively; the latter claim uses the assertion that $\Delta_K(e^{2i\alpha}) \neq 0$.  Then we have
\[ \dim H^1_A(T^2) = \dim T_{\rho_T}R(T^2) - (3-\dim Z(\rho_T)), \]
since the tangent spaces at $\rho_T$ to $R(T^2)$ and to the $SU(2)$-orbit of $\rho_T$ are identified with $\ker(d_A|_{\Omega^1(T^2)})$ and $\img(d_A|_{\Omega^0(T^2)})$ respectively, and likewise for $\rho_K$.  By construction we have $\dim Z(\rho_T) = \dim Z(\rho_K) = 1$, since $\rho_T$ and $\rho_K$ are reducible (with image in a $U(1)$ subgroup of $SU(2)$) but not central.  Thus we have
\[ \dim H^1_A(T^2) = 2, \qquad \dim H^1_A(S^3 \ssm N(K)) = 1. \]
We will show that $\dim(H^1_{A,\Phi}(V)) = 1$.  In this case, since each $H^1$ is identified with the tangent space to the corresponding moduli space at $[A]$, the assumption that $C'$ meets the line $\beta \equiv 0 \pmod{2\pi}$ transversely at the image of $[\rho]=[\Hol(A)]$ will guarantee that the map $h$ in \eqref{eq:mayer-vietoris-Y} is injective, and hence that $H^1_{A,\Phi}(Y) = 0$.

We decompose $V = ([0,1]\times T^2) \cup_{T^2} (F\times S^1)$, where we have slightly enlarged the interval $[0,1]$ so that $\Phi$ is supported on the interior of $[0,1]\times T^2$, and where the intersection $T^2$ is identified with $\{1\}\times T^2 \subset [0,1]\times T^2$.  Again the Mayer-Vietoris sequence says that
\[ 0 \to H^1_{A,\Phi}(V) \xrightarrow{i} H^1_{A,\Phi}([0,1]\times T^2) \oplus H^1_A(F\times S^1) \xrightarrow{j} H^1_A(T^2) \]
is exact, since the restriction map $H^0_A([0,1]\times T^2) \to H^0_A(T^2)$ is surjective.  But restriction of connections from $[0,1]\times T^2$ to $T^2$ induces an isomorphism of the respective (perturbed) flat moduli spaces, hence the map $H^1_{A,\Phi}([0,1]\times T^2) \to H^1_A(T^2)$ between their tangent spaces at $[A]$ is an isomorphism.  It follows that $\ker(j)$ is isomorphic to $H^1_A(F\times S^1)$, which is 1-dimensional, and since $\img(i) \cong H^1_{A,\Phi}(V)$ we conclude that $\dim H^1_{A,\Phi}(V) = 1$, as desired.

The above argument has shown that each connection in the moduli space
\[ X^w_{\pertdata(1)}(Y) = \{ A_1,\dots, A_{2n} \} \]
is irreducible and nondegenerate.  Thus by standard methods, the perturbed Chern-Simons functional $\mathit{CS}+\Phi$ gives rise to an instanton Floer homology group $I_*^{\pertdata(1)}(Y)_w$, and we have shown that its rank is at most $2n$.  Since there are no reducible connections in any of the moduli spaces $X^w_{\pertdata(t)}(Y)$ for $0 \leq t \leq 1$, the cobordism argument presented in \cite[Section~5.3]{donaldson-book}  carries over to our situation and shows that $I_*^{\pertdata(1)}(Y)_w \cong I_*(Y)_w$.  Then $\KHI(K)$ is by definition the $+2$-eigenspace of the operator $\mu(\pt): I_*(Y)_w \to I_*(Y)_w$, which is isomorphic to the $-2$-eigenspace of $\mu(\pt)$ and hence has rank at most $n$.
\end{proof}

Here we remark on some points arising in the proof of Theorem~\ref{thm:perturbed-khi-bound}.  First, we could make the moduli space of perturbed flat connections regular by a generic holonomy perturbation, but such a perturbation cannot necessarily be confined to the $[0,1]\times T^2$ region of $Y_1(K)$: in general, one may need to introduce holonomy perturbations along a collection of curves which generates $\pi_1(Y_1(K))$ (see \cite[Section~5.5.1]{donaldson-book}), and this destroys the relationship of Lemma~\ref{le:repvariety_Y1_perturbed} between the perturbed moduli space and the pillowcase image of $R(K)$.  Thus we prefer to check directly that the relevant moduli space is already regular.

Second, in constructing instanton homology groups over 3-manifolds equipped with nontrivial bundles, one usually only allows small perturbations of the Chern-Simons functional in order to ensure that all critical points remain irreducible.  However, this is not a concern for us since our perturbed flat connections are still flat on $\ad(E)|_{F \times S^1}$ and hence irreducible.  Instanton homology groups defined with large perturbations have occasionally appeared in the literature before, e.g.\ in Floer's proof of the surgery exact triangle \cite{braam-donaldson}.

Third, the statement of the theorem requires us to have an \emph{area-preserving} isotopy taking $C$ to $C'$, and if we only know that there is an isotopy such that the image of $C$ divides the pillowcase in half by area at all times, then it is not obvious that we can find an isotopy which preserves area.  The following lemma provides the desired isotopy; the proof is a bit long but mostly standard, so we postpone it to Appendix~\ref{sec:symplectic-isotopy}.

\begin{lemma} \label{lem:symplectic-isotopy}
Let $\Sigma$ be a compact surface with area form $\omega$, and let $K$ be a compact subset of $\Sigma$.
Let $C_t \subset \Sigma$ ($0 \leq t \leq 1$) be a smooth isotopy of smoothly embedded curves (not necesarily connected), and suppose that
\begin{enumerate}
\item the union $\bigcup_{t\in[0,1]} C_t$ is contained in the interior of a compact, connected subsurface $\Sigma'$ with smooth boundary, such that $\Sigma'$ is disjoint from $K$;
\item each component $\gamma \subset C_t$ is separating in $\Sigma \setminus (C_t \setminus \gamma)$; 
\item the areas of the respective components of $\Sigma \setminus C_t$ are independent of $t$.
\end{enumerate}
Then there is a smooth isotopy $\psi_t: \Sigma \to \Sigma$ with $\psi_0 = \mathrm{id}_\Sigma$ such that
\begin{enumerate}
\item $\psi_t(C_0) = C_t$ for all $t$;
\item $\psi_t$ is constant on $\overline{\Sigma\setminus\Sigma'}$, and hence on a neighborhood of $K$, for all $t \in [0,1]$;
\item each $\psi_t$ is a symplectomorphism, i.e., $\psi_t^*\omega = \omega$ for all $t$.
\end{enumerate}
\end{lemma}

In the application to Theorem~\ref{thm:perturbed-khi-bound}, we take $(\Sigma,\omega)$ to be the pillowcase with its usual area form.  Then $K \subset \Sigma$ is the set of orbifold points, and $\Sigma'$ is the complement of sufficiently small open disks around these points.

We can now complete the proof of Theorem~\ref{thm:pillowcase-near-pi/2}.

\begin{proof}[Proof of Theorem~\ref{thm:pillowcase-near-pi/2}.]
Suppose that $K$ is a non-trivial knot and neither property (1) nor (2) in the statement of Theorem~\ref{thm:pillowcase-near-pi/2} holds. By Lemma \ref{lem:space-near-pi/2}, there are some $b \in (0,\pi)$ and $\epsilon > 0$ such that the line segment 
\[ 
L = \left\{ \left(\frac{\pi}{2}, \beta\right) \,\middle|\, \beta \in (b-\epsilon,b+\epsilon) \right\}
\]
of the pillowcase does not meet either of the lines $\beta = 0$ and $\beta = \pi$, 
and it contains no element of the image of $R^*(K,i)$. (The same statement holds for the image of $L$ under the involution $(\alpha,\beta) \mapsto (\pi-\alpha, 2\pi - \beta)$ of Proposition \ref{prop:pillowcase-symmetry}.) 
Furthermore, any irreducible character in $\ch^*(K)$ whose image in the pillowcase $\ch(T^2)$ has $\alpha$-coordinate satisfying $\frac{\pi}{2} - \epsilon < \alpha < \frac{\pi}{2} + \epsilon$ must actually have $\alpha$-coordinate equal to $\frac{\pi}{2}$. 

We may suppose that $\epsilon$ is smaller than $\pi/2$. Let $\delta = \frac{\epsilon^2}{4 \pi - 2 \epsilon}$. 
Let $C \subset \ch(T^2)$ be the embedded circle in the pillowcase defined by $\alpha = \frac{\pi}{2}$, and let $Z \subset \ch(T^2)$ be the piecewise linear closed curve which passes through the successive points
\begin{equation*}
\begin{split}
		z_0 & = \left(\frac{\pi}{2} - \delta, 0\right), \\
 		z_1 & = \left(\frac{\pi}{2}- \delta, b - \frac{\epsilon}{2}\right), \\ 
		z_2 & = \left(\frac{\pi}{2} + \frac{\epsilon}{2},b-\frac{\epsilon}{2}\right), \\ 
 		z_3 & = \left(\frac{\pi}{2} + \frac{\epsilon}{2},b+\frac{\epsilon}{2}\right), \\
		z_4 & = \left(\frac{\pi}{2} - \delta, b +\frac{\epsilon}{2}\right), \\
		z_5 & = \left(\frac{\pi}{2} - \delta, 2 \pi\right) = z_0.
\end{split}
\end{equation*}
The value of $\delta$ is chosen such that the curve $Z$ separates the pillowcase $\ch(T^2)$ into two halves of equal area with respect to the standard Euclidean area measure.  See Figure~\ref{perturbed curve}.
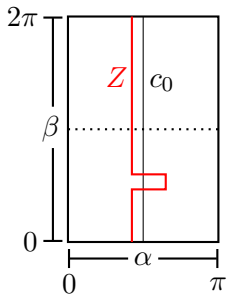
\begin{figure}
\begin{tikzpicture}[style=thick]
\begin{scope}
  \draw (0,0) rectangle (2,3);
  \draw[style=thin] (1,0) -- (1,3);
  \draw[dotted] (0,1.5) -- (2,1.5);
  \node at (1.25,2.12) {$c_0$};
  \draw[color=red] (0.85,0) -- (0.85,0.7) -- (1.3,0.7) -- (1.3,0.9) -- (0.85,0.9) -- (0.85,3);
  \node[color=red] at (0.65,2.2) {$Z$};
  \draw[|-] (-0.2,3) -- (-0.2,1.8);
  \node at (-0.23,1.5) {$\beta$};
  \draw[-|] (-0.2,1.25) -- (-0.2,0);
  \node at (-0.5,0) {$0$};
  \node at (-0.6,3) {$2\pi$};
  \draw[|-] (0,-0.22) -- (0.8,-0.22);
  \node at (1,-0.22) {$\alpha$};
  \draw[-|] (1.2,-0.22) -- (2,-0.22);
  \node at (2,-0.55) {$\pi$};
  \node at (0.02,-0.55) {$0$};
\end{scope}
\end{tikzpicture}
\caption{The curves $Z$ and $c_0 = \left\{\alpha=\frac{\pi}{2}\right\}$ in the pillowcase.}
\label{perturbed curve}
\end{figure}

It is clear that we can continuously approximate $Z$ arbitrarily well by rounding the corners $z_1$, $z_2$, $z_3$, and $z_4$ and fixing an open neighborhood of $z_0=z_5$ to get a new curve $C'$, which by Lemma~\ref{lem:symplectic-isotopy} is isotopic to $C$ through an area-preserving isotopy which fixes an open neighborhood of the corners of the pillowcase.  Then $C'$ still intersects the line $\{\beta \equiv 0 \pmod{2\pi}\}$ transversely in the single point $(\alpha,0) = (\frac{\pi}{2} - \delta, 0)$, and for $\epsilon$ sufficiently small we must have $\Delta_K(e^{2i\alpha}) \neq 0$, since $\Delta_K(e^{2i\cdot \pi/2}) \neq 0$.  We now apply Theorem~\ref{thm:perturbed-khi-bound} to the curve $C'$ to conclude that $\rank(\KHI(K)) \leq 1$.  But this is a contradiction just as in the proof of Theorem~\ref{thm:km-traceless}, since $\KHI(K)$ has rank strictly greater than 1 unless $K$ is unknotted.
\end{proof}

\section{Rationality of limit slopes} \label{sec:rationality}

In this section we prove that limit slopes of $SU(2)$-averse knots are always rational numbers, having already shown in Theorem~\ref{thm:limit-slope-finite} that they are finite.  In what follows we will let $\Qbar \subset \C$ denote the field of algebraic numbers, and $\Ralg = \Qbar \cap \R$ will denote the field of real algebraic numbers.

\begin{proposition} \label{prop:alpha-beta-algebraic}
Let $K$ be nontrivial and $SU(2)$-averse with limit slope $r = r(K)$, and let $(\alpha_0,\beta_0)$ be a point in the pillowcase image of $\ch^*(K)$.
\begin{enumerate}
\item If $e^{i\alpha_0}$ is algebraic, then $e^{i\beta_0}$ is algebraic. \label{i:alpha-algebraic}
\item If $e^{i\beta_0}$ is algebraic, then $e^{i\alpha_0}$ is algebraic. \label{i:beta-algebraic}
\item If $(\alpha_0,\beta_0)$ is an isolated point in the image of $\ch^*(K)$, then $e^{i\alpha_0}$ and $e^{i\beta_0}$ are both algebraic numbers. \label{i:isolated-point}
\end{enumerate}
\end{proposition}

\begin{proof}
Certainly $e^{i\alpha_0}$ and $e^{i\beta_0}$ are algebraic if and only if $e^{i(\pi-\alpha_0)}$ and $e^{i(2\pi-\beta_0)}$ are, so by Proposition~\ref{prop:pillowcase-symmetry} we can replace $(\alpha_0,\beta_0)$ with $(\pi-\alpha_0,2\pi-\beta_0)$ as needed to ensure that $0 \leq \beta_0 \leq \pi$.  We will take $\rho_0: \pi_1(S^3\ssm K) \to SU(2)$ to be an irreducible representation with $\rho_0(\mu) = \diag(e^{i\alpha_0},e^{-i\alpha_0})$ and $\rho_0(\lambda) = \diag(e^{i\beta_0},e^{-i\beta_0})$.

Theorem~\ref{thm:finitely-many-lines} implies that any sufficiently small neighborhood of $(\alpha_0,\beta_0)$ intersects the pillowcase image of $\ch^*(K)$ only in the point $(\alpha_0,\beta_0)$ if it is isolated, or in a connected segment $\ell$ of the line of slope $-r$ through $(\alpha_0,\beta_0)$ otherwise.  Therefore for sufficiently small rational $\epsilon>0$, if $(\alpha,\beta)$ is in the image of $\ch^*(K)$ and
\[ \left| \cos(\alpha) - \cos(\alpha_0) \right|^2 + \left| \cos(\beta) - \cos(\beta_0) \right|^2 < 2\epsilon^2, \]
then either $\beta > \pi$, or $(\alpha,\beta)=(\alpha_0,\beta_0)$, or $(\alpha,\beta) \in \ell$.  We then pick rational numbers $a$ and $b$ so that
\[ \left| a - 2\cos(\alpha_0) \right| < \epsilon \qquad\mathrm{and}\qquad \left| b - 2\cos(\beta_0) \right| < \epsilon, \]
and we define a semi-algebraic subset $R_0$ of $R^*(K)$ by imposing the inequalities
\[ \left(\tr(\rho(\mu))-a\right)^2 < \epsilon^2 \qquad\mathrm{and}\qquad \left(\tr(\rho(\lambda))-b\right)^2 < \epsilon^2, \]
as well as the relations $\rho(\mu) = \left(\begin{smallmatrix}*&0\\0&*\end{smallmatrix}\right)$ and $\rho(\lambda) = \left(\begin{smallmatrix}*&0\\0&*\end{smallmatrix}\right)$, and the additional inequalities 
$\img(\rho(\mu)_{1,1}) \geq 0$ and $\img(\rho(\lambda)_{1,1}) \geq 0$.  Here $\img(\rho(g)_{i,j})$ denotes the imaginary part of the $(i,j)$-entry of the matrix $\rho(g)$; this is always a polynomial with $\ZZ$-coefficients in the variables used to construct $R^*(K)$.  The relations on $\rho(\mu)$ and $\rho(\lambda)$ require them to have the form $\diag(e^{i\alpha},e^{-i\alpha})$ and $\diag(e^{i\beta},e^{-i\beta})$ respectively.  Then $\img(\rho(\mu)_{1,1}) = \sin(\alpha)$ and $\img(\rho(\lambda)_{1,1}) = \sin(\beta)$, both of which must be nonnegative, so we can take $0 \leq \alpha,\beta \leq \pi$ at any such $\rho \in R_0$, and then $\rho$ has pillowcase coordinates $(\alpha,\beta)$.

We will also impose one additional relation on $R_0$ in cases~\eqref{i:alpha-algebraic} and \eqref{i:beta-algebraic}.  If we are in case~\eqref{i:alpha-algebraic}, so that $e^{i\alpha_0}$ is algebraic, we also impose the relation $\tr(\rho(\mu)) = 2\cos(\alpha_0)$ on $R_0$.  Similarly, in case~\eqref{i:beta-algebraic} we impose $\tr(\rho(\lambda)) = 2\cos(\beta_0)$.  We note that $2\cos(\alpha_0)$ (resp.\ $2\cos(\beta_0)$) is algebraic whenever $e^{i\alpha_0}$ (resp.\ $e^{i\beta_0}$) is.

In each case $R_0$ is defined over the real closed field $\Ralg$, and $\R$ is a real closed extension of $\Ralg$, so the Tarski-Seidenberg principle implies that $R_0$ has an $\Ralg$-point if and only if it has an $\R$-point \cite[Proposition~4.1.1]{bcr}.  But $R_0$ contains the $\R$-point $\rho_0$ by construction, so $R_0$ must also have an $\Ralg$-point $\rho_\alg$, for which the real and imaginary parts of the entries of every matrix $\rho_\alg(g)$, $g \in \pi_1(S^3\ssm K)$, are real algebraic numbers.  In particular, if $\rho_\alg$ has image $(\alpha,\beta)$ in the pillowcase, then we have
\[ \rho_\alg(\mu) = \twomatrix{e^{i\alpha}}{0}{0}{e^{-i\alpha}}, \qquad \rho_\alg(\lambda) = \twomatrix{e^{i\beta}}{0}{0}{e^{-i\beta}} \]
where $0 \leq \alpha,\beta \leq \pi$, and $e^{i\alpha}$ and $e^{i\beta}$ are both algebraic numbers.

Since $\rho_\alg \in R_0$, we have $\lvert 2\cos(\alpha) - a\rvert = \lvert\tr(\rho(\mu))-a\rvert < \epsilon$ and $\lvert a-2\cos(\alpha_0)\rvert < \epsilon$, so $\lvert\cos(\alpha)-\cos(\alpha_0)\rvert < \epsilon$ by the triangle inequality; likewise $\lvert\cos(\beta)-\cos(\beta_0)\rvert < \epsilon$.  Since we ensured that $\beta \leq \pi$, it follows that if $(\alpha_0,\beta_0)$ is an isolated point then we must have $(\alpha,\beta) = (\alpha_0,\beta_0)$, and that otherwise $(\alpha,\beta) \in \ell$.  The latter only applies to cases~\eqref{i:alpha-algebraic} and \eqref{i:beta-algebraic} of the proposition, and then we know that $(\alpha_0,\beta_0)$ is the only point of $\ell$ satisfying $\alpha = \alpha_0$ or $\beta = \beta_0$, so again we must have $(\alpha,\beta) = (\alpha_0,\beta_0)$.  In each case we conclude that $e^{i\alpha_0}$ and $e^{i\beta_0}$ are algebraic, since this is true of $e^{i\alpha}$ and $e^{i\beta}$.
\end{proof}

\begin{proposition} \label{prop:exp-limit-slope}
Let $K$ be a nontrivial $SU(2)$-averse knot with limit slope $r = r(K)$, and suppose that the pillowcase image of $\ch(K)$ contains a nontrivial segment of the line $\beta = -r\alpha + c$.  Then $e^{i\pi r}$ and $e^{ic}$ are both algebraic numbers.
\end{proposition}

\begin{proof}
We pass to a fundamental domain $[0,\pi] \times [0,2\pi]$ for the cut-open pillowcase.  Since $r \neq \infty$, we can pick two points $(\alpha_1,\beta_1)$ and $(\alpha_2,\beta_2)$ in this domain, with $\alpha_1 < \alpha_2$, such that the image of $\ch(K)$ contains the entire line segment $\ell$ connecting these points.

If we fix an integer $n > \frac{2\pi}{\alpha_2-\alpha_1}$ and let $m = \lceil\frac{\alpha_1n}{\pi}\rceil$, then $\ell$ contains two points of the form $p = (\frac{m}{n}\pi, \beta_-)$ and $p' = (\frac{m+1}{n}\pi,\beta_+)$, and its slope is
\[ -r = \frac{n(\beta_+ - \beta_-)}{\pi}. \]
Both $e^{im\pi/n}$ and $e^{i(m+1)\pi/n}$ are algebraic numbers, so Proposition~\ref{prop:alpha-beta-algebraic} says that $e^{i\beta_-}$ and $e^{i\beta_+}$ are algebraic as well, and hence so is
\[ \left(\frac{e^{i\beta_-}}{e^{i\beta_+}}\right)^n = e^{-in(\beta_+-\beta_-)} = e^{i\pi r}. \]
Finally, we rearrange $\beta_- = -r(\frac{m}{n}\pi) + c$ and apply $x \mapsto \exp(ix)$ to both sides to get
\[ e^{ic} = e^{i\beta_-} \cdot \left(e^{i\pi r}\right)^{m/n}, \]
and since both $e^{i\beta_-}$ and $e^{i\pi r}$ are algebraic it follows that $e^{ic}$ is as well.
\end{proof}

\begin{remark}
We will eventually show in Theorem~\ref{thm:c-rational} that the constants $e^{ic}$ appearing in Proposition~\ref{prop:exp-limit-slope} are in fact roots of unity.  The proof will require the fact that $r$ is rational, which we are about to prove.
\end{remark}

\begin{theorem} \label{thm:limit-slope-rational}
Let $K$ be a nontrivial $SU(2)$-averse knot with limit slope $r = r(K)$.  Then $r$ is a rational number.
\end{theorem}

\begin{proof}
As in the proof of Proposition~\ref{prop:exp-limit-slope}, we take two points $p = (\frac{m}{n}\pi, \beta_-)$ and $p' = (\frac{m+1}{n}\pi, \beta_+)$ in the fundamental domain $[0,\pi] \times [0,2\pi]$ of the pillowcase (so that $0 \leq m < n$) such that the entire line segment $\ell$ from $p$ to $p'$ lies in the pillowcase image of $\ch^*(K)$; we recall that the existence of $p$ and $p'$ follows from Theorem~\ref{thm:open-pillowcase-image}.  The line containing $\ell$ has the form $\beta = -r\alpha + c$, and both $e^{i\pi r}$ and $e^{ic}$ are algebraic.

Given any $\alpha$ such that $e^{i\alpha}$ is algebraic and $\frac{m}{n}\pi < \alpha < \frac{m+1}{n}\pi$, we know that $\ell$ contains a point $(\alpha,\beta) = (\alpha, -r\alpha+c)$ in the image of $\ch^*(K)$, and that $e^{i\beta}$ is algebraic by Proposition~\ref{prop:alpha-beta-algebraic}.  But then
\[ e^{ir\alpha} = e^{i(c-\beta)} = \frac{e^{ic}}{e^{i\beta}} \]
is also an algebraic number, so we have shown for $\alpha$ in the range $(\frac{m}{n}\pi, \frac{m+1}{n}\pi)$ that if $e^{i\alpha}$ is algebraic then so is $(e^{i\alpha})^r$.  In fact, given any $a \in \R$ such that $e^{ia}$ is algebraic, we can let $\alpha = a+q\pi$ for some rational $q$ such that $\frac{m}{n}\pi < \alpha < \frac{m+1}{n}\pi$; then $e^{i\alpha}$ is also algebraic, so it follows from the above that $(e^{ia})^r = e^{i\alpha r} \cdot (e^{i\pi r})^{-q}$ is as well.

Finally, if $x_1,x_2,x_3$ are nonzero, multiplicatively independent algebraic numbers (meaning that if integers $n_1,n_2,n_3$ are not all zero, then $x_1^{n_1}x_2^{n_2}x_3^{n_3} \neq 1$), and if each $x_j^r$ is algebraic, then $r$ must be rational \cite[\S II.1, Corollary 1]{lang}.  According to Proposition~\ref{prop:log-algebraic-vector-space}, we can take $x_j = e^{ia_j}$ for some $a_1,a_2,a_3 \in \R$ such that the numbers $a_1,a_2,a_3,\pi$ are linearly independent over $\Q$; then $x_j^r$ is also algebraic, as argued above.  The condition $\prod x_j^{n_j} = 1$ is equivalent to $\sum n_j\alpha_j = 2k\pi$ for some integer $k$, and then $n_1=n_2=n_3=0$ by the assumed linear independence, so the $x_j$ are algebraically independent and hence $r$ is rational.
\end{proof}

The following proposition, which completes the proof of Theorem~\ref{thm:limit-slope-rational}, is surely known to experts, but we could not find a proof in the literature so we include one for completeness.

\begin{proposition} \label{prop:log-algebraic-vector-space}
Let $V$ denote the $\Q$-vector space of all $x \in \R$ such that $e^{ix}$ is an algebraic number.  Then $V$ has infinite dimension over $\Q$.
\end{proposition}

\begin{proof}
Suppose that $V$ has a finite basis $a_1,\dots,a_n$, and let $K = \Q(e^{ia_1},\dots,e^{ia_n})$.  Since each $e^{ia_j}$ is algebraic, $K$ is a finite extension of $\Q$.  Any nonzero $x\in V$ can be written as a $\Q$-linear combination of the $a_i$, so upon clearing denominators we have integers $q \geq 1$ and $c_1,\dots,c_n$ such that $qx = c_1a_1 + c_2a_2 + \dots, + c_na_n$, and then
\[ (e^{ix})^q = (e^{ia_1})^{c_1} (e^{ia_2})^{c_2} \dots (e^{ia_n})^{c_n} \in K. \]
So for every algebraic number $z \in \C$ with $\lvert z\rvert = 1$, there is a positive integer $q$ such that $z^q$ has degree at most $[K:\Q]$ over $\Q$.  If for any $d \geq 1$ we can find an algebraic number $z$ such that $|z|=1$ and every power $z^q$ ($q \geq 1$) has degree at least $d$, then taking $d > [K:\Q]$ will give a contradiction, proving that $\dim_\Q(V) = \infty$.

In order to find such $z$, we turn to Salem numbers, for which we refer to the survey \cite{smyth}.  These are real algebraic integers $s > 1$ all of whose Galois conjugates satisfy $|z| \leq 1$, and such that at least one conjugate satisfies $|z|=1$.  It follows from these properties that if $s$ has minimal polynomial $p(t) \in \Z[t]$ of degree $d$, then $p(s)=p(\frac{1}{s})=0$ and all other roots of $p$ lie on the unit circle \cite[Lemma~1]{smyth}; and that $s^q$ is also an algebraic integer of degree $d$ for all $q \geq 1$ \cite[Lemma~2]{smyth}.

Let $s$ be a Salem number with minimal polynomial $p(t)$ and splitting field $L$.  Letting $z$ be a root of $p$ with $|z|=1$ and taking $\sigma \in \Gal(L/\Q)$ for which $\sigma(s)=z$, we know that $z^q = \sigma(s^q)$ has the same minimal polynomial as $s^q$, since if $f \in \Z[t]$ is irreducible and $f(s^q) = 0$ then $f(z^q) = f(\sigma(s^q)) = \sigma(f(s^q)) = 0$ as well.  We conclude that $z^q$ is algebraic of degree $\deg(p)$ for all integers $q \geq 1$, so the proof will be complete if we can take $\deg(p)$ arbitrarily large, i.e.\ if there are Salem numbers of arbitrarily large degree.

Given any fixed $c > 0$, the set of Salem numbers $s \leq c$ of degree at most $d$ is finite.  This is because if $s$ is a Salem number with minimal polynomial $p(t) = t^d + a_1t^{d-1} + \dots + a_d \in \Z[t]$, then each coefficient $a_k$ is (up to sign) the sum of all $k$-fold products of distinct roots of $p$; there are ${d-1 \choose k-1}$ such products which involve $s$ and ${d-1 \choose k}$ which only involve the other roots, and since the other roots all have modulus at most 1, we get bounds $|a_k| \leq c{d-1 \choose k-1} + {d-1 \choose k}$.  But the set of Salem numbers has limit points, including what are called Pisot numbers \cite[Section~3.1]{smyth} (which include $\sqrt{2}+1$, for concreteness), so a bounded neighborhood of any limit point contains infinitely many Salem numbers and hence Salem numbers of arbitrarily high degree, as desired.
\end{proof}

\section{The $A$-polynomial of an $SU(2)$-averse knot} \label{sec:a-polynomial}

The $A$-polynomial of a knot $K \subset S^3$, defined by Cooper, Culler, Gillet, Long, and Shalen \cite{ccgls}, is a polynomial which captures the appropriate analogue of the pillowcase image of $\ch(K)$ for $SL_2(\C)$ representations of $\pi_1(S^3 \ssm K)$.  Since $SU(2) \subset SL_2(\C)$, it is possible to use what we know about $\ch(K)$ for an $SU(2)$-averse knot $K$ to say something about its $A$-polynomial, and by deep results of \cite{ccgls} we can then in turn understand the geometry of the complement of $K$.  Our goals are to show that the limit slope $r(K)$ is always a \emph{boundary slope} for $K$, meaning that there is an essential surface in the exterior of $K$ which intersects $\partial N(K)$ in parallel curves of slope $r(K)$ -- the set of such slopes for any fixed knot is finite \cite{hatcher} -- and to study the constants $c_j$ which appear in Theorem~\ref{thm:finitely-many-lines}.

We begin by recalling the construction of the $A$-polynomial, borrowing our notation from \cite{dunfield-garoufalidis}.  We define the $SL_2(\C)$ character variety of a manifold $M$ to be the quotient
\[ \chsl(M) = \Hom(\pi_1(M), SL_2(\C)) / SL_2(\C), \]
in which $SL_2(\C)$ acts by conjugation.  Given a knot $K\subset S^3$ and writing $\chsl(K)$ in place of $\chsl(S^3 \ssm N(K))$, the inclusion map $i: T^2 = \partial N(K) \hookrightarrow S^3 \ssm N(K)$ induces a restriction map
\[ i^*: \chsl(K) \to \chsl(T^2). \]
The part of $\chsl(T^2)$ consisting of conjugacy classes of diagonal representations $\rho$ has $\C^* \times \C^*$ as a branched double cover, where we take peripheral elements $\mu,\lambda \in \pi_1(T^2)$ and identify a point $(M,L) \in \C^*\times\C^*$ with the conjugacy class of $\rho$ such that
\[ \rho(\mu) = \twomatrix{M}{0}{0}{M^{-1}}, \qquad \rho(\lambda) = \twomatrix{L}{0}{0}{L^{-1}}. \]
(The covering map identifies $(M,L)$ with $(M^{-1},L^{-1})$, exactly as in the $SU(2)$ case.)

We let $V$ be the union of the closures $\overline{i^*(X)}$ over all irreducible components $X \subset \chsl(K)$ such that $i^*(X)$ has complex dimension 1; by \cite[Lemma~2.1]{dunfield-garoufalidis}, we know that $\dim_\C(i^*(X))$ is always either 0 or 1.  We then let $\overline{V}$ be the preimage of $V$ in $\C^* \times \C^*$.  Then $\overline{V}$ is a complex algebraic plane curve, so it is defined by a single polynomial in $\C[M,L]$, which is uniquely determined up to a choice of normalization and can even be taken with integer coefficients.  The $A$-polynomial $A_K(M,L) \in \Z[M,L]$ is this defining polynomial.

In what follows we will also make use of the \emph{Newton polygon} $\cN_f \subset \R^2$ of a polynomial $f \in \C[M,L]$: this is the convex hull in the $(M,L)$-plane of the lattice points $(a,b) \in \Z^2$ for which the $M^aL^b$-coefficient of $f$ is nonzero.  The following fact about Newton polygons is standard, but we include a proof for completeness.

\begin{lemma} \label{lem:minkowski-sum}
Given two nonzero polynomials $f,g \in \C[M,L]$, the Newton polygon $\cN_{fg}$ is the \emph{Minkowski sum} of $\cN_f$ and $\cN_g$, defined as the set of points $p+q$ where $p \in \cN_f$ and $q\in\cN_g$.
\end{lemma}

\begin{proof}
To prove that $\cN_{fg} \subset \cN_f + \cN_g$, we note that any vertex $(a,b)$ of $\cN_{fg}$ comes from a monomial $cM^aL^b$ of $fg$ with $c$ nonzero, and this means that there must be nonzero monomials $c_1M^{a_1}L^{b_1}$ and $c_2M^{a_2}L^{b_2}$ of $f$ and $g$ respectively with $a=a_1+a_2$ and $b=b_1+b_2$.  Then $(a_1,b_1) \in \cN_f$ and $(a_2,b_2) \in \cN_g$, and so $(a+b) = (a_1+b_1,a_2+b_2)$ belongs to $\cN_f+\cN_g$.  Since $\cN_f+\cN_g$ is convex and contains every vertex of $\cN_{fg}$, it contains all of $\cN_{fg}$.

For the opposite inclusion, we pick a vertex $v = (a,b)$ of $\cN_f+\cN_g$.  We claim that if we write $v = v_f+v_g$ where $v_f =(a_f,b_f) \in \cN_f$ and $v_g = (a_g,b_g) \in \cN_g$, then $v_f$ and $v_g$ are vertices of $\cN_f$ and $\cN_g$ and are uniquely determined by $v$.  Since they are vertices, we know that the coefficients of $M^{a_f}L^{b_f}$ in $f$ and of $M^{a_g}L^{b_g}$ in $g$ are nonzero.  The uniqueness implies that if we write $fg$ as a sum of monomials of the form $c_{f,a_1,b_1}c_{g,a_2,b_2}M^{a_1+a_2}L^{b_1+b_2}$, where $c_{f,a_1,b_1}M^{a_1}L^{b_1}$ is a nonzero monomial of $f$ and $c_{g,a_2,b_2}M^{a_2}L^{b_2}$ is a nonzero monomial of $g$, then exactly one summand is a scalar multiple of $M^aL^b$, and its coefficient is $c_{f,a_f,b_f}c_{g,a_g,b_g} \neq 0$.  But then $v=(a,b)$ also belongs to $\cN_{fg}$, and since $\cN_{fg}$ contains all of the vertices of $\cN_f+\cN_g$ we conclude that $\cN_f+\cN_g \subset \cN_{fg}$.

To prove the claim, suppose we have written the vertex $v=(a,b)$ of $\cN_f+\cN_g$ as a sum of two points $v_f = (a_f,b_f) \in \cN_f$ and $v_g = (a_g,b_g) \in \cN_g$.  If $v_f$ is not a vertex of $\cN_f$, then there is a nonzero $w\in\R^2$ and an $\epsilon > 0$ such that $\cN_f$ contains the line segment $v_f + tw$, $-\epsilon<t<\epsilon$, and then $\cN_f+\cN_g$ contains the line segment
\[ \left\{ \left(v_f+tw\right) + v_g = v+tw \mid -\epsilon < t < \epsilon \right\}. \]
But this contradicts the claim that $v$ is a vertex of $\cN_f+\cN_g$, so $v_f$ must have been a vertex of $\cN_f$, and likewise for $v_g \in \cN_g$.  Now if we can write $v$ as a sum of elements of $\cN_f$ and $\cN_g$ in two different ways, say
\[ v = v_f + v_g = v'_f + v'_g, \qquad v_f,v'_f \in \cN_f, \ v_g,v'_g \in \cN_g, \]
then by the convexity of $\cN_f$ and $\cN_g$ we also have
\[ v = \left(\frac{v_f+v'_f}{2}\right) + \left(\frac{v_g+v'_g}{2}\right) \in \cN_f + \cN_g. \]
But $\frac{v_f+v'_f}{2}$ cannot be a vertex of $\cN_f$, since it is the midpoint of a line segment in $\cN_f$, so this is a contradiction and we conclude that the representation $v=v_f+v_g$ is unique.
\end{proof}

Lemma~\ref{lem:minkowski-sum} works equally well for Laurent polynomials $f,g \in \C[M^{\pm1},L^{\pm1}]$, and we will generally consider the $A$-polynomial in this setting.

\begin{lemma} \label{lem:irreducible-newton}
If $p,q$ are relatively prime integers which are not both zero and $k$ is a nonzero complex number, then the polynomial $f(M,L) = M^pL^q - k \in \C[M^{\pm 1},L^{\pm 1}]$ is irreducible.
\end{lemma}

\begin{proof}
The Newton polygon $\cN_f$ is the line segment from $(0,0)$ to $(p,q)$, which only intersects $\Z^2$ at its two endpoints; there are no lattice points on its interior since $p$ and $q$ are relatively prime.  If $f = gh$ where $g$ and $h$ are not monomials, then $\cN_g$ and $\cN_h$ each contain at least two lattice points as well, so their Minkowski sum $\cN_{gh} = \cN_f$ must have at least three lattice points and this is impossible.
\end{proof}

\begin{proposition} \label{prop:a-polynomial-linear-factor}
Let $K$ be a nontrivial knot, and suppose that the pillowcase image of $\ch^*(K)$ contains a nontrivial line segment of slope $-r \in \Q$; we write $r = \frac{p}{q}$, where $p$ and $q$ are relatively prime integers and $q \geq 1$.  If this segment belongs to the line $\beta \equiv -r\alpha + c \pmod{2\pi}$, then $M^pL^q - e^{iqc}$ divides $A_K(M,L)$.
\end{proposition}

\begin{proof}
By hypothesis there are infinitely many $\rho \in R^*(K)$ whose conjugacy classes have distinct images in the pillowcase $\ch(T^2)$, and which satisfy $\rho(\mu^p\lambda^q) = \diag(e^{iqc},e^{-iqc})$.  We view these classes $[\rho]$ as points of $\chsl(K)$ via the inclusion $SU(2) \subset SL_2(\C)$.  Since $\chsl(K)$ has finitely many irreducible components, some irreducible component $X \subset \chsl(K)$ contains infinitely many of the conjugacy classes $[\rho]$, whose images under $i^*$ are all distinct.  Then $i^*(X)$ cannot be 0-dimensional, so $\dim_\C(i^*(X)) = 1$.

Suppose that the preimage of $i^*(X)$ in $\C^* \times \C^*$ has defining polynomial $P(M,L)$, and let $f(M,L) = M^pL^q - e^{iqc}$.  Then the equation $f(M,L) = 0$ is satisfied by infinitely many points of the plane curve $P(M,L) = 0$, namely each of the $i^*([\rho])$ described above, and so $f(M,L)$ and $P(M,L)$ have a common factor.   This factor must be $f(M,L) = M^pL^q - e^{iqc}$ itself, since Lemma~\ref{lem:irreducible-newton} says that $f$ is irreducible, and so $f$ divides $P$.  But $P$ divides $A_K(M,L)$ by definition, so $f(M,L)$ does as well.
\end{proof}

\begin{theorem} \label{thm:c-rational}
Let $K$ be a nontrivial knot, and suppose for some $r \in \Q$ that the pillowcase image of $\ch^*(K)$ contains a nontrivial segment of the line $\beta \equiv -r\alpha + c \pmod{2\pi}$.  Then
\begin{itemize}
\item $r$ is the boundary slope of an essential surface in $S^3 \ssm N(K)$; and
\item $e^{ic}$ is a root of unity, or equivalently $c \in 2\pi\Q$.
\end{itemize}
\end{theorem}

\begin{proof}
Write $r = \frac{p}{q}$ as usual.  From Proposition~\ref{prop:a-polynomial-linear-factor} we know that $f(M,L) = M^pL^q - e^{iqc}$ divides the $A$-polynomial of $K$, so we can write $A_K(M,L) = f(M,L) \cdot g(M,L)$ for some nonzero polynomial $g$.  The Newton polygon $\cN_{A_K}$ is the Minkowski sum of $\cN_f$ and $\cN_g$, and $\cN_f$ is a line segment of slope $\frac{q}{p} = \frac{1}{r}$.

Let $h(M,L)$ be the sum of all monomials of $g(M,L)$ which correspond to points of the vertex or side of $\cN_g$ on which $-qM+pL$ is maximized.  Then $\cN_h$ is either a point or a line segment of slope $\frac{q}{p} = \frac{1}{r}$ on the boundary of $\cN_g$, and $\cN_{A_K}$ also has a side $E$ of slope $\frac{1}{r}$, namely the Minkowski sum $\cN_{fh}$ of $\cN_f$ and $\cN_h$.

Now \cite[Theorem~3.4]{ccgls} says that the slopes of the sides of $\cN_{A_K}$ are all boundary slopes of essential surfaces.  Their result is stated in the $(L,M)$-plane whereas we have described $\cN_{A_K}$ in the $(M,L)$-plane, so the side $E$ of $\cN_{A_K}$ with slope $\frac{1}{r}$ actually produces a surface with boundary slope $r$, as desired.

In fact, if we write $A_K(M,L) = \sum_{m,n} b_{mn}M^mL^n$ and define the \emph{edge polynomial}
\[ \Theta_E(z) = \sum_{(m,n) \in E} b_{mn}z^m, \]
then \cite[Proposition~5.10]{ccgls} asserts that $\Theta_E(z)$ is a product of cyclotomic polynomials.  Writing $h(M,L) = \sum_{j=0}^k a_j M^{pj+m_0}L^{qj+l_0}$ for some integers $k,m_0,l_0$ and constants $a_j$, and recalling that $E$ is the Newton polygon of $fh$, we have
\[ \Theta_E(z) = (z^p - e^{iqc}) \left(\sum_{j=0}^k a_j z^{pj+m_0}\right). \]
In particular, the factor $z^p-e^{iqc}$ divides a product of cyclotomic polynomials, so its roots are all roots of unity.  But these roots are rational powers of $e^{ic}$, and therefore $e^{ic}$ must be a root of unity as well.
\end{proof}

\begin{theorem} \label{thm:su2-averse-rational}
Let $K$ be a nontrivial $SU(2)$-averse knot.  Then the limit slope $r = r(K)$ is a boundary slope for $K$, and the pillowcase image of $\ch^*(K)$ consists of finitely many isolated points together with finitely many nontrivial segments of lines $\beta \equiv -r\alpha + c_j$, where $c_j \in 2\pi\Q$ for all $j$.
\end{theorem}

\begin{proof}
Theorem~\ref{thm:open-pillowcase-image} guarantees the existence of a nontrivial segment of slope $-r$ in the image of $\ch^*(K)$, and $r \in \Q$ by Theorems~\ref{thm:limit-slope-finite} and \ref{thm:limit-slope-rational}.  Theorem~\ref{thm:c-rational} now says that $r$ is a boundary slope, and that for all such segments the corresponding constant $c_j$ is a rational multiple of $2\pi$.
\end{proof}

Theorem~\ref{thm:su2-averse-rational} gives us a criterion in terms of Alexander polynomials by which we can show that many knots are not $SU(2)$-averse.

\begin{proposition} \label{prop:alexander-root-averse}
Let $K \subset S^3$ be a knot with Alexander polynomial $\Delta_K(t)$, normalized so that $\Delta_K(t) = \Delta_K(t^{-1})$ and $\Delta_K(1)=1$.  If $\Delta_K(t)$ has a root $z=e^{i\theta}$ of odd order which is not a root of unity, then $K$ is not $SU(2)$-averse.  In particular, if $\Delta_K(t)$ has no cyclotomic factors and $\Delta_K(t_0) < 0$ for some $t_0=e^{i\theta_0}$, then $K$ is not $SU(2)$-averse.
\end{proposition}

\begin{proof}
We first check that the given conditions on the factorization and values of $\Delta_K(t)$ suffice to produce a root $z$ of the desired form.  Since $\Delta_K(t)$ has real coefficients and $\Delta_K(t) = \Delta_K(t^{-1})$, we have
\[ \overline{\Delta_K(e^{i\theta})} = \Delta_K\left(\overline{e^{i\theta}}\right) = \Delta_K(e^{-i\theta}) = \Delta_K(e^{i\theta}) \]
for all real $\theta$, and so $\Delta_K(e^{i\theta}) \in \R$.  The continuous function $f: \R \to \R$ defined by $f(\theta) = \Delta_K(e^{i\theta})$ satisfies $f(0) = \Delta_K(1) = 1$ and $f(\theta_0) = \Delta_K(t_0) < 0$ by hypothesis, so the total multiplicity of the zeroes of $f$ in the interval $[0,\theta_0]$ must be odd, and hence some $\theta \in [0,\theta_0]$ is a zero of odd multiplicity.  Then $z=e^{i\theta}$ is a zero of $\Delta_K(t)$ with odd multiplicity, and by assumption it is not a root of unity.

In general, given the odd-order root $z=e^{i\theta}$ and letting $\gamma = \frac{\theta}{2}$, it now follows as in work of Heusener and Kroll \cite{heusener-kroll} that the abelian representation $\rho_\gamma: \pi_1(S^3 \ssm K) \to SU(2)$ with $\rho_\gamma(\mu) = \diag(e^{i\gamma}, e^{-i\gamma})$ is a limit of irreducible representations.  In fact, one can adapt their proof of \cite[Proposition~3.8]{heusener-kroll} to show that the pillowcase image of $\ch^*(K)$ has $(\gamma,0)$ as a limit point.  The rough idea is that the equivariant signatures $\sigma_K(e^{2i(\gamma \pm \epsilon)})$ are signed counts of points in the spaces of irreducible $\rho$ with $\tr(\rho(\mu)) = 2\cos(\gamma \pm \epsilon)$, and if $(\gamma,0)$ is not a limit point then these spaces are cobordant as in \cite{heusener-kroll}, hence $\sigma_K(e^{i(\theta-2\epsilon)})=\sigma_K(e^{i(\theta+2\epsilon)})$.  We briefly indicate the details below, adopting the same notation as in \cite{heusener-kroll}.

Let $\sigma \in B_n$ be a braid with closure $K$ and take $\epsilon>0$ sufficiently small, so that in particular $\gamma$ is the only root of $\Delta_K(e^{2it})$ in the interval $\gamma-\epsilon \leq t \leq \gamma+\epsilon$.  We define
\[ R_{2n} = \{(A_1,A_2,\dots,A_{2n}) \in SU(2)^{2n} \mid \tr(A_i) = \tr(A_j) \ \forall i,j \} \setminus \{ \pm(1,1,\dots,1) \}, \]
and then define several subspaces of $R_{2n}$ by
\begin{align*}
H_n &= \{(A_1,\dots,A_n,B_1,\dots,B_n) \mid A_1\dots A_n=B_1\dots B_n \}, \\
H_n^\alpha &= \{(A_1,\dots,A_n,B_1,\dots,B_n) \in H_n \mid \tr(A_1)=2\cos(\alpha)\}, \\
\Lambda_n & = \{(A_1,\dots,A_n,A_1,\dots,A_n)\}, \\
\Gamma_\sigma &= \{A_1,\dots,A_n, \sigma(A_1),\dots,\sigma(A_n)\}.
\end{align*}
(In the last case we use the fact that the action of $\sigma$ on the free group of rank $n$ induces an action of $\sigma$ on $SU(2)^n$, which we write $\sigma(A_1,\dots,A_n) = (\sigma(A_1),\dots,\sigma(A_n))$ as in \cite{heusener-kroll}.)  For each $\Theta \in \{H_n,H_n^\alpha,\Lambda_n,\Gamma_\sigma\}$ there is a diagonal action of $SU(2)$ by conjugation on $\Theta$, and we write $\hat\Theta$ to denote the quotient $\Theta / SU(2)$, minus the orbits of tuples $(A_1,\dots,A_{2n})$ such that all $A_i$ and $A_j$ commute.

Having set this up, Heusener and Kroll identify the intersection
\[ \hat{H} := \left(\hat\Lambda_n \cap \hat\Gamma_\sigma\right) \cap \bigcup_{\alpha \in [\gamma-\epsilon,\gamma+\epsilon]} \hat{H}^\alpha_n \]
in $\hat{H}_n$ with the space of conjugacy classes of irreducible representations $\rho: \pi_1(S^3 \ssm K) \to SU(2)$ satisfying
\[ 2\cos(\gamma-\epsilon) \leq \tr(\rho(\mu)) \leq 2\cos(\gamma+\epsilon). \]
The set $\hat{H}$ is not quite compact because it contains a sequence which limits to the conjugacy class $[\rho_\gamma] \not\in \hat{H}$.  Supposing that $(\gamma,0)$ is an isolated point, however, its preimage in $\hat{H}$ is a union of connected components of $\hat{H}$ and is contained in $\hat{H}^\gamma_n$, so we let $\hat{H}' \subset \hat{H}_n$ be the complement of these components.   Then $\hat{H}'$ is a compact cobordism from
\[ \hat\Lambda^{\gamma-\epsilon}_n \cap \tilde\Gamma^{\gamma-\epsilon}_\sigma \subset \hat{H}^{\gamma-\epsilon}_n \qquad\mathrm{to}\qquad \hat\Lambda^{\gamma+\epsilon}_n \cap \tilde\Gamma^{\gamma+\epsilon}_\sigma \subset \hat{H}^{\gamma+\epsilon}_n, \]
which implies as in \cite{heusener-kroll} that $h^{\gamma-\epsilon}(K) = h^{\gamma+\epsilon}(K)$.  Now \cite[Theorem~1.2]{heusener-kroll} gives us $\sigma_K(e^{2i(\gamma-\epsilon)}) = \sigma_K(e^{2i(\gamma+\epsilon)})$, where $\sigma_K$ denotes equivariant signature, and since $\Delta_K(t)$ has a root of odd order at $e^{2i\gamma}$, we have a contradiction.

Finally, we suppose that $K$ is $SU(2)$-averse with limit slope $r(K) = \frac{p}{q}$.  As shown above, we can find a sequence of distinct points in the pillowcase image of $\ch^*(K)$ which limit to $(\gamma,0)$, and by passing to a subsequence we can take them all to lie in the same path component.  This path component must then be a nontrivial segment of the line of slope $-\frac{p}{q}$ passing through $(\gamma,0)$, namely
\[ p\alpha + q\beta \equiv p\gamma \pmod{2\pi}, \]
but $p\gamma$ is not a rational multiple of $2\pi$ since $z = e^{2i\gamma}$ is not a root of unity.  This contradicts Theorem~\ref{thm:su2-averse-rational}, so $K$ cannot be $SU(2)$-averse after all.
\end{proof}

\begin{corollary}
If $K$ has Alexander polynomial $\Delta_K(t) = at - (2a-1) + at^{-1}$ for some integer $a \geq 2$, then $K$ is not $SU(2)$-averse.
\end{corollary}

\begin{proof}
We have $\Delta_K(-1) < 0$, and $\Delta_K(t)$ is irreducible and not cyclotomic.
\end{proof}

\begin{example} \label{ex:pretzel-non-averse}
The pretzel knot $K = P(-2,3,7)$ has Alexander polynomial
\[ \Delta_K(t) = t^5 - t^4 + t^2 - t + 1 - t^{-1} + t^{-2} - t^{-4} + t^{-5}, \]
which has a root of odd order on the unit circle since $\Delta_K(i) = -3$.  Then $\Delta_K(t)$ is irreducible and not cyclotomic (in fact, $-t^5\Delta_K(-t)$ is the minimal polynomial of the smallest known Salem number: see \cite[Section~2]{smyth}), so this root is not a root of unity.  Proposition~\ref{prop:alexander-root-averse} thus ensures that $P(-2,3,7)$ is not $SU(2)$-averse.
\end{example}

\section{Small knots} \label{sec:small}

We recall that a knot $K \subset S^3$ is \emph{small} if its complement does not contain any closed  incompressible surfaces other than boundary-parallel tori.  In this section we use the results of Section~\ref{sec:a-polynomial}, together with work of Boyer and Zhang \cite{boyer-zhang-seminorms}, to develop further restrictions on the limit slopes of small $SU(2)$-averse knots.

\begin{theorem} \label{thm:small-knot-integer}
If $K$ is a small, $SU(2)$-averse knot, then its limit slope $r(K)$ is an integer.  In this case, we have $\rho(\mu^{r(K)}\lambda) = \pm I$ for all but finitely many conjugacy classes of irreducible representations $\rho: \pi_1(S^3 \ssm K) \to SU(2)$.
\end{theorem}

\begin{proof}
We write $r(K) = \frac{p}{q}$ with $p$ and $q$ relatively prime.  Since $K$ is small, every irreducible component of the $SL_2(\C)$ character variety $\chsl(K)$ is a curve \cite[Proposition~2.4]{ccgls}.  One of these components, say $X_0$, contains infinitely many distinct points which are the images of $SU(2)$-representations on a line of the form $p\alpha+q\beta=c$ but with distinct $\beta$-coordinates in the pillowcase, so $\tr(\rho(\mu^p\lambda^q)) = 2\cos(c)$ at infinitely many points on the curve $X_0$, hence $\tr(\rho(\mu^p\lambda^q))$ is constant on $X_0$.  The image of $X_0$ in the $PSL_2(\C)$ character variety of $K$ is then a $\frac{p}{q}$-curve in the terminology of Boyer-Zhang \cite{boyer-zhang-seminorms} -- we note that the traces of other slopes cannot also be constant on $X_0$, since for example $\tr(\rho(\lambda))$ is not constant -- and the meridian of $K$ is not a boundary slope \cite[Theorem~2.0.3]{cgls}, so by \cite[Corollary~6.7]{boyer-zhang-seminorms}, the slope $\frac{p}{q} = r(K)$ must be integral.

Letting $r = r(K)$, any nontrivial segment of a line $\beta = -r\alpha + c$ in the pillowcase image of $\ch^*(K)$ produces an $r$-curve $X_0$ as above, and clearly $X_0$ contains the character of an irreducible representation.  Since the exterior of $K$ is small and admits a cyclic filling of slope $\infty \neq r$, we apply \cite[Proposition~5.7(2)]{boyer-zhang-seminorms} to see that $\tr(\rho(\mu^r\lambda)) = \pm 2$ on all of $X_0$, hence the $SU(2)$-representations which contribute to $X_0$ must satisfy $\rho(\mu^r\lambda) = \pm I$.  Thus $\rho(\mu^r\lambda) = \pm I$ for every $\rho$ with pillowcase image on the line $\beta = -r\alpha + c$.

All but finitely many isolated points in the pillowcase image of $\ch^*(K)$ belong to a nontrivial line segment, so it remains to be seen that an isolated point in the pillowcase can only be the image of finitely many conjugacy classes of $SU(2)$ representations.  Suppose that infinitely many conjugacy classes $[\rho_n] \in R^*(K)$ have the same image $(\alpha,\beta)$.  The $\rho_n$ remain pairwise non-conjugate as $SL_2(\C)$ representations by \cite[Proof of Proposition 15]{klassen}, so they produce infinitely many points in $\chsl(K)$ and hence in some irreducible component $X_1$ of the $PSL_2(\C)$ character variety.  By \cite[Proposition~5.7]{boyer-zhang-seminorms}, the function $f_\gamma = \tr(\rho(\gamma))^2-4 : X_1 \to \C$ cannot be constant for two different slopes $\gamma$, but it is clearly constant for the slopes $\mu$ and $\lambda$ and this is a contradiction.
\end{proof}

The claim in Theorem~\ref{thm:small-knot-integer} that $\rho(\mu^{r(K)}\lambda) = \pm I$ for representations along a nontrivial segment of $\beta = -r\alpha+c$ can also be proved using the $A$-polynomial as follows.  Proposition~\ref{prop:a-polynomial-linear-factor} says that $M^rL - e^{ic}$ divides $A_K(M,L)$, so $e^{ic}$ is a root of $A_K(1,L)$.  Since $K$ is small, the latter is equal to $\pm(L+1)^\alpha (L-1)^\beta L^\gamma$ for some integers $\alpha,\beta,\gamma$ by \cite[Corollary~4.5]{cooper-long-remarks}, so its nonzero root $e^{ic}$ must be $\pm 1$.

\begin{example} \label{ex:twist-knots}
Let $K_n$ denote the twist knot shown in Figure~\ref{twist knot}, where $n \in \frac{1}{2}\Z$.  Then $K_n$ is a two-bridge knot, hence small, and moreover $K_n$ is hyperbolic for all $n$ other than $-1$, $-\frac{1}{2}$, $0$, and $\frac{1}{2}$, for which it is the right-handed trefoil, unknot, unknot, and left-handed trefoil, respectively.  We claim that no hyperbolic twist knot is $SU(2)$-averse; since $K_n$ is the mirror of $K_{-\frac{1}{2}-n}$, it suffices to consider $n \in \Z$.

Indeed, suppose for some integer $n \not\in \{-1,0\}$ that $K_n$ is $SU(2)$-averse with limit slope $r = r(K_n)$.  Then just as in the proof of Theorem~\ref{thm:small-knot-integer} we see that the $SL_2(\C)$ character variety $\chsl(K_n)$ has an irreducible component $X_0$, containing irreducible characters, for which the function $\tr(\rho(\mu^r\lambda))|_{X_0}$ is constant.  Burde \cite[Section~3]{burde} showed that all irreducible $SL_2(\C)$ characters of $\pi_1(S^3 \ssm K_n)$ lie on a single irreducible curve in $\chsl(K)$, so this curve must be $X_0$; then since $K_n$ is hyperbolic we see that $X_0$ must also contain the character of a discrete faithful $SL_2(\C)$ representation, which exists by a theorem of Thurston \cite[Proposition~3.1.1]{culler-shalen-splittings}.  But then \cite[Proposition~2]{culler-shalen-bounded} says that if $\gamma$ is any nontrivial peripheral curve then the function $\tr(\rho(\gamma))$ cannot be constant on $X_0$, so we have a contradiction.
\end{example}

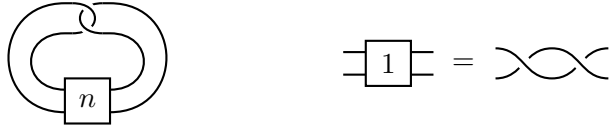
\begin{figure}
\begin{tikzpicture}[style=thick]
\begin{scope} 
  \draw (-0.3,-0.15) .. controls (-1.3,-0.15) and (-1.3,1.3) .. (-0.3,1.3) -- (-0.2,1.3);
  \draw (0.3,-0.15) .. controls (1.3,-0.15) and (1.3,1.3) .. (0.3,1.3) -- (0.2,1.3);
  \draw (-0.3,0.15) .. controls (-0.9,0.15) and (-0.9,0.9) .. (-0.3,0.9) -- (-0.2,0.9);
  \draw (0.3,0.15) .. controls (0.9,0.15) and (0.9,0.9) .. (0.3,0.9) -- (0.2,0.9);
  \begin{scope} 
    \clip (-0.4,0.8) -- (-0.4,1.4) -- (0.2,1.4) -- (0.2,0.85) -- (0,1.05) -- (-0.1,0.95) -- (0.05,0.8) -- cycle;
    \draw (-0.2,1.3) .. controls (0.2,1.3) and (0.2,0.9) .. (-0.2,0.9);
  \end{scope}
  \begin{scope} 
    \clip (-0.2,1.4) -- (-0.2,0.8) -- (0.4,0.8) -- (0.4,1.4) -- (-0.02,1.4) -- (0.18,1.15) -- (0,1.15) -- (-0.1,1.25) -- (0.05,1.4) -- cycle;
    \draw (0.2,1.3) .. controls (-0.2,1.3) and (-0.2,0.9) .. (0.2,0.9);
  \end{scope}
  \draw (-0.3,-0.3) rectangle (0.3,0.3);
  \node at (0,0) {$n$};
\end{scope}
\begin{scope}[shift={(4,0.5)}] 
  \begin{scope}
    \clip (-0.3,0.3) -- (-0.3,-0.3) -- (-1,-0.3) -- (-1,0.3) -- (1,0.3) -- (1,-0.3) -- (0.3,-0.3) -- (0.3,0.3) -- cycle;
    \draw (-0.6,0.15) -- (0.6,0.15);
    \draw (-0.6,-0.15) -- (0.6,-0.15);
  \end{scope}
  \draw (-0.3,-0.3) rectangle (0.3,0.3);
  \node at (0,0) {$1$};
  \node at (1,0) {$=$};
  \begin{scope}[xscale=0.75, shift={(-0.1,0)}] 
    \begin{scope}
      \clip (1.8,-0.3) -- (1.8,0.3) -- (4.2,0.3) -- (4.2,-0.3) -- (3,-0.3) -- (2.5,0.155) -- (2.35,0) -- (2.7,-0.3) -- cycle;
      \draw (2,-0.2) .. controls (2.5,-0.2) and (2.5,0.2) .. (3,0.2) .. controls (3.5,0.2) and (3.5,-0.2) .. (4,-0.2);
    \end{scope}
    \begin{scope}
      \clip (1.8,-0.3) -- (1.8,0.3) -- (4.2,0.3) -- (4.2,-0.3) -- (4,-0.3) -- (3.5,0.155) -- (3.35,0) -- (3.7,-0.3) -- cycle;
      \draw (2,0.2) .. controls (2.5,0.2) and (2.5,-0.2) .. (3,-0.2) .. controls (3.5,-0.2) and (3.5,0.2) .. (4,0.2);
    \end{scope}
  \end{scope}
\end{scope}
\end{tikzpicture}
\caption{The twist knot $K_n$.}
\label{twist knot}
\end{figure}

\begin{theorem} \label{thm:det-limit-slope}
Let $K$ be a small, $SU(2)$-averse knot with boundary slope $r(K)$ and with Alexander polynomial
\[ \Delta_K(t) = \sum_{j=-d}^d a_jt^j, \]
normalized so that $\Delta_K(t)=\Delta_K(t^{-1})$ and $\Delta_K(1) = 1$.  Then
\[ \det(K) \leq \sum_{j=-d}^{d} |a_j| \leq \rank(\KHI(K)) \leq |r(K)|-1, \]
where $\KHI(K)$ is the instanton knot homology of $K$.
\end{theorem}

\begin{proof}
The leftmost two inequalities come from the fact that $\KHI(K)$ categorifies the Alexander polynomial \cite{km-alexander,lim}, meaning that it has a decomposition
\[ \KHI(K) = \bigoplus_{j\in\Z} \KHI(K,j) \]
into canonically $\Z/2\Z$-graded summands satisfying
\[ \Delta_K(t) = - \sum_{j\in\Z} \chi(\KHI(K,j))t^j, \]
and so $\det(K) = |\Delta_K(-1)| \leq \sum_j |a_j| \leq \sum_j \rank(\KHI(K,j)) = \rank(\KHI(K))$.  Thus we only need to prove that $\rank(\KHI(K)) \leq |r(K)|-1$.

This last inequality will be proved by an application of Theorem~\ref{thm:perturbed-khi-bound}.  It suffices to find an area-preserving isotopy of the pillowcase which fixes the four corners and carries $C = \{\alpha = \frac{\pi}{2}\}$ to a curve $C'$ which avoids both the image of $\ch^*(K)$ and the points $(\alpha,0)$ where $\Delta_K(e^{2i\alpha})=0$.  If $C'$ intersects the line $\beta \equiv 0 \pmod{2\pi}$ transversely in at most $|r(K)|-1$ points, then this will prove the desired inequality.

Since $K$ is small, all but finitely many conjugacy classes of irreducible representations $\rho: \pi_1(S^3 \ssm K) \to SU(2)$ satisfy $\rho(\mu^r\lambda) = \pm I$ where $r=r(K)$, and hence their images $(\alpha,\beta)$ in the pillowcase satisfy $r\alpha+\beta \equiv 0 \pmod{\pi}$.  We take $\epsilon > 0$ so that the irreducible characters all satisfy $\epsilon < \alpha < \pi-\epsilon$.  We then let $U$ be a $\delta$-neighborhood of the set
\[ \left\{(\alpha,\beta) \ \middle|\ r\alpha+\beta \in \pi\Z, \frac{\epsilon}{2} \leq \alpha \leq \frac{\pi}{2} \right\} \cup \left\{ \left(\frac{\pi}{2},\beta\right) \ \middle| \ \beta \in \R/2\pi\Z \right\}, \]
where $0 < \delta < \frac{\epsilon}{2}$, and let $Z_0$ denote the component of $\partial U$ contained in the region $\alpha < \frac{\pi}{2}$.  The curves corresponding to different choices of $\delta$ are all disjoint, so only finitely many of them can pass through the images of irreducible characters or through points $(\alpha,0)$ where $\Delta_K(e^{2i\alpha})=0$.  Taking $\delta$ sufficiently small thus ensures that $Z_0$ avoids all such points.

Next, we perturb $Z_0$ to a smooth curve $Z$ which separates $\ch(T^2)$ into two components of equal area, by rounding corners and then pushing some segment of $Z_0$ slightly into the $\alpha > \frac{\pi}{2}$ region of the pillowcase away from the lines $r\alpha+\beta\equiv 0 \pmod{\pi}$ and $\beta\equiv 0 \pmod{2\pi}$, as shown in Figure~\ref{small knot perturbation}.  This is possible because the difference in area between the two components of $\ch(T^2) \ssm Z_0$ approaches zero as $\delta$ does, and once again it can easily be arranged to avoid all images of irreducible characters.  We also arrange that $Z$ crosses the line $\beta \equiv 0 \pmod{2\pi}$ transversely.  It does so twice near each point $(\alpha,0)$ on the pillowcase with $r\alpha \in \pi\Z$ and $0 < \alpha < \frac{\pi}{2}$, and once near $(\frac{\pi}{2},0)$, for a total of $|r|-1$ points if $r$ is even and $|r|$ points if $r$ is odd.
\begin{figure}
\begin{tikzpicture}[style=thick]
\begin{scope}[scale=1.5] 
  \draw[style=thin] (1,0) -- (1,3);
  \draw[dotted] (0,1.5) -- (2,1.5);
  \node at (1.15,2.1) {$C$};
  \begin{scope}[color=red,style=semithick,rounded corners=0.75mm] 
    \draw (0.94,3) -- (0.94,1.9) -- (0.567,3);
    \draw (0.567,0) -- (0.15,1.25) to[out=108,in=108,looseness=2] (0.15,3*0.28) -- (0.43,0);
    \draw (0.43,3) -- (0.94,3-3*0.51) -- (0.94,1.3) -- (1.2,0.52) -- (1.2,0.17) -- (0.94,0.95) -- (0.94,0.39) -- (0.15,2.76) to[out=105,in=105,looseness=2] (0.15,2.37) -- (0.94,0);
    \node at (0.2,1.75) {$Z$};
  \end{scope}
  \begin{scope}[dotted] 
    \clip (0,0) rectangle (2,3);
    \foreach \i in {-1,0,...,3}
      \draw (0.5*\i,3) -- (1+0.5*\i,0);
  \end{scope}
  \draw (0,0) rectangle (2,3);
  \draw[|-] (-0.2,3) -- (-0.2,1.8);
  \node at (-0.23,1.5) {$\beta$};
  \draw[-|] (-0.2,1.25) -- (-0.2,0);
  \node at (-0.5,0) {$0$};
  \node at (-0.6,3) {$2\pi$};
  \draw[|-] (0,-0.22) -- (0.8,-0.22);
  \node at (1,-0.22) {$\alpha$};
  \draw[-|] (1.2,-0.22) -- (2,-0.22);
  \node at (2,-0.55) {$\pi$};
  \node at (0.02,-0.55) {$0$};
\end{scope}
\begin{scope}[scale=1.5,shift={(4,0)}] 
  \draw[style=thin] (1,0) -- (1,3);
  \draw[dotted] (0,1.5) -- (2,1.5);
  \node at (1.15,2.5) {$C$};
  \begin{scope}[color=red,style=semithick,rounded corners=0.75mm] 
    \draw (0.94,3) -- (0.94,2.7) -- (0.87,3);
    \draw (0.87,0) -- (0.15,3.75*0.72) to[out=105,in=105,looseness=2] (0.15,3.75*0.58) -- (0.73,0);
    \draw (0.73,3) -- (0.94,3-3.75*0.19) -- (0.94, 2) -- (1.2, 1.025) -- (1.2,0.425) -- (0.94, 1.4) -- (0.94,1.2) -- (0.47,3);
    \draw (0.47,0) -- (0.15,3.75*0.32) to[out=105,in=105,looseness=2] (0.15,3.75*0.18) -- (0.33,0);
    \draw (0.33,3) -- (0.94, 3 - 3.75*0.61) -- (0.94,0);
    \node at (0.14,1.7) {$Z$};
  \end{scope}
  \begin{scope}[dotted] 
    \clip (0,0) rectangle (2,3);
    \foreach \i in {-1,0,...,4}
      \draw (0.4*\i,3) -- (0.8+0.4*\i,0);
  \end{scope}
  \draw (0,0) rectangle (2,3);
  \draw[|-] (-0.2,3) -- (-0.2,1.8);
  \node at (-0.23,1.5) {$\beta$};
  \draw[-|] (-0.2,1.25) -- (-0.2,0);
  \node at (-0.5,0) {$0$};
  \node at (-0.6,3) {$2\pi$};
  \draw[|-] (0,-0.22) -- (0.8,-0.22);
  \node at (1,-0.22) {$\alpha$};
  \draw[-|] (1.2,-0.22) -- (2,-0.22);
  \node at (2,-0.55) {$\pi$};
  \node at (0.02,-0.55) {$0$};
\end{scope}
\end{tikzpicture}
\caption{A perturbation $Z$ of the line $C = \left\{\alpha=\frac{\pi}{2}\right\}$ in the pillowcase, shown for $r=4$ (left) and $r=5$ (right).  The diagonal dotted lines are the points where $r\alpha+\beta \equiv 0 \pmod{\pi}$.}
\label{small knot perturbation}
\end{figure}
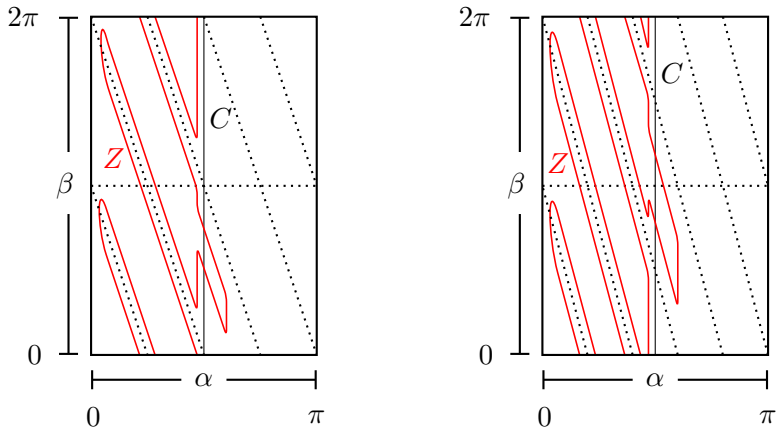

In the case where $r$ is odd, we can perturb $Z$ further to remove two of its points of intersection with $\beta\equiv 0 \pmod{2\pi}$, as shown in Figure~\ref{small knot perturbation for r odd}.  The idea is that in a fundamental domain $[0,\pi]\times[0,2\pi]$ for the pillowcase, if we assume for convenience that $r>0$ (the case $r<0$ is nearly identical), then there is a triangle with vertices
\[ \left(\frac{r-1}{2r}\pi, 2\pi\right), \qquad \left(\frac{\pi}{2},2\pi\right), \qquad \left(\frac{\pi}{2},\frac{3\pi}{2}\right), \]
whose area is $\frac{\pi^2}{8r}$, and an arc of $Z$ which intersects this triangle in two points along the line $\beta=2\pi$.  We can push this arc up across the line $\beta=2\pi$ to remove these two points, and the total area lost on that side of $Z$ (i.e., the side incident to $\alpha=0$) is less than $\frac{\pi^2}{8r}$.  We then offset this change in area by simultaneously pushing $Z$ further into the interior of the quadrilateral bounded by the points
\[ \left(\frac{r+1}{2r}\pi, 0\right), \qquad \left(\frac{\pi}{2},\frac{\pi}{2}\right), \qquad \left(\frac{\pi}{2},\frac{3\pi}{2}\right), \qquad \left(\frac{r+3}{2r}\pi,0\right), \]
whose area is $\frac{\pi^2}{r}$, so that $Z$ still divides $\ch(T^2)$ into two halves of equal area.

The end result is that the curve $C'=Z$ intersects $\beta\equiv 0\pmod{2\pi}$ in $|r|-1$ points if $r$ is even and $|r|-2$ points if $r$ is odd, and Lemma~\ref{lem:symplectic-isotopy} gives us an area-preserving isotopy of $X(T^2)$ which fixes a neighborhood of the four corners and carries $C'$ to $C$, so we conclude that $\rank(\KHI(K)) \leq |r|-1$.
\begin{figure}
\begin{tikzpicture}[style=thick]
\begin{scope}[scale=1.5] 
  \draw[style=thin] (1,0) -- (1,3);
  \draw[dotted] (0,1.5) -- (2,1.5);
  \node at (1.15,2.5) {$C$};
  \begin{scope}[color=red,style=semithick,rounded corners=0.75mm] 
    \draw (0.94,3) -- (0.94,2.7) -- (0.87,3);
    \draw (0.87,0) -- (0.15,3.75*0.72) to[out=105,in=105,looseness=2] (0.15,3.75*0.58) -- (0.73,0);
    \draw (0.73,3) -- (0.94,3-3.75*0.19) -- (0.94, 2) -- (1.2, 1.025) -- (1.2,0.425) -- (0.94, 1.4) -- (0.94,1.2) -- (0.47,3);
    \draw (0.47,0) -- (0.15,3.75*0.32) to[out=105,in=105,looseness=2] (0.15,3.75*0.18) -- (0.33,0);
    \draw (0.33,3) -- (0.94, 3 - 3.75*0.61) -- (0.94,0);
    \node at (0.14,1.7) {$Z$};
  \end{scope}
  \begin{scope}[dotted] 
    \clip (0,0) rectangle (2,3);
    \foreach \i in {-1,0,...,4}
      \draw (0.4*\i,3) -- (0.8+0.4*\i,0);
  \end{scope}
  \draw (0,0) rectangle (2,3);
  \draw[|-] (-0.2,3) -- (-0.2,1.8);
  \node at (-0.23,1.5) {$\beta$};
  \draw[-|] (-0.2,1.25) -- (-0.2,0);
  \node at (-0.5,0) {$0$};
  \node at (-0.6,3) {$2\pi$};
  \draw[|-] (0,-0.22) -- (0.8,-0.22);
  \node at (1,-0.22) {$\alpha$};
  \draw[-|] (1.2,-0.22) -- (2,-0.22);
  \node at (2,-0.55) {$\pi$};
  \node at (0.02,-0.55) {$0$};
\end{scope}
\begin{scope}[scale=1.5,shift={(3,1.5)}]
\node at (-0.2,0) {\scalebox{2}{$\leadsto$}};
\end{scope}
\begin{scope}[scale=1.5,shift={(4,0)}] 
  \draw[style=thin] (1,0) -- (1,3);
  \draw[dotted] (0,1.5) -- (2,1.5);
  \node at (1.15,2.5) {$C$};
  \begin{scope}[color=red,style=semithick,rounded corners=0.75mm] 
    \draw (0.33,3) -- (0.94, 3 - 3.75*0.61) -- (0.94,0.1) -- (0.8433,0.1) -- (0.15,3.75*0.72) to[out=105,in=105,looseness=2] (0.15,3.75*0.58) -- (0.73,0);
    \draw (0.73,3) -- (0.94,3-3.75*0.19) -- (0.94, 2) -- (1.37, 0.425) -- (1.2,0.425) -- (0.94, 1.4) -- (0.94,1.2) -- (0.47,3);
    \draw (0.47,0) -- (0.15,3.75*0.32) to[out=105,in=105,looseness=2] (0.15,3.75*0.18) -- (0.33,0);
    \node at (0.14,1.7) {$Z$};
  \end{scope}
  \begin{scope}[dotted] 
    \clip (0,0) rectangle (2,3);
    \foreach \i in {-1,0,...,4}
      \draw (0.4*\i,3) -- (0.8+0.4*\i,0);
  \end{scope}
  \draw (0,0) rectangle (2,3);
  \draw[|-] (-0.2,3) -- (-0.2,1.8);
  \node at (-0.23,1.5) {$\beta$};
  \draw[-|] (-0.2,1.25) -- (-0.2,0);
  \node at (-0.5,0) {$0$};
  \node at (-0.6,3) {$2\pi$};
  \draw[|-] (0,-0.22) -- (0.8,-0.22);
  \node at (1,-0.22) {$\alpha$};
  \draw[-|] (1.2,-0.22) -- (2,-0.22);
  \node at (2,-0.55) {$\pi$};
  \node at (0.02,-0.55) {$0$};
\end{scope}
\end{tikzpicture}
\caption{If $r$ is odd, we can perturb the curve $Z$ from Figure~\ref{small knot perturbation} so that it only intersects the curve $\beta\equiv 0 \pmod{2\pi}$ in $|r|-2$ points.}
\label{small knot perturbation for r odd}
\end{figure}
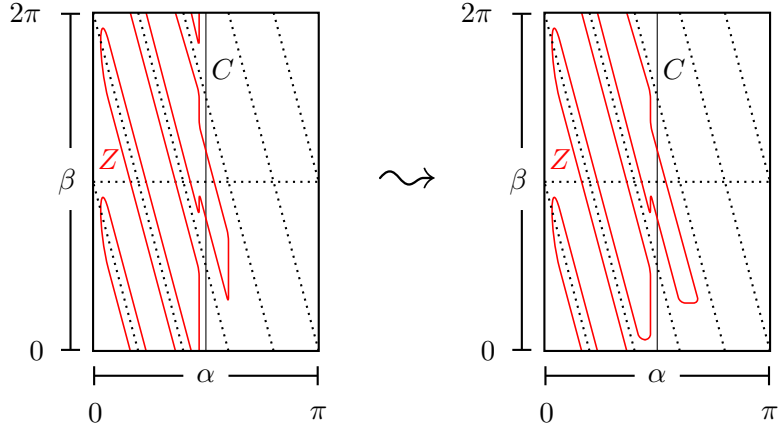
\end{proof}

Although Proposition~\ref{prop:alexander-root-averse} does not say anything useful about knots whose Alexander polynomials have zeros at roots of unity, we can use such zeros to deduce useful information for small knots as follows.

\begin{proposition} \label{prop:small-alexander-root}
Let $K$ be a small knot which is $SU(2)$-averse.  If the Alexander polynomial $\Delta_K(t)$ has a root of odd order at $t=e^{i\cdot 2\pi p/q}$, where $p$ and $q$ are relatively prime integers, then the limit slope $r(K)$ is a multiple of $q$.
\end{proposition}

\begin{proof}
Let $r=r(K)$, and note that the line $r\alpha + \beta \equiv r\cdot \frac{\pi p}{q} \pmod{2\pi}$ has slope $-r$ and passes through the image $(\frac{\pi p}{q},0)$ of the character of the reducible representation $\rho: \pi_1(S^3 \ssm K) \to SU(2)$ with $\rho(\mu) = \diag(e^{i\pi p/q}, e^{-i\pi p/q})$.  Arguing exactly as in the proof of Proposition~\ref{prop:alexander-root-averse}, we see that the pillowcase image of $\ch^*(K)$ contains a nontrivial segment of this line.  By Theorem~\ref{thm:small-knot-integer}, we have $r \cdot \frac{\pi p}{q} \equiv 0 \pmod{\pi}$, and so $q$ divides $r$ as claimed.
\end{proof}

\begin{proposition} \label{prop:small-r-6}
If $K$ is a small, $SU(2)$-averse knot, then $|r(K)| \geq 6$.
\end{proposition}

\begin{proof}
Suppose that $|r(K)| \leq 5$.  Then Theorem~\ref{thm:det-limit-slope} says that $\rank(\KHI(K)) \leq 4$, and in fact the rank of $\KHI(K)$ is odd (since it categorifies the Alexander polynomial and $\Delta_K(1) = 1$ is odd), so we have $\rank(\KHI(K)) \leq 3$.  By assumption $K$ is not the unknot, so we must have $\rank(\KHI(K)) = 3$.

Kronheimer and Mrowka proved that $\rank \KHI(K,g) \geq 1$ \cite[Proposition~7.16]{km-excision}, where $g \geq 1$ is the Seifert genus of $K$.  Since $\KHI(K,j) \cong \KHI(K,-j)$ for all $j$, it follows that $\KHI(K,j)$ has rank one for each of $j=-g,0,g$ and rank zero for all other $j$, and this immediately implies that
\[ \Delta_K(t) = t^g - 1 + t^{-g}. \]
But then $\Delta_K(t)$ has a simple root at $t = e^{i \cdot 2\pi/(6g)}$, so Proposition~\ref{prop:small-alexander-root} tells us that $|r(K)|$, which is an integer between 1 and 5, is a multiple of $6g \geq 6$, which is absurd.
\end{proof}

We remark that Proposition~\ref{prop:small-r-6} is sharp, since the right- and left-handed trefoils are $SU(2)$-averse with limit slope $\pm 6$.

\section{$SU(2)$-cyclic surgeries from the pillowcase} \label{sec:converse}

In this section we prove a converse to Theorem~\ref{thm:finitely-many-lines}.  Our goal is to say that if the pillowcase image of $\ch^*(K)$, the irreducible part of the character variety, is contained in finitely many lines of rational slope $-r$, then $K$ must be $SU(2)$-averse with limit slope $r$. We will use the following lemma to prove such a result.

\begin{lemma} \label{lem:cyclic-slope-approximation}
Let $r, c \in \R$ be constants, and fix a rational number $\frac{p}{q} \neq r$ with $q>0$ and $p$ relatively prime to $q$.  Let $\{x\} = x - \lfloor x\rfloor$ denote the fractional part of $x$.  If
\[ 0 \leq \frac{q}{2}\left(\frac{p}{q}-r\right) + \left\{\frac{qc}{2\pi}\right\} \leq 1, \]
then no point on the line $\beta=-r\alpha+c$, $0 < \alpha < \pi$, satisfies $p\alpha+q\beta \equiv 0 \pmod{2\pi}$.
\end{lemma}

\begin{proof}
Let $k = \lfloor \frac{qc}{2\pi} \rfloor$.  Then the given inequalities are equivalent to
\[ k \leq \frac{q}{2}\left(\frac{p}{q}-r\right) + \frac{qc}{2\pi} \leq k+1, \]
which can be rearranged to get
\[ \frac{-p\pi + 2\pi k}{q} \leq - r\pi + c \leq \frac{-p\pi + 2\pi (k+1)}{q}. \]
Thus we have a pair of linear inequalities in $\alpha$ of the form
\[ \frac{-p\alpha + 2\pi k}{q} \leq -r\alpha + c \leq \frac{-p\alpha + 2\pi (k+1)}{q}, \]
which are satisfied at $\alpha = 0$ (since our choice of $k$ implies that $\frac{2\pi k}{q} \leq c < \frac{2\pi (k+1)}{q}$) and at $\alpha = \pi$, and hence for all $\alpha$ in between.  In fact, each of these inequalities must be strict for $0 < \alpha < \pi$: otherwise, since $\frac{p}{q}\neq r$, the corresponding inequality would be violated on one side of the $\alpha$ in question.  Thus $\beta = -r\alpha+c$ lies strictly between $\beta = \frac{-p\alpha+2\pi k}{q}$ and $\beta = \frac{-p\alpha + 2\pi (k+1)}{q}$ for $0 < \alpha < \pi$, and so it cannot contain any point on the lines $p\alpha + q\beta \equiv 0 \pmod{2\pi}$, as claimed.
\end{proof}

\begin{theorem} \label{thm:finite-lines-converse}
Let $K$ be a nontrivial knot, and suppose that there is a constant $r\in\Q$ such that the pillowcase image of $\ch^*(K)$ is contained in the union of finitely many isolated points and finitely many nontrivial segments of lines $\beta = -r\alpha + c_i$.  Then there is some $N \geq 1$ such that $(r + \frac{1}{kN})$-surgery on $K$ is $SU(2)$-cyclic for all but finitely many $k \in \Z$.  In particular, $K$ is $SU(2)$-averse.
\end{theorem}

The hypotheses of Theorem~\ref{thm:finite-lines-converse} are satisfied by all $SU(2)$-averse knots: this is guaranteed for some $r \in \mathbb{RP}^1$ by Theorem~\ref{thm:finitely-many-lines}, and then Theorems~\ref{thm:limit-slope-finite} and \ref{thm:limit-slope-rational} assert that $r$ is rational.

\begin{proof}
According to Theorem~\ref{thm:c-rational}, the constants $c_i$ are all rational multiples of $2\pi$.  Let $n \geq 1$ be an integer such that $nr$ and each of the $\frac{nc_i}{\pi}$ are integers, and let $m = nr$ and $N = n^2$.  For a given slope $\frac{p}{q} \in \Q$ not equal to $r$, with $q>0$ and $p$ relatively prime to $q$, if the inequalities
\[ 0 \leq \frac{q}{2}\left(\frac{p}{q} - r\right) + \left\{\frac{qc_i}{2\pi}\right\} \leq 1 \]
are satisfied for all $i$, then Lemma~\ref{lem:cyclic-slope-approximation} says that no point $(\alpha,\beta)$ on one of the lines $\beta = -r\alpha + c_i$ which is the image of an irreducible $\rho: \pi_1(S^3 \ssm K) \to SU(2)$  can satisfy $p\alpha + q\beta \equiv 0 \pmod{2\pi}$ (recall that $0 < \alpha < \pi$, by Proposition~\ref{prop:pillowcase-facts}).  We will take $(p,q) = (knm+1, kn^2)$ for any integer $k \geq 1$, noting that $p$ and $q$ are relatively prime since $km^2q - (knm-1)p = 1$.  We compute for each $i$ that
\begin{align*}
\frac{q}{2}\left(\frac{p}{q} - r\right) + \left\{\frac{qc_i}{2\pi}\right\} &= \frac{kn^2}{2}\left(\frac{knm+1}{kn^2} - \frac{m}{n}\right) + \left\{kn \cdot \frac{nc_i}{2\pi}\right\} \\
&= \frac{1}{2} + \left\{\frac{1}{2} \cdot kn \cdot \frac{nc_i}{\pi}\right\},
\end{align*}
which is either $\frac{1}{2}$ or $1$, so the irreducible representations on the lines $\beta = -r\alpha + c_i$ do not produce irreducible representations of the $\frac{kmn+1}{kn^2}$-surgery on $K$.

We are left with finitely many isolated points $(\alpha_j,\beta_j)$; letting $c'_j = \beta_j + r\alpha_j$, we may assume without loss of generality that $c'_j \not\in\ 2\pi\Q$, since otherwise we could have just added the line $\beta=-r\alpha + c'_j$ to the collection of lines considered above, possibly at the cost of taking a larger value of $n$.  If for some $k \geq 1$ we have $(kmn+1)\alpha_j + kn^2\beta_j  \in 2\pi\Z$, then this together with $kn^2(\frac{m}{n}\alpha_j + \beta_j - c'_j) = 0$ implies that $\alpha_j + kn^2c'_j \in 2\pi\Z$.  For fixed $j$, there is at most one $k$ for which this can be satisfied since $c'_j$ is not a rational multiple of $2\pi$, and hence each $(\alpha_j,\beta_j)$ produces an irreducible representation of at most one $\frac{kmn+1}{kn^2}$-surgery on $K$.  Therefore the surgery on $K$ of slope $\frac{kmn+1}{kn^2} = r + \frac{1}{kN}$ is $SU(2)$-cyclic for all but finitely many $k \geq 1$.

Next, the mirror $\mirror{K}$ satisfies the same hypotheses with $-r$ and $2\pi-c_i$ in place of $r$ and $c_i$ by Proposition~\ref{prop:mirror-image}, so we can take the same value of $n$ (and hence $N$) as above and conclude that all but finitely many $(-r+\frac{1}{lN})$-surgeries on $\mirror{K}$ are $SU(2)$-cyclic when $l \geq 1$.  These are the $(r-\frac{1}{lN})$-surgeries on $K$ up to reversing orientation, so taking $k = -l$ we see that all but finitely many $(r+\frac{1}{kN})$-surgeries on $K$ with $k \leq -1$ are $SU(2)$-cyclic as well.
\end{proof}

If $K$ is small and $SU(2)$-averse, then by Theorem~\ref{thm:small-knot-integer} we have $r(K) \in \Z$ and $c_i \in \pi \Z$, so if there are no isolated points in the pillowcase image of $X^*(K)$ then we can take $n=1$ and hence $N=1$ in the proof of Theorem~\ref{thm:finite-lines-converse}.  In this case, we conclude that every slope $r(K) + \frac{1}{k}$ (with $k\in\Z$ nonzero) is an $SU(2)$-cyclic slope for $K$.

\begin{corollary} \label{cor:common-slopes}
If $K_1,K_2,\dots,K_m$ are $SU(2)$-averse knots which all have the same limit slope, then they have infinitely many $SU(2)$-cyclic surgery slopes in common.
\end{corollary}

\begin{proof}
Let $r\in\Q$ be the common limit slope and take $N_1,N_2,\dots,N_m \geq 1$ such that for each $i$, all but finitely many of the $(r + \frac{1}{kN_i})$-surgeries on $K_i$ are $SU(2)$-cyclic.  Then for all but finitely many $k$, all of the $(r+\frac{1}{kN})$-surgeries on the various $K_i$ are $SU(2)$-cyclic, where $N = \lcm(N_1,N_2,\dots,N_m)$.
\end{proof}

\section{Relation to instanton L-spaces} \label{sec:l-spaces}

Let $Y$ be a closed, connected, oriented 3-manifold, and let $I^\#(Y)$ be the singular instanton knot homology $I^\#(Y,\emptyset) = I^\natural(Y,U)$ of the empty link in $Y$, as defined by Kronheimer and Mrowka \cite{km-khovanov}.  This is a $\Z/4\Z$-graded abelian group, and if $Y$ is a rational homology sphere then it has Euler characteristic $|H_1(Y;\Z)|$, as shown by Scaduto \cite[Corollary~1.4]{scaduto}.  Thus we have a rank inequality
\[ \rank(I^\#(Y)) \geq |H_1(Y;\Z)|, \]
and if equality holds then we say that $Y$ is an \emph{instanton L-space}.

\begin{theorem} \label{thm:l-space-surgeries}
If $K$ is $SU(2)$-averse with limit slope $r = r(K)$, then $S^3_s(K)$ is an instanton L-space for all rational
\[ s \in \begin{cases}
\big[\lceil r\rceil -1, \infty\big), & r > 0 \\
\big({-}\infty, \lfloor r\rfloor+1\big], & r < 0.
\end{cases} \]
In either case, $S^3_r(K)$ is an instanton L-space.
\end{theorem}

\begin{remark} \label{rem:l-space-knot}
A knot $K$ with instanton L-space surgeries is fibered, and either $K$ or its mirror is strongly quasipositive \cite{bs-lspace}.  Conjecturally, being an instanton L-space is the same as being an L-space in Heegaard Floer homology; if true, this conjecture would further imply that all of the nonzero coefficients of $\Delta_K(t)$ are $\pm 1$ \cite{osz-lens}.
\end{remark}

\begin{proof}[Proof of Theorem~\ref{thm:l-space-surgeries}]
Assume for now that $r(K) > 0$, and write $r(K) = \frac{p}{q}$ for coprime positive integers $p$ and $q$.  Theorem~\ref{thm:finite-lines-converse} provides an integer $N > 0$ such that for all but finitely many $k>0$, $r(K) - \frac{1}{kN}$ is an $SU(2)$-cyclic surgery slope for $K$.  (It is also at least $r(K) - 1 > 1$, hence positive.)  In particular, this is true for all but finitely many multiples $k=qi$ of $q$, in which case the slopes in question are
\[ \frac{p}{q} - \frac{1}{qi\cdot N} = \frac{pNi-1}{qNi} \]
for some $i>0$.  The numerator $pNi-1$ is prime for infinitely many values of $i$ by Dirichlet's theorem on primes in arithmetic progressions, so $K$ must have a positive $SU(2)$-cyclic surgery slope of the form $\frac{a}{b} < \frac{p}{q}$, where $a$ is prime.

If $S^3_{a/b}(K)$ is not an instanton L-space, then \cite[Corollary~4.8]{baldwin-sivek} says that either $S^3_{a/b}(K)$ is not $SU(2)$-cyclic or some zero $\zeta$ of $\Delta_K(t^2)$ is an $a$th root of unity.  The former is false by construction, and the latter cannot hold either: the minimal polynomial $\Phi_a(t)=1+t+t^2+\dots+t^{a-1}$ of $\zeta$ also divides $\Delta_K(t^2)$, and so $\Phi_a(1)=a$ must divide $\Delta_K(1^2) = 1$, which is impossible.  Thus $S^3_{a/b}(K)$ is an instanton L-space.

Now if $n = \lfloor\frac{a}{b}\rfloor$, then $S^3_s(K)$ is an instanton L-space for all $s \geq n$ by \cite[Theorem~4.20]{baldwin-sivek}.  Since $n$ is an integer which is strictly less than $r(K)$, we have $n \leq \lceil r(K) \rceil - 1$, and this completes the proof in the case $r(K) > 0$.

If instead $K$ has negative limit slope, then the mirror $\mirror{K}$ satisfies $r(\mirror{K}) = -r(K) > 0$.  The above argument shows that $S^3_s(\mirror{K}) = -S^3_{-s}(K)$ is an instanton L-space whenever
\[ s \geq \lceil -r(K)\rceil-1 = -\lfloor r(K)\rfloor - 1, \]
or equivalently whenever $-s \leq \lfloor r(K)\rfloor + 1$.  Since the property of being an instanton L-space is preserved under orientation reversal, this proves the case $r(K) < 0$.
\end{proof}

It is generally hard to show that a knot does not have any instanton L-space surgeries, but when we can do so it follows by Theorem~\ref{thm:l-space-surgeries} that the knot in question is not $SU(2)$-averse.

\begin{theorem} \label{thm:smoothly-slice}
If a nontrivial knot $K \subset S^3$ is smoothly slice, then it has no nontrivial instanton L-space surgeries.  In particular, $K$ is not $SU(2)$-averse.
\end{theorem}

\begin{proof}
Suppose that $K$ is slice and that $S^3_r(K)$ is an instanton L-space for some rational $r$; by taking mirror images as needed we can assume that $r > 0$.  Then \cite[Theorem~4.20]{baldwin-sivek} says that $S^3_n(K)$ is an instanton L-space for all integers $n \geq r$.

Letting $m \geq 1$ be an integer, we now apply Floer's surgery exact triangle
\[ \dots \to I^\#(S^3) \to I^\#(S^3_m(K)) \to I^\#(S^3_{m+1}(K)) \to \dots \]
as stated in \cite[Section~7.5]{scaduto}, cf.\ also \cite{floer-surgery,braam-donaldson}.  (This sequence involves twisted coefficients, but following the third row of \cite[Figure~1]{scaduto} we can place the nontrivial coefficients in the first term; the twisted group $I^\#(S^3;\mu)$ is isomorphic to $I^\#(S^3)$, since $[\mu]=0$ in $H_1(S^3;\Z/2\Z)$.)  Each of these maps is induced by a 2-handle cobordism, and the cobordism from $S^3$ to $S^3_m(K)$ contains a smoothly embedded 2-sphere $S$ of self-intersection $m \geq 1$, obtained by gluing a slice disk for $K$ to the core of the 2-handle.

Since the sphere $S$ has positive self-intersection, the map $I^\#(S^3) \to I^\#(S^3_m(K))$ necessarily vanishes.  This can be proved directly from the results of \cite{mmr-book}, or indirectly from \cite[Corollary~16.0.2]{mmr-book} by using $S$ to construct a torus of self-intersection at least 2: first, if $m=1$ then we patch two copies of $S$ together to get a sphere $S'$ of self-intersection 4, and then we attach a handle to either $S$ or $S'$ to get the desired torus.  Thus the above triangle splits, and since $I^\#(S^3) \cong \Z$ we have
\[ \rank(I^\#(S^3_{m+1}(K))) = \rank(I^\#(S^3_m(K))) + 1, \]
hence $\rank(I^\#(S^3_m(K))) = m + \big(\rank(I^\#(S^3_1(K)))-1\big)$ for all $m \geq 1$ by induction.

We conclude that $S^3_n(K)$ is an instanton L-space for each integer $n \geq r$ if and only if $S^3_1(K)$ is.  By \cite[Theorem~4.20]{baldwin-sivek}, the latter can only be true if $K$ has Alexander polynomial $t-1+t^{-1}$ or $-t+3-t^{-1}$.  But then $K$ violates the Fox-Milnor condition, which asserts that slice knots satisfy $\Delta_K(t) = f(t)f(t^{-1})$ for some polynomial $f(t)$ with integer coefficients \cite{fox-milnor}, so it cannot be slice and we have a contradiction.
\end{proof}

\begin{proposition} \label{prop:slice-genus-1}
If $K$ is $SU(2)$-averse with smooth slice genus 1, then either $S^3_1(K)$ or $S^3_{-1}(K)$ is an instanton L-space depending on whether $r(K)$ is positive or negative.  Thus $K$ has Seifert genus 1 and its Alexander polynomial is either $t-1+t^{-1}$ or $-t+3-t^{-1}$.
\end{proposition}

\begin{proof}
We arrange for $r(K) > 0$ by replacing $K$ with its mirror as necessary, and then repeat the proof of Theorem~\ref{thm:smoothly-slice}: each $S^3_n(K)$ is an instanton L-space for integers $n \geq r(K)$, and so by the surgery exact triangle and induction, $S^3_1(K)$ must be as well.  The 2-handle cobordisms from $S^3$ to $S^3_m(K)$ now have an embedded torus (rather than sphere) of self-intersection $m \geq 1$, and for $m \geq 2$ we can still apply \cite[Corollary~16.0.2]{mmr-book} to see that the cobordism maps on $I^\#$ vanish.

The only change needed is in the case $m=1$, where we prove the vanishing of the cobordism map $I^\#(X)$ by using the adjunction inequality of \cite{km-gauge2} instead.  This requires some attention, since \cite{km-gauge2} assumes the ambient manifold is closed with $b^+ \geq 3$, but in fact they only really need this to ensure that the Donaldson invariants are well-defined (which is not an issue for cobordism maps) and to handle the boundary case where $X$ contains an essential sphere of self-intersection zero.  In our situation the same argument applies.  Given a torus $T$ with self-intersection 1, Kronheimer and Mrowka explain in \cite[\S6 (ii)]{km-gauge2} how to blow up $X$ with exceptional divisor $E$ and then construct for all sufficiently large, even $d$ a surface $\Sigma_d \subset X\#\overline{\mathbb{CP}}^2$ homologous to $d[T]-[E]$ with odd genus and 
\begin{align*}
\Sigma_d\cdot\Sigma_d &= d^2 - 1, &
\frac{\Sigma_d\cdot\Sigma_d}{2} + 2 &< 2g(\Sigma_d)-2.
\end{align*}
Since $g(\Sigma_d)$ is odd, the self-intersection $n_d = \Sigma_d\cdot\Sigma_d$ is positive and not a multiple of $4$, and $\frac{1}{2}n_d + 2 < 2g(\Sigma_d)-2$, we can apply \cite[Proposition~6.5]{km-gauge2} to see that the polynomial invariants associated to the cobordism $X\#\overline{\mathbb{CP}}^2$ vanish on the orthogonal complement of each class $d[T]-[E]$, hence they are identically zero; and then the same is true of the polynomial invariants for $X$, which include the map $I^\#(X)$ as a special case.

Thus if $r(K) > 0$ then $S^3_1(K)$ is an instanton L-space, and if $r(K) < 0$ then the same is true of $S^3_1(\mirror{K}) = -S^3_{-1}(K)$, hence also of $S^3_{-1}(K)$.  We now apply \cite[Theorem~4.20]{baldwin-sivek} to $K$ or $\mirror{K}$ to draw the desired conclusions about its genus and Alexander polynomial.
\end{proof}

After the original version of this paper appeared, Baldwin and the first author \cite{bs-lspace} proved that the trefoils are the only nontrivial knots of slice genus at most 1 which admit instanton L-space surgeries.  In particular, this implies that the trefoils are the only $SU(2)$-averse knots with slice genus at most 1.  We remark that the proof used an adjunction inequality for cobordisms asserting that if $b_1(X)=0$ and $X$ contains a surface $\Sigma$ with $\Sigma\cdot\Sigma > 2g(\Sigma)-2 \geq 0$, then the cobordism map $I^\#(X)$ is zero, analogous to the main result of \cite{km-gauge2}; this generalizes the vanishing results used to prove Theorem~\ref{thm:smoothly-slice} and Proposition~\ref{prop:slice-genus-1}.

\begin{example} \label{ex:twist-knots-redux}
We consider the twist knots $K_n$ of Example~\ref{ex:twist-knots}, assuming $n \in \Z$ without loss of generality.  These all have Seifert genus and hence smooth slice genus at most 1, and Alexander polynomial $\Delta_{K_n}(t) = -nt + (2n+1) - nt^{-1}$, so by Proposition~\ref{prop:slice-genus-1} these cannot possibly be $SU(2)$-averse except when $n=-1,0,1$.  The case $n=1$ is the figure eight, which cannot be $SU(2)$-averse because it is amphichiral: otherwise we would have $r(K_1) = r(\mirror{K_1}) = -r(K_1)$ and hence $r(K_1)=0$, which is impossible.  Thus we see once again that no hyperbolic twist knot is $SU(2)$-averse.
\end{example}

\section{Satellites and cabling} \label{sec:satellites}

Let $P \subset S^1 \times D^2$ be an oriented knot in the solid torus which does not lie in any embedded 3-ball, and let $K \subset S^3$ be a knot with tubular neighborhood $N(K)$ and peripheral curves $\mu,\lambda \in \partial N(K)$.  Then we can take a homeomorphism
\[ \varphi: S^1 \times D^2 \xrightarrow{\sim} N(K) \subset S^3 \]
which sends $\{\pt\} \times \partial D^2$ to $\mu$ and $S^1 \times \{\pt\} \subset S^1 \times \partial D^2$ to $\lambda$, and we define the \emph{satellite} $P(K) \subset S^3$ to be the image $\varphi(P)$.  The knots $P$ and $K$ are called the \emph{pattern} and the \emph{companion}, respectively, and $P$ is \emph{nontrivial} if it is not isotopic to the core of $S^1 \times D^2$.  The \emph{winding number} of $P$ is the unique integer $w$ such that $P$ represents an element $w[S^1 \times \{\pt\}]$ of $H^1(S^1 \times D^2)$.  From now on $U$ will denote the unknot.

The following is a result of Silver and Whitten, stated here in a slightly stronger form.
\begin{proposition}[{\cite[Proposition~3.4]{silver-whitten}}] \label{prop:satellite-quotient}
Given any pattern $P$ and companion $K$, there is a surjection
\[ \psi: \pi_1(S^3 \ssm P(K)) \to \pi_1(S^3 \ssm P(U)) \]
which preserves meridians and longitudes, i.e.\ such that $\psi(\mu_{P(K)}) = \mu_{P(U)}$ and $\psi(\lambda_{P(K)}) = \psi(\lambda_{P(U)})$.
\end{proposition}

\begin{proof}
We can write the knot group of $P(K)$ as an amalgamated free product
\[ \pi_1(S^3 \ssm P(K)) = \pi_1(S^3 \ssm K) \ast_{\Z^2} \pi_1((S^1 \times D^2) \ssm P), \]
where generators of $\Z^2 = \pi_1(\partial N(K))$ identify a meridian and longitude of $K$ with the curves $\{\pt\} \times \partial D^2$ and $S^1 \times \{\pt\}$ on $S^1 \times \partial D^2$ respectively.  Silver and Whitten argue that taking the quotient of this product by the normal closure $\cN$ of the commutator subgroup of $\pi_1(S^3 \ssm K)$ reduces the first factor to $\Z$, generated by $\mu_K \sim (\{\pt\} \times \partial D^2)$, and it sends the element $\lambda_K \sim (S^1 \times \{\pt\})$ to the identity, so the quotient is the knot group of $P(U)$ and $\psi$ is defined as the quotient map.

It is clear from the construction that $\psi(\mu_{P(K)})$ is still a meridian of $P(U)$.  To see the analogous claim for $\lambda_{P(K)}$, we take an embedded surface $\Sigma$ in $S^1 \times D^2$ with boundary $P \sqcup (S^1 \times \underline{w})$, where $\underline{w} \subset \partial D^2$ is a set of oriented points, and let $\lambda_P = \Sigma \cap \partial N(P) \subset (S^1\times D^2) \ssm N(P)$.  Then $\varphi(\lambda_P)$ is a longitude of any satellite $P(K)$, since the image $\varphi(S^1 \times \underline{w})$ bounds a collection of parallel Seifert surfaces for $K$.  When we identify $\pi_1(S^3 \ssm P(U))$ as the quotient
\[ \pi_1\big(S^3 \ssm P(K)\big) / \cN \cong \pi_1\big((S^1 \times D^2) \ssm P\big) / \langle S^1 \times \{\pt\} \rangle, \]
it follows that both $\lambda_{P(K)}$ and $\lambda_{P(U)}$ are represented by the same curve $\lambda_P \subset (S^1 \times D^2) \ssm N(P)$, and so $\psi(\lambda_{P(K)}) = \lambda_{P(U)}$.
\end{proof}

The following are straightforward consequences of Proposition~\ref{prop:satellite-quotient}.

\begin{proposition} \label{prop:satellite-properties}
Let $P \subset S^1 \times D^2$ be a pattern and $K \subset S^3$ any knot.
\begin{enumerate}
\item \label{i:image-pk} The image of $\ch^*(P(U))$ in the pillowcase is a subset of the image of $\ch^*(P(K))$.
\item \label{i:slopes-pk} The $SU(2)$-cyclic surgery slopes for $P(K)$ are all $SU(2)$-cyclic slopes for $P(U)$.
\item \label{i:limit-pk} If $P(K)$ is $SU(2)$-averse and $P(U)$ is not the unknot, then $r(P(K)) = r(P(U))$.
\end{enumerate}
\end{proposition}

\begin{proof}
\eqref{i:image-pk} If an irreducible $\rho: \pi_1(S^3 \ssm P(U)) \to SU(2)$ has pillowcase coordinates $(\alpha,\beta)$, then so does $\tilde\rho = \rho\circ\psi: \pi_1(S^3 \ssm P(K)) \to SU(2)$, since $\psi(\mu_{P(K)}) = \mu_{P(U)}$ and $\psi(\lambda_{P(K)}) = \lambda_{P(U)}$.  We have $\img\tilde\rho = \img\rho$ since $\psi$ is surjective, so $\tilde\rho$ is irreducible as well.

\eqref{i:slopes-pk} If $\frac{m}{n}$-surgery on $P(K)$ is $SU(2)$-cyclic then the line $m\alpha+n\beta\equiv 0\pmod{2\pi}$ is disjoint from the image of $\ch^*(P(K))$, and hence also from the image of $\ch^*(P(U))$ by \eqref{i:image-pk}.

\eqref{i:limit-pk} The set of $SU(2)$-cyclic slopes for $P(K)$ has limit point $r(P(K))$ by Theorem~\ref{thm:open-pillowcase-image}, hence so does the set of $SU(2)$-cyclic slopes for $P(U)$ by \eqref{i:slopes-pk}.  If $P(U)$ is not the unknot, then this limit point must therefore be $r(P(U))$, again by Theorem~\ref{thm:open-pillowcase-image}.
\end{proof}

We can also prove more about satellites with nonzero winding number, as follows.

\begin{proposition} \label{prop:satellites-nonzero-winding}
Let $P$ be a satellite with winding number $w \neq 0$.  If some $r$-surgery on $P(K)$ is $SU(2)$-cyclic, then so is $\frac{r}{w^2}$-surgery on $K$.  In particular, if $K$ is nontrivial and $P(K)$ is $SU(2)$-averse, then $K$ is also $SU(2)$-averse, and $r(P(K)) = w^2r(K)$.
\end{proposition}

\begin{proof}
We will use the standard identification
\[ S^3 \ssm P(K) = \big(S^3 \ssm N(K)\big) \cup \big((S^1 \times D^2) \ssm P\big) \]
to construct representations $\rho \in R(P(K))$ whose restrictions to $S^3 \ssm N(K)$ are irreducible (so that $\rho$ is also irreducible), but whose restrictions to $(S^1 \times D^2) \ssm P$ are abelian.  
Let $\mu_P, \lambda_P$ denote peripheral curves for $P$ in $\partial N(P) \subset S^1 \times D^2$, so that $\mu_P$ bounds a meridian of $P$ and the image $\varphi(\lambda_P) \subset \partial(S^3 \ssm N(P(K)))$ is a longitude of $P(K)$.  We note that $[\lambda_P] = w[S^1 \times \{\pt\}]$ as elements of $H_1((S^1 \times D^2) \ssm N(P))$.

Let $\rho_K: \pi_1(S^3 \ssm N(K)) \to SU(2)$ be an irreducible representation, where up to conjugacy we can write
\[ \rho_K(\mu_K) = \twomatrix{e^{i\alpha}}{0}{0}{e^{-i\alpha}}, \qquad \rho_K(\lambda_K) = \twomatrix{e^{i\beta}}{0}{0}{e^{-i\beta}} \]
for some pair $(\alpha,\beta)$.  By the Mayer-Vietoris sequence we have $H^1((S^1\times D^2) \ssm N(P)) \cong \Z^2$, with generators $[\mu_P]$ and $[S^1 \times \{\pt\}]$ where $\pt \in \partial D^2$.  We know that $[\lambda_P] = w[S^1 \times \{\pt\}]$, and similarly by examining a disk $\{\pt\} \times D^2$ which intersects $P$ transversely in $w$ points (with sign) we see that $w[\mu_P] = [\{\pt\} \times \partial D^2]$.  So if we define an abelian representation
\[ \rho_P: \pi_1\big((S^1\times D^2) \ssm N(P)) \xrightarrow{\ab} H_1\big((S^1\times D^2) \ssm N(P)\big) \to SU(2) \]
by the formulas
\[ \rho_P(\mu_P) = \twomatrix{e^{i(\alpha+2\pi k)/w}}{0}{0}{e^{-i(\alpha+2\pi k)/w}}, \qquad \rho_P(S^1 \times \{\pt\}) = \twomatrix{e^{i\beta}}{0}{0}{e^{-i\beta}} \]
for some integer $k$, then we have $\rho_P(\{\pt\} \times \partial D^2) = \rho_K(\mu_K)$ and $\rho_P(S^1 \times \{\pt\}) = \rho_K(\lambda_K)$.  We glue $\rho_K$ and $\rho_P$ to get the desired representation $\rho \in R^*(P(K))$, with peripheral curves $\mu = \varphi(\mu_P)$ and $\lambda = \varphi(\lambda_P)$ satisfying
\[ \rho(\mu) = \twomatrix{e^{i(\alpha+2\pi k)/w}}{0}{0}{e^{-i(\alpha+2\pi k)/w}}, \qquad \rho(\lambda) = \twomatrix{e^{i\beta w}}{0}{0}{e^{-i\beta w}}, \]
and it follows that $\rho$ has image $(\frac{\alpha+2\pi k}{w}, \beta w \pmod{2\pi})$ in the pillowcase.

Now suppose that some $\frac{p}{q}$-surgery on $P(K)$ is $SU(2)$-cyclic, where $p$ and $q$ are relatively prime.  Let $d = \gcd(p, w^2)$, so that $\frac{p}{d}$ and $\frac{qw^2}{d}$ are also relatively prime.  If the $\frac{p/d}{qw^2/d}$-surgery on $K$ is not $SU(2)$-cyclic, then there is some irreducible $\rho_K \in R^*(K)$ with pillowcase image $(\alpha,\beta)$ satisfying
\[ \frac{p}{d}\alpha + \frac{qw^2}{d}\beta = 2\pi n \]
for some integer $n$.  The corresponding $\rho \in R^*(P(K))$ has image $(\frac{\alpha + 2\pi k}{w}, \beta w)$, satisfying
\[ p\left(\frac{\alpha + 2\pi k}{w}\right) + q\left(\beta w\right) = \frac{2\pi d}{w} \left(n+k(p/d)\right). \]
But $\frac{p}{d}$ is invertible modulo $\frac{w^2}{d}$ since these are coprime integers, so we can choose $k$ so that $n+k(\frac{p}{d})$ is a multiple of $\frac{w^2}{d}$, and then the right side is a multiple of $2\pi w$.  The representation $\rho$ then descends to an irreducible representation of the $\frac{p}{q}$-surgery on $P(K)$, which is impossible, so we conclude that $\frac{p/d}{qw^2/d} = \frac{p}{qw^2}$ is an $SU(2)$-cyclic surgery slope for $K$ after all.
\end{proof}

\subsection{Connected sums}

We can now show that all $SU(2)$-averse knots are prime.

\begin{theorem} \label{thm:averse-connected-sum}
If $K_1$ and $K_2$ are nontrivial knots, then $K_1 \# K_2$ is not $SU(2)$-averse.
\end{theorem}

\begin{proof}
We know that $K_1\# K_2$ is a satellite $P(K_2)$, where $P(U) = K_1$.  Indeed, we can remove a small open neighborhood of a meridian $m \subset \partial N(K_1)$ from $S^3$ to get a solid torus $S^1 \times D^2$ with embedded knot $P = K_1$, and this is the desired pattern.  Similarly, there is a pattern $P'$ such that $P'(K_1) = K_1 \# K_2$ and $P'(U) = K_2$.  Proposition~\ref{prop:satellite-properties} applied to $P$ and to $P'$ now says that if $K_1 \# K_2$ is $SU(2)$-averse, then each summand must be $SU(2)$-averse as well and
\[ r(K_1) = r(K_1 \# K_2) = r(K_2). \]
We let $r$ denote the common limit slope, recalling that $r \neq 0$.

The knot group of $K_1 \# K_2$ is an amalgamated free product
\[ \pi_1(S^3 \ssm K_1) \ast_\Z \pi_1(S^3 \ssm K_2) \]
in which we identify the meridians $\mu_1 \sim \mu_2$, with peripheral data $\mu_\# = \mu_1 = \mu_2$ and $\lambda_\# = \lambda_1\lambda_2$.  Since both $K_1$ and $K_2$ are $SU(2)$-averse, they satisfy case~\eqref{i:arc-near-pi/2} of Theorem~\ref{thm:pillowcase-near-pi/2}, so the pillowcase images of $\ch^*(K_1)$ and $\ch^*(K_2)$ contain arcs of the form
\[ \beta = -r\alpha + c_i, \qquad \frac{\pi}{2} \leq \alpha < \frac{\pi}{2} + \epsilon \]
for some constants $c_1$ and $c_2$ respectively and some $\epsilon > 0$.  For each $\alpha$ in this interval, we let $\rho_\alpha^j: \pi_1(S^3 \ssm K_j) \to SU(2)$, $j=1,2$, be corresponding representations such that
\[ \rho_\alpha^j(\mu_j) = \twomatrix{e^{i\alpha}}{0}{0}{e^{-i\alpha}}, \qquad
\rho_\alpha^j(\lambda_j) = \twomatrix{e^{i(-r\alpha+c_j)}}{0}{0}{e^{-i(-r\alpha+c_j)}}. \]
Combining these produces a representation $\rho^\#_\alpha = \rho^1_\alpha \ast \rho^2_\alpha$ of $\pi_1(S^3 \ssm (K_1\#K_2))$, with
\[ \rho^\#_\alpha(\mu_\#) = \twomatrix{e^{i\alpha}}{0}{0}{e^{-i\alpha}}, \qquad
\rho^\#_\alpha(\lambda_\#) = \twomatrix{e^{i(-2r\alpha+c_1+c_2)}}{0}{0}{e^{-i(-2r\alpha+c_1+c_2)}}. \]
But the $\rho^\#_\alpha$ produce a line segment of slope $-2r$ in the pillowcase image of $\ch^*(K_1 \# K_2)$, so by Theorem~\ref{thm:finitely-many-lines} we must actually have $r(K_1 \# K_2) = 2r$, and since $r \neq 2r$ this is impossible.
\end{proof}

\subsection{Cables}

If $p$ and $q$ are relatively prime integers with $q \geq 2$, and $K$ is an arbitrary knot, we define the {$(p,q)$-cable} $C_{p,q}(K)$ to be a peripheral curve in $\partial N(K)$ in the class $\mu^p\lambda^q$.  This is clearly a satellite with winding number $q$, and if $|p| \geq 2$ then $C_{p,q}(U)$ is the $(p,q)$-torus knot.

\begin{proposition} \label{prop:cable-converse}
If $K$ is a nontrivial, $SU(2)$-averse knot with limit slope $r(K) = \frac{p}{q}$, where $p$ and $q$ are relatively prime and $q \geq 2$, then the $(p,q)$-cable $C = C_{p,q}(K)$ is $SU(2)$-averse with limit slope $r(C) = pq$.
\end{proposition}

\begin{proof}
According to Theorem~\ref{thm:finite-lines-converse}, there is an integer $N \geq 1$ such that $(\frac{p}{q}+\frac{1}{kN})$-surgery on $K$ is $SU(2)$-averse for all but finitely many $k$.  Gordon \cite[Corollary~7.3]{gordon} showed that
\[ S^3_{pq+\frac{1}{n}}(C_{p,q}(K)) \cong S^3_{\frac{p}{q}+\frac{1}{q^2n}}(K) \]
for all $n \geq 1$, and the latter is $SU(2)$-cyclic for all but finitely many multiples $n$ of $N$, so $C_{p,q}(K)$ has an infinite sequence of $SU(2)$-cyclic slopes which converges to $pq$.
\end{proof}

\begin{theorem} \label{thm:cables}
Let $K$ be a nontrivial knot, and let $p$ and $q$ be relatively prime with $q \geq 2$.  The cable $C_{p,q}(K)$ is $SU(2)$-averse if and only if $K$ is also $SU(2)$-averse with limit slope $r(K) = \frac{p}{q}$, and in this case $r(C_{p,q}(K)) = pq$.
\end{theorem}

We begin by proving the hardest case of Theorem~\ref{thm:cables}, namely $p=\pm1$.

\begin{lemma} \label{lem:cables-p-1}
Let $K$ be a nontrivial knot and $q \geq 2$.  Then the pillowcase image of $\ch^*(C_{\pm1,q}(K))$ contains a nontrivial line segment of slope $\mp q$, and the cable $C_{\pm1,q}(K)$ is not $SU(2)$-averse.
\end{lemma}

\begin{proof}
We begin with some generalities about arbitrary cables.  Many of the details below are adapted from the proof of \cite[Theorem~2.8]{ni-zhang}.  Letting $C=C_{p,q}$ denote the corresponding pattern in $S^1 \times D^2$, the complement $(S^1 \times D^2) \ssm C$ is Seifert fibered over the annulus with one singular fiber of order $q$, where the generic fibers on $S^1 \times \partial D^2$ have slope
\[ p[\mu_K] + q[\lambda_K] = p[\{\pt\} \times \partial D^2] + q[S^1 \times \{\pt\}], \]
and those on $\partial N(C)$ have slope $pq[\mu_C] + [\lambda_C]$; see \cite[Lemma~7.2]{gordon}.  These are central in $\pi_1((S^1\times D^2) \ssm C)$, so any representation
\[ \rho: \pi_1(S^3 \ssm C_{p,q}(K)) \to SU(2) \]
whose restriction to $\pi_1((S^1 \times D^2) \ssm C)$ is irreducible must send these to the center of $SU(2)$, meaning that $\rho(\mu_C^{pq}\lambda_C)$ is either $I$ or $-I$.

We now assume that $p=\pm 1$ and let $\rho^+_K: \pi_1(S^3 \ssm K) \to SU(2)$ be an irreducible representation satisfying $\rho^+_K(\mu_K^{\pm 1}\lambda_K^q) = I$; this exists since $\pm\frac{1}{q}$-surgery on $K$ is not $SU(2)$-cyclic, by \cite{km-su2}.  Then $\rho^+_K(\mu_K) \not\in \{I,-I\}$, or else $\rho^+_K$ would be reducible, and since $\rho^+_K(\lambda_K^q) \not\in \{I,-I\}$ as a result, the same is true of $\rho^+_K(\lambda_K)$.  Likewise, we can multiply $\rho^+_K$ by the central character
\[ \chi: \pi_1(S^3\setminus K) \to H_1(S^3\setminus K) \cong \Z \xrightarrow{n \mapsto (-1)^n} \{\pm 1\} \]
to get an irreducible $\rho^-_K$ with $\rho^-_K(\mu_K^{\pm 1}\lambda_K^q) = -I$ and $\rho^-_K(\mu_K), \rho^-_K(\lambda_K) \not\in \{I,-I\}$.

We will use either $\rho^+_K$ or $\rho^-_K$ to construct a family of representations
\[ \rho_s: \pi_1(S^3 \ssm C_{\pm 1,q}(K)) \to SU(2) \]
which are irreducible on $(S^1 \times D^2) \ssm C$, hence satisfy $\rho_s(\mu_C^{\pm q}\lambda_C) \in \{I,-I\}$, but for which the trace of $\rho_s(\mu_C)$ takes values everywhere in an open interval.  This produces a line segment of slope $\mp q$ in the pillowcase image of $\ch^*(C_{\pm 1,q}(K))$, which implies that if $C_{\pm 1,q}(K)$ is $SU(2)$-averse then it has limit slope $\pm q$.  But in this case $K$ must also be $SU(2)$-averse with limit slope $r(K) = \frac{1}{q^2}r(C_{\pm 1,q}(K)) = \pm \frac{1}{q}$ by Proposition~\ref{prop:satellites-nonzero-winding}, and this contradicts $|r(K)| > 2$, so the proof will be complete.

Taking $C = C_{\pm 1,q} \subset S^1\times D^2$, we have a presentation
\[ \pi_1((S^1 \times D^2) \ssm C) = \langle h, \lambda_K \mid h^q\lambda_K = \lambda_K h^q \rangle \]
with $\mu_C = h\lambda_K$, as deduced in \cite{ni-zhang}; here $h$ is freely homotopic to the singular fiber of $(S^1 \times D^2) \ssm C$.  (Note that we are using $K$ to refer to the companion knot and $C$ to refer to the cabling pattern, in contrast with \cite{ni-zhang}.)  Thus we can extend either $\rho^+_K$ or $\rho^-_K$ to a representation of $\pi_1(S^3 \ssm C_{\pm 1,q}(K))$ by finding an element of $SU(2)$ whose $q$th power commutes with $\rho_K(\lambda_K)$; this will provide an extension of $\rho^+_K$ if the Seifert fiber $\mu_C^{\pm q}\lambda_C \simeq \mu_K^{\pm1}\lambda_K^q$ is sent to $+I$, and of $\rho^-_K$ if it is sent to $-I$ instead.

Viewing $SU(2)$ as the unit quaternions for convenience, we can write $\rho_K(\lambda_K) = e^{i\beta}$ up to conjugacy for some $\beta$, noting that $\sin(\beta) \neq 0$ since $e^{i\beta} \not\in \{1,-1\}$.  We define a family of purely imaginary unit quaternions by 
\[ v_s = \cos(s)i + \sin(s)j, \qquad 0 \leq s \leq \pi. \]
Then for all $s$ in the same range we extend $\rho_K$ to $\rho_s: \pi_1(S^3 \ssm C_{\pm 1,q}(K)) \to SU(2)$ with pillowcase coordinates $(\alpha_s,\beta_s)$ by setting
\[ \rho_s(h) = \cos\left(\frac{\pi}{q}\right) + \sin\left(\frac{\pi}{q}\right)v_s. \]
It is clear that $\rho_s(h^q) = -1$, which commutes with $\rho_s(\lambda_K) = e^{i\beta}$, and so these representations are well-defined.  For $0<s<\pi$ they restrict to irreducible representations on $\pi_1((S^1\times D^2)\ssm C)$, since $\rho_s(h)$ does not commute with $e^{i\beta}$.  From $\rho_s(\mu_K) = \rho_s(h)\rho_s(\lambda_K)$ we also compute that
\[ \cos(\alpha_s) = \re(\rho_s(\mu_K)) = \cos\left(\frac{\pi}{q}\right)\cos(\beta) - \cos(s)\sin\left(\frac{\pi}{q}\right)\sin(\beta). \]
Since $\sin(\frac{\pi}{q})\sin(\beta) \neq 0$, the value of $\cos(\alpha_s)$ and hence of $\alpha_s$ covers an open interval of values for $0 < s < \pi$, as desired.
\end{proof}

\begin{proof}[Proof of Theorem~\ref{thm:cables}]
One direction was already established in Proposition~\ref{prop:cable-converse}, so we will suppose that $C_{p,q}(K)$ is $SU(2)$-averse.  By Proposition~\ref{prop:satellites-nonzero-winding}, $K$ is $SU(2)$-averse with $r(C_{p,q}(K)) = q^2 r(K)$.  Lemma~\ref{lem:cables-p-1} says that $p \neq \pm 1$, so then $C_{p,q}(U) = T_{p,q}$ is not the unknot.  But then $r(C_{p,q}(K)) = r(T_{p,q}) = pq$ by Proposition~\ref{prop:satellite-properties} and we conclude that $r(K) = \frac{p}{q}$.
\end{proof}

Theorem~\ref{thm:cables} says that if any cable at all of $K$ is $SU(2)$-averse then exactly one cable $C_{p,q}(K)$ is, with limit slope $r(C_{p,q}(K)) = pq$; and in this case, not only is $K$ also $SU(2)$-averse but the limit slope $r(K) = \frac{p}{q}$ cannot be an integer.  Since $SU(2)$-averse small knots (which include torus knots) and cables must have integral limit slopes, we conclude the following.

\begin{corollary}
Cables of small knots and of nontrivial cables are never $SU(2)$-averse.
\end{corollary}

We have the following special case, since algebraic knots are iterated cables of the unknot.

\begin{corollary} \label{cor:algebraic-knots}
An algebraic knot is $SU(2)$-averse if and only if it is a torus knot.
\end{corollary}

\subsection{Whitehead doubles} \label{sec:whitehead-doubles}

Using what we know so far about satellites, we cannot say much about whether the Whitehead double $\Wh(K)$ of a nontrivial knot is $SU(2)$-averse: neither Proposition~\ref{prop:satellite-properties} nor Proposition~\ref{prop:satellites-nonzero-winding} can be used here, since $\Wh(U)$ is the unknot and the corresponding pattern knot has winding number zero.   On the other hand, we can completely answer Conjecture~\ref{conj:main-conjecture} for Whitehead doubles by other means.

\begin{proposition}
If $K$ is a nontrivial knot, then its Whitehead double $\Wh(K)$ is not $SU(2)$-averse.
\end{proposition}

\begin{proof}
Since $\Wh(K)$ has Seifert genus 1, its smooth slice genus is at most 1.  If it is smoothly slice, then Theorem~\ref{thm:smoothly-slice} says that it is not $SU(2)$-averse.  Otherwise its slice genus is exactly 1, and then since its Alexander polynomial is 1 we know by Proposition~\ref{prop:slice-genus-1} that it is not $SU(2)$-averse.
\end{proof}

\section{Montesinos knots and knots with small crossing number} \label{sec:evidence}

In this section we provide evidence for Conjecture~\ref{conj:main-conjecture}, that the only $SU(2)$-averse knots are the torus knots, by proving it for several classes of examples.  Our main tools will be: \begin{itemize}
\item Theorem~\ref{thm:det-limit-slope} for small knots, and in particular the inequality $\det(K) \leq |r(K)|-1$;
\item restrictions on the roots of the Alexander polynomial from Proposition~\ref{prop:alexander-root-averse}; and
\item Theorem~\ref{thm:smoothly-slice} and Proposition~\ref{prop:slice-genus-1}, which assert that an $SU(2)$-averse knot cannot be smoothly slice, and that if it has slice genus 1 then its Seifert genus is also 1.
\end{itemize}
We have obtained data about knots with low crossing number from the Knot Atlas \cite{knotatlas} and KnotInfo \cite{knotinfo}, and we rely on SnapPy \cite{snappy} for a computation.  In general, we only discuss each knot with the chirality specified in these tables, but if $K$ is not $SU(2)$-averse then neither is its mirror $\mirror{K}$.

First, we will use Theorem~\ref{thm:det-limit-slope} together with bounds on the boundary slopes of two-bridge knots, and of Montesinos knots with three rational tangles, in terms of the crossing number $c(K)$ to verify Conjecture~\ref{conj:main-conjecture} for these knots.

\begin{theorem} \label{thm:2-bridge-averse}
Let $K$ be an $SU(2)$-averse knot which is an alternating Montesinos knot with at most three rational tangles.  Then $K$ is a $(2,2n+1)$-torus knot.
\end{theorem}

We will prove Theorem~\ref{thm:2-bridge-averse} by showing that the limit slope of such a knot is at most twice the crossing number, which is in turn less than the determinant in almost all cases, and this is impossible for small knots.  We begin with the following slight refinement of a theorem of Crowell \cite[Theorem~6.5]{crowell}.

\begin{proposition} \label{prop:alternating-det-cr}
If $K$ is a prime alternating knot, then $\det(K) \geq 3c(K)-8$ unless $K$ is a $(2,2k+1)$ torus knot or a twist knot.
\end{proposition}

\begin{proof}
Crowell \cite[Theorem~6.5]{crowell} proved that any prime alternating knot $K$ satisfies $\det(K) \geq 2c(K)-3$ unless $K$ is of ``elementary torus type''; in the latter case, its Alexander polynomial has the form
\[ \Delta_K(t) \doteq 1 - t + t^2 - \dots + (-t)^n \]
\cite[(6.6)]{crowell} for some $n$, which guarantees by \cite[Proposition~4.1]{osz-lens} that it is a $(2,n+1)$ torus knot.  We assume from now on that $K$ is not a torus knot.

Let $B$ be the black graph associated to a checkerboard coloring of a diagram of $K$ with $c(K)$ crossings.  We let $b$ and $w$ denote the number of vertices and faces of $B$, or equivalently the number of faces and vertices of the white graph.  Then $b,w \geq 3$ by \cite[(6.3)]{crowell}, and since $B$ is planar with $c(K)$ edges, we have $b+w=c(K)+2$.

The proof of \cite[Theorem~6.5]{crowell} combines the relation
\[ (b-1)(w-1) = 2(c(K)-2) + (b-3)(w-3), \]
which is equivalent to the relation $b+w=c(K)+2$, with the inequality
\begin{equation} \label{eq:det-tree-inequality}
\det(K)-1 \geq (b-1)(w-1)
\end{equation}
of \cite[Theorem~5.10]{crowell}.  We note that $b+w=c(K)+2$ is also equivalent to
\[ (b-3)(w-3) = (b-4)(w-4) + (c(K)-5), \]
so given that $b,w \geq 3$, if the left side is nonzero then $b,w \geq 4$, and it follows that $(b-3)(w-3)$ is either zero or at least $c(K)-5$.  Thus either
\begin{align*}
\det(K) &\geq (b-1)(w-1) + 1 \\
&= 2c(K)-3 + (b-3)(w-3) \\
&\geq 3c(K)-8
\end{align*}
or one of $b-3$ and $w-3$ is zero.

In the remaining cases, we take $b=3$ without loss of generality; in this case $B$ has three vertices, $c(K)$ edges, and $w=c(K)-1$ faces.  Crowell deduces the inequality \eqref{eq:det-tree-inequality} from the fact that for alternating knots, $\det(K)$ is the number of spanning trees of $B$, so we must determine the cases in which (given $b=3$ and $w=c(K)-1$) the graph $B$ has at most $3c(K)-9$ spanning trees.

Label the vertices of $B$ as $v_1,v_2,v_3$, and let $e_i$ denote the number of edges between $v_j$ and $v_k$ whenever $\{i,j,k\} = \{1,2,3\}$; note that $e_i \geq 1$ for all $i$, since there are no vertices whose removal from $B$ disconnects it (see \cite[(5.3)~and~(5.4)]{crowell}).  There are $e_1e_2 + e_2e_3 + e_3e_1$ spanning trees and $c(K) = e_1+e_2+e_3$, so we wish to solve the inequality
\[ e_1e_2+e_2e_3+e_3e_1 \leq 3(e_1+e_2+e_3) - 9. \]
Reordering the $v_i$ if needed so that $e_1 \leq e_2 \leq e_3$, we rearrange the inequality to get
\[ (e_1+e_2-3)e_3 + (e_1-3)(e_2-3) \leq 0. \]
If $e_1+e_2 \geq 3$, then it follows that $e_2 \geq 2$ (hence $2e_2-3 > 0$) and that 
\[ (e_1+e_2-3)e_2 + (e_1-3)(e_2-3) \leq 0, \]
or equivalently $e_2^2 \leq (3-e_1)(2e_2-3)$.  But this implies that $e_2^2 \leq 2(2e_2-3)$, which has no solutions, so we must therefore have $e_1+e_2 < 3$, which means that $e_1=e_2=1$.

We conclude that up to permuting the coordinates, the possible values of $(e_1,e_2,e_3)$ are $(1,1,a)$ for some $a \geq 4$.  In each of these cases, the corresponding black graph can only come from a twist knot, completing the proof.
\end{proof}

\begin{corollary} \label{cor:alternating-det-cr-2}
If $K$ is a prime alternating knot, then $\det(K) > 2c(K)$ unless $K$ is a $(2,2k+1)$ torus knot, a twist knot, $6_2$, or $7_3$.
\end{corollary}

\begin{proof}
We inspect the knot types with up to 7 crossings by hand to see that the knots with $\det(K) \leq 2c(K)$ are either torus knots ($3_1$, $5_1$, and $7_1$), twist knots ($4_1$, $5_2$, $6_1$, and $7_2$), or $6_2$ or $7_3$.  Assuming that $c(K) > 7$ and $K$ is not a torus knot or a twist knot, we have $\det(K) \geq 3c(K)-8 > 2c(K)-1$ by Proposition~\ref{prop:alternating-det-cr}, and since $\det(K)$ is always odd we in fact have $\det(K) > 2c(K)$, so there are no further examples.
\end{proof}

\begin{proof}[Proof of Theorem~\ref{thm:2-bridge-averse}]
By hypothesis, the set of finite boundary slopes of $K$ has diameter equal to $2c(K)$; this is due to Mattman, Maybrun, and Robinson \cite{mmr} in the two-bridge case (i.e., at most two rational tangles), and to Ichihara and Mizushima \cite{ichihara-mizushima} for alternating Montesinos knots in general.  Both $0$ and the limit slope $r(K)$ are boundary slopes, so we must have $|r(K)| \leq 2c(K)$.  Two-bridge knots are small \cite{hatcher-thurston}, as are Montesinos knots with at most three rational tangles \cite{oertel}, so Theorem~\ref{thm:det-limit-slope} implies that
\begin{equation} \label{eq:det-cr}
\det(K) \leq |r(K)|-1 \leq 2c(K)-1.
\end{equation}
Theorem~\ref{thm:averse-connected-sum} says that $K$ is prime, so by Corollary~\ref{cor:alternating-det-cr-2}, we see that $K$ is either a $(2,2n+1)$ torus knot, a twist knot, $6_2$, or $7_3$.  But we have already seen in Example~\ref{ex:twist-knots} that the trefoil is the only nontrivial, $SU(2)$-averse twist knot, so this leaves only $6_2$ and $7_3$.  These have Alexander polynomials
\begin{align*}
\Delta_{6_2}(t) &= -t^2 + 3t - 3 + 3t^{-1} - t^{-2} \\
\Delta_{7_3}(t) &= 2t^2 - 3t + 3 - 3t^{-1} + 2t^{-2},
\end{align*}
both of which are irreducible and not cyclotomic; since $\Delta_{6_2}(-1) = -11$ and $\Delta_{7_3}(i) = -1$,  Proposition~\ref{prop:alexander-root-averse} says that neither knot is $SU(2)$-averse.
\end{proof}

We now begin to verify Conjecture~\ref{conj:main-conjecture} for knots with low crossing number.  Nearly all knots through 10 crossings can be handled by elementary obstructions, meaning they do not depend on the results of \cite{bs-lspace}.

\begin{theorem} \label{thm:up-to-9-crossings}
If a knot with crossing number at most 9 is $SU(2)$-averse, then it is a torus knot.
\end{theorem}

\begin{proof}
Most knots with at most 8 crossings are alternating Montesinos knots with at most three tangles, and thus handled by Theorem~\ref{thm:2-bridge-averse}; the exceptions are $8_n$ for $16 \leq n \leq 21$, of which $8_{19}$ is the $(3,4)$ torus knot.  The remaining five knots all have slice genus at most 1 and Seifert genus at least 2, so Theorem~\ref{thm:smoothly-slice} and Proposition~\ref{prop:slice-genus-1} say that they are not $SU(2)$-averse.

Similarly, there are 49 knot types with crossing number 9, and the knot $9_n$ is alternating and Montesinos with at most three rational tangles for $1\leq n \leq 28$ and $n=30,31,35,36,37$, so Theorem~\ref{thm:2-bridge-averse} applies for these values of $n$.  The remaining knots are $9_n$ for $n$ in
\[ 29, 32, 33, 34, 38, 39, 40, 41, 42, 43, 44, 45, 46, 47, 48, 49. \]
Among these knots, all but $9_{38}$, $9_{43}$, and $9_{49}$ are either smoothly slice, or have smooth slice genus 1 but Seifert genus greater than 1, so they cannot be $SU(2)$-averse by Theorem~\ref{thm:smoothly-slice} and Proposition~\ref{prop:slice-genus-1}.

We now use Proposition~\ref{prop:alexander-root-averse} to rule out $9_{38}$, since
\[ \Delta_{9_{38}}(t) = (t-1+t^{-1})(5t-9+5t^{-1}) \]
and the roots of $5t-9+5t^{-1}$ lie on the unit circle but are not roots of unity; and $9_{49}$, since
\[ \Delta_{9_{49}}(t) = 3t^2 - 6t + 7 - 6t^{-1} + 3t^{-2} \]
is irreducible with $\Delta_{9_{49}}(e^{i\pi/3}) = -2$.  This leaves only the small Montesinos knot $9_{43}$: if it were $SU(2)$-averse then $r(9_{43})$ would be one of its boundary slopes, which are
\[ -4, 0, 6, 8, 32/3 \]
by \cite{hatcher-oertel,dunfield-table}, but since $\det(9_{43}) = 13$, it cannot satisfy $\det(9_{43}) < |r(9_{43})|$.
\end{proof}

\begin{theorem} \label{thm:10-crossing}
The only 10-crossing knot which is $SU(2)$-averse is $10_{124} = T(3,5)$, with the possible exception of $10_{98}$.
\end{theorem}

\begin{proof}
The 10-crossing knots are labeled $10_n$ for $n \leq 165$, and $10_n$ is alternating and Montesinos with at most three rational tangles for all $n \leq 78$, so Theorem~\ref{thm:2-bridge-averse} applies to them.  For $79 \leq n \leq 165$, the knot $10_n$ is either smoothly slice, or has slice genus 1 and Seifert genus at least 2, for all $n$ except $124$ (i.e., the $(3,5)$ torus knot) and
\[ \begin{gathered} 80, 85, 92, 98, 100, 101, 111, 120, 127, 128, 134, \\ 139, 142, 145, 149, 150, 152, 154, 157, 160, 161. \end{gathered} \]
Thus Theorem~\ref{thm:smoothly-slice} and Proposition~\ref{prop:slice-genus-1} rule out all of the remaining 10-crossing knots except for these.

Among the remaining 21 knot types $10_n$, most of these are not $SU(2)$-averse by Proposition~\ref{prop:alexander-root-averse}.  The Alexander polynomials of the following knots have no cyclotomic factors, and are negative at the indicated points on the unit circle:
\begin{alignat*}{3}
\Delta_{10_n}(e^{i\pi/4}) &< 0 &: \quad n &= 80, 92, 101, 111, 127, 128, 134, 145, 149, 150, 154, 157, 160, 161 \\
\Delta_{10_n}(e^{i\pi/6}) &< 0 &: \quad n &= 120, 152.
\end{alignat*}
Proposition~\ref{prop:alexander-root-averse} can also be applied to the Alexander polynomials
\begin{align*}
\Delta_{10_{85}}(t) &= (t-1+t^{-1})(t^3-3t^2+4t-3+4t^{-1}-3t^{-2}+t^{-3}), \\
\Delta_{10_{139}}(t) &= (t-1+t^{-1})(t^3-t+1-t^{-1}+t^{-3}), \\
\Delta_{10_{142}}(t) &= (t-1+t^{-1})(2t^2-t-1-t^{-1}+2t^{-2}).
\end{align*}
Indeed, if we call the second factors $f_{85}(t)$, $f_{139}(t)$, and $f_{142}(t)$ respectively, then these are real-valued on the unit circle and equal to 1 at $t=1$, with $f_{85}(-1)=-19$, $f_{139}(e^{i\pi/3})=-2$, and $f_{142}(i)=-5$, so that each $f_n(t)$ and hence $\Delta_{10_n}(t)$ has a simple root on the unit circle which is not a root of unity.

This leaves only the knots $10_{98}$ and $10_{100}$.  Since $10_{100}$ is small, we use SnapPy \cite{snappy} to compute the list of boundary slopes of spun normal surfaces in its complement:
\[ \texttt{\detokenize{Manifold('10_100').normal_boundary_slopes()}} \]
The list includes all nonzero boundary slopes since $10_{100}$ is a hyperbolic knot in $S^3$, and thus it gives us an upper bound of $20$ on the absolute value of a boundary slope for $10_{100}$.  This upper bound is smaller than $\det(10_{100}) = 65$, so Theorem~\ref{thm:det-limit-slope} says that it cannot be $SU(2)$-averse.
\end{proof}

On the other hand, if we appeal to \cite{bs-lspace} in the form of Corollary~\ref{cor:fibered-sqp}, then we can push these results even further with minimal effort.

\begin{theorem} \label{thm:11-crossing}
If $K$ is a prime knot of crossing number at most 11, then $K$ is $SU(2)$-averse if and only if it is a torus knot.
\end{theorem}

\begin{proof}
Corollary~\ref{cor:fibered-sqp} says that $K$ must be fibered, and either $K$ or its mirror is strongly quasipositive.  According to KnotInfo \cite{knotinfo}, the only such prime knots up to 11 crossings are torus knots and
\[ 10_{139},\ 10_{145},\ 10_{152},\ 10_{154},\ 10_{161},\ 11n_{77},\ 11n_{183}. \] 
In the proof of Theorem~\ref{thm:10-crossing}, we used the Alexander polynomial obstruction of Proposition~\ref{prop:alexander-root-averse} to rule out each of the 10-crossing knots listed above.  We can do the same for $11n_{77}$ and $11n_{183}$, whose Alexander polynomials factor as
\begin{align*}
\Delta_{11n_{77}}(t) &= (t-1+t^{-1})^2 (t^2+t-3+t^{-1}+t^{-2}) \\
\Delta_{11n_{183}}(t) &= (t-1+t^{-1}) (t^2+2t-5+2t^{-1}+t^{-2}).
\end{align*}
Indeed, $f_{77}(t) = t^2+t-3+t^{-1}+t^{-2}$ and $f_{183}(t) = t^2+2t-5+2t^{-1}+t^{-2}$ are irreducible and not cyclotomic, and $f_{77}(1) = f_{183}(1) = 1$ while $f_{77}(-1) = -3$ and $f_{183}(-1) = -7$, so both $f_{77}$ and $f_{183}$ have simple roots on the unit circle which are not roots of unity.
\end{proof}

\appendix

\section{$C^r$-approximation through shearing maps} \label{sec:c1-approximation}
\allowdisplaybreaks

In this appendix, we prove Theorem \ref{C1 approximation}, following the strategy of the proof of \cite[Theorem~3.3]{zentner}.

We start by fixing some notational and linguistic conventions which are self-suggesting by the fact that the tangent bundle of the torus is parallelizable. If $X$ is a vector field on $T^2$ we regard it as a map $T^2 \to \R^2$. We consider its derivative as a map $DX:T^2 \times \R^2 \to \R^2$, and if the point $p$ on the torus is fixed we denote by $DX_p$ the corresponding linear map. The expression $\norm{DX}_\infty$ denotes the supremum of the norms $\norm{DX_p}$ as $p$ varies over $T^2$. 

Likewise we regard the flow $\phi^t_{X}$ as a map $T^2 \to \R^2$, and we consider the derivative as a map $D\phi^t_X:T^2 \times \R^2 \to \R^2$. 

By a slight abuse of notation, we shall denote by 
\[
	\norm{\zeta - \xi}_{C^r} := \sup_{(x,y) \in T^2} d(\zeta(x,y),\xi(x,y)) + \sum_{l=1}^{r} \norm{D^{(l)}\zeta- D^{(l)} \xi}_\infty
\]
the $C^r$-distance of the maps $\zeta, \xi: T^2 \to T^2$. 

\subsection*{The Fourier decomposition}
The first point is that Lemma 3.4 in \cite{zentner} generalizes to the $C^l$ norm for any $l\geq 1$ if the vector field $X$ one starts with is differentiable infinitely often. More precisely, for any smooth vector field $X$, and for any $\epsilon > 0$, there is a sum of Fourier vector fields $Z = W_1 + \dots + W_m$ such that we have $\norm{X-Z}_{C^l} < \epsilon$. A Fourier vector field $W_i$ is of the form
\[
	W_i(x,y) = {\bf a}_i \sin({\bf k}_i \cdot (x,y)) + {\bf b}_i \cos({\bf k}_i \cdot (x,y)) \,,
\]
where ${\bf a}_i, {\bf b}_i \in \R^2$, and ${\bf k}_i \in \Z^2$. It is shown in \cite[Lemma 3.4]{zentner} that if the vector field $Z$ is divergence-free, then the flow of the vector fields $W_i$ are isotopies through shearing maps. 

\subsection*{A few lemmata}
In the sequel we denote by $\R_+$ the non-negative real integers. 
The following lemma is a generalization of \cite[Lemma 3.6]{zentner}.

\begin{lemma}\label{le:bound_a}
	Let $r \geq 0$. For any $0 \leq l \leq r$, there are functions 
	\begin{equation*}
		a_l\colon \R_+ \times \R_+ \to \R_+
	\end{equation*}
	which are continuous and monotonically increasing in both variables, such that the following holds.  Let $X\colon T^2 \to \R^2$ be a vector field on $T^2$ with flow $\phi^t_X: T^2 \to T^2$, and suppose that for any $l \leq r+1$ its derivatives satisfy
	\[
		\norm{D^{(i)}X}_\infty \leq K_l \quad \mathrm{for\ all\ } i \leq l
	\] 
	for some positive constants $K_l$.  Then:
\begin{enumerate}
\item For any $p,q\in T^2$, we have
\[ \norm{\phi^t_X(p) - \phi^t_X(q)} \leq a_0(t,K_{1}) \cdot \norm{p-q}. \]
\item For any $p,q\in T^2$ and any $l \geq 1$, we have
\[ \norm{(D^{(l)}\phi^t_X)(p) - (D^{(l)}\phi^t_X)(q)} \leq t \cdot a_l(t,K_{l+1}) \cdot \norm{p-q}. \]
\end{enumerate}
\end{lemma}

\begin{proof}
The case $r=0$ is Lemma 3.6 in \cite{zentner}, with $a_0(t,K_1) = e^{K_1t}$. For $r \geq 1$, the proof will be by induction on $r$. 
 
By definition the flow $\phi^t_X$ satisfies the differential equation 
	\[
		\frac{d \phi^t_X}{dt}(p) = X (\phi^t_X(p)) \, ,
	\]	
for any $p \in T^2$, which integrated gives 
	\begin{equation}\label{eq:integrated flow}
		\phi^t_X (p) = p + \int_{0}^{t} X(\phi^s_X(p)) \, ds \, .
	\end{equation}
Differentiating this equation with respect to $p$ yields
\begin{equation}\label{eq:flow differentiated}
	D\phi^t_X(p) = \id + \int_{0}^{t} DX_{\phi^s_X(p)} \circ D\phi^s_X(p) \, ds \, .
\end{equation}
Therefore we get, using the triangle inequality and the mean value theorem,
\begin{equation*}
\begin{split}
	\norm{D\phi^t_X(p) - D\phi^t_X(q)} 
		 \leq & \int_0^{t} \norm{DX_{\phi^s_X(p)} \circ D\phi^s_X(p) -  DX_{\phi^s_X(q)} \circ D\phi^s_X(q)}\, ds \\
		 \leq & \int_0^{t} \norm{DX_{\phi^s_X(p)} \circ D\phi^s_X(p) -  DX_{\phi^s_X(q)} \circ D\phi^s_X(p)} \, ds \\*
			&\qquad + \int_0^{t} \norm{DX_{\phi^s_X(q)} \circ D\phi^s_X(p) -  DX_{\phi^s_X(q)} \circ D\phi^s_X(q)} \, ds \\
			\leq & \, \norm{D\phi^{s}_X}_{L^\infty(T^2 \times [0,t])} \cdot K_2 \cdot \int_0^{t}\norm{\phi^s_X(p) - \phi^s_X(q)} \, ds \\*
			&\qquad + \norm{DX}_\infty \, \int_0^{t} \norm{D\phi^s_X(p) - D\phi^s_X(q)} \, ds \, . 
\end{split}
\end{equation*}
But \cite[Lemma 3.6]{zentner} states that $\norm{ \phi^s_X(p) - \phi^s_X(q)} \leq e^{K_1 s} \norm{p-q}$. Therefore, we get the inequality
\begin{equation*}
\begin{split}
	\norm{D\phi^t_X(p) - D\phi^t_X(q)} 
		 \leq & \, \norm{D\phi^{s}_X}_{L^\infty(T^2 \times [0,t])} \cdot K_2 \cdot t \cdot \, e^{K_1 t} \, \norm{p-q} \\
		&\qquad + \norm{DX}_\infty \, \int_0^{t} \norm{D\phi^s_X(p) - D\phi^s_X(q)} \, ds \, . 
\end{split}
\end{equation*}
Gronwall's inequality now yields
\begin{equation} \label{eq:estimate flow}
	\norm{D\phi^t_X(p) - D\phi^t_X(q)} \leq t \cdot \norm{p-q} \cdot \norm{D\phi^{s}_X}_{L^\infty(T^2 \times [0,t])} \cdot K_2 \cdot  \exp(t (\norm{DX}_\infty + K_1)) \, 
\end{equation}
for all $p,q \in T^2$, and for all $t \geq 0$. 

In a similar but easier application of Gronwall's inequality, Equation~\eqref{eq:flow differentiated} yields
	\begin{equation}\label{eq:flow_vs_its_vector_field}
		\norm{D\phi^{s}_X}_{L^\infty(T^2 \times [0,t])} \leq \exp(K_1 t) \, .
	\end{equation}
Hence the bound in Equation~\eqref{eq:estimate flow} implies the $l=1$ case of the lemma once we set $a_1(t,K_2):= K_2 \cdot \exp(3K_2 t)$.

We suppose now that the result holds for $r-1$ for some $r\geq 2$. Differentiating Equation~\eqref{eq:flow differentiated}, we get the formula
\begin{equation} \label{eq:flow higher derivatives}
	\begin{split}
	(D^{(r)} \phi^t_X)(p) = &  \int_0^{t} L^{(r-1)}(s) \, ds \, \\
							& +  \int_0^{t} (D X)_{\phi^s_X(p)} \circ (D^{(r)} \phi^s_X)(p) \, ds \, .
	\end{split}
\end{equation}
Here, the term $L^{(r-1)}(s)$ is a polynomial expression of derivatives of $X$ and of $\phi^{s}_X$ of order up to $r-1$. A precise expression could be given, but we don't need this here. The formula for higher derivatives of a composition of functions of one variable is known as Fa\`a di Bruno's formula. 

The claim now follows easily from this, using the same sort of estimates as above, the induction step, and then Gronwall's inequality to obtain an inequality for $\norm{D^{(r)} \phi^s_X}_\infty$, and for the difference $\norm{(D^{(r)} \phi^s_X)(p) - (D^{(r)} \phi^s_X)(q)}$. This will use the fact that sums and products of non-negative monotonely increasing functions are again monotonely increasing. We leave the details as an exercise for the interested reader. 
\end{proof} 

One can prove the following generalization of \cite[Lemma 3.7]{zentner} along the same lines.
\begin{lemma}\label{le:bound_b}
Let $r \geq 0$. Then for any $l \leq r$ there are functions 
	\begin{equation*}
		b_l \colon \R_+ \times \R_+ \times \R_+ \to \R_+
	\end{equation*}
which are continuous and monotonically increasing in all three variables, 
such that the following holds: 
\begin{enumerate}
\item
	For all $t\geq 0$ and $K \geq 0$ we have	$b_l(t,0,K)=0$.
\item
Let $X,Y \colon T^2 \to \R^2$ be vector fields on $T^2$, and suppose that for any $l \leq r+1$ the derivatives of $X$ satisfy
	\[
		\norm{D^{(i)}X}_\infty \leq K_l \quad \mathrm{for\ all\ } i \leq l
	\] 
	for some positive constants $K_l$. Then the flows $\phi^t_X, \phi^t_Y \colon T^2 \to T^2$ satisfy
	\begin{equation*}
	\norm{(D^{(l)}\phi^t_X)(p) - (D^{(l)}\phi^t_Y)(p)} \leq t \cdot b_l(t,\norm{X-Y}_{C^l},K_{l+1}) \cdot 
	\end{equation*}
for any $p \in T^2$. 
\end{enumerate}
\end{lemma} 

\begin{proof}
The case $r=0$ is Lemma 3.7 of \cite{zentner}, with $b_0(t,d,K_1) = d\cdot e^{K_1t}$. The proof for $r \geq 1$ again goes by induction.
To start the induction, let $r=1$.  Applying the triangle inequality as above, one gets 
\begin{align*}
	\norm{D\phi^t_X(p) - D\phi^t_Y(p)} 
		 \leq & \int_0^{t} \norm{DX_{\phi^s_X(p)} \circ D\phi^s_X(p) -  DY_{\phi^s_Y(p)} \circ D\phi^s_Y(p)}\, ds \\
		 \leq & \int_0^{t} \norm{DX_{\phi^s_X(p)} \circ D\phi^s_X(p) -  DX_{\phi^s_Y(p)} \circ D\phi^s_X(p)} \, ds \\* 
		 	&\qquad + \int_0^{t} \norm{DX_{\phi^s_Y(p)} \circ D\phi^s_X(p) - DX_{\phi^s_Y(p)} \circ D\phi^s_Y(p) } \, ds \\*
		 	&\qquad + \int_0^{t} \norm{DX_{\phi^s_Y(p)} \circ D\phi^s_Y(p) - DY_{\phi^s_Y(p)} \circ D\phi^s_Y(p)} \, ds \\
			\leq & \, \norm{D\phi^{s}_X}_{L^\infty(T^2 \times [0,t])} \cdot K_2 \cdot \int_0^{t}\norm{\phi^s_X(p) - \phi^s_Y(p)} \, ds \\*
			&\qquad + \norm{DX}_\infty \, \int_0^{t} \norm{D\phi^s_X(p) - D\phi^s_Y(p)} \, ds \, \\*
			&\qquad + t\,  \norm{DX - DY}_\infty \cdot \norm{D\phi^s_Y}_{L^\infty(T^2 \times [0,t])}.  
\end{align*}
Now \cite[Lemma 3.7]{zentner} yields the estimate $\norm{\phi^s_X(p) - \phi^s_Y(p)} \leq s \norm{X-Y}_{\infty} e^{K_1 s}$, where $K_1$ can be taken as a Lipschitz constant for $X$. Gronwall's inequality now yields
\begin{align*}
	\norm{D\phi_X^t(p) - D\phi_Y^t(p)} &\leq t \big( \norm{DX - DY}_\infty \cdot \norm{D\phi^s_Y}_{L^\infty(T^2 \times [0,t])} \\ 
	&\qquad +  K_2 \cdot \norm{X-Y}_{\infty} \cdot e^{K_1t} \cdot \norm{D\phi^s_X}_{L^\infty(T^2 \times [0,t])}\big) \cdot \exp(t \, \norm{DX}_\infty) \, . 
\end{align*}
This has the desired form once we notice that by \eqref{eq:flow_vs_its_vector_field} and the triangle inequality,
\begin{equation*}
\begin{split}
\norm{D\phi^s_Y}_\infty \leq \exp(\norm{DY}_\infty s) & \leq \exp(\norm{DX-DY}_\infty \, s + \norm{DX}_\infty \, s ) \\ & \leq \exp(\norm{DX-DY}_\infty s + K_1 s) \, . 
\end{split}
\end{equation*}

Assuming the claim holds up to $r-1$, the induction step now follows again from Equation~\eqref{eq:flow higher derivatives} similarly as in the proof of the previous lemma. Again, we leave the details to the reader.
\end{proof} 

We will also need a generalization of \cite[Lemma 3.11]{zentner}. 

\begin{lemma}\label{le:bound_c}
Let $r \geq 0$. Then for any $l \leq r$ and any $m \geq 2$ there are functions 
	\begin{equation*}
		c_l^{(m)} \colon \R_+ \times \R_+ \times \R_+ \to \R_+
	\end{equation*}
which are continuous and monotonically increasing in all three variables, 
such that the following holds: 

Let $Z$ be a vector field given as a sum of vector fields $Z= W_0 + \dots + W_{m-1}$.  Suppose that there are positive constants $K_0,\dots,K_{r+1}$ and $M_0,\dots,M_r$ such that for any $l \leq r+1$ one has
	\[
		\norm{D^{(i)} (W_{j_1} + \dots + W_{j_n})}_\infty \leq K_l
	\] 
for all $i \leq l$ and all nonempty subsets $\{j_1, \dots, j_n\}$ of $\{0, \dots, m-1\}$, and that for any $l \leq r$ and $1 \leq n \leq m-1$ we have
	\[
		\norm{D^{(i)}[W_n,W_0+ \dots + W_{n-1}]}_\infty \leq M_l
	\] 
for all $i \leq l$.  Then for any $p \in T^2$ we have 	
\begin{equation*}
	\norm{(D^{(l)}\phi^t_Z)(p) - (D^{(l)}(\phi^t_{W_m} \circ \dots \circ \phi^t_{W_1}))(p)} \leq t^2 \cdot c_l^{(m)}(t,M_l,K_{l+1}) 
	\end{equation*}
for all $t \geq 0$. 
\end{lemma}

\begin{proof}
We recall that for two vector fields $X,Y$ the Lie bracket is defined to be
\[
[X,Y](p) = \lim_{h \to 0} \frac{D\phi^h_Y (X(\phi_Y^{-h}(p)) - X(p)}{h}
		= \frac{d}{dt} [(\phi^t_Y)_*(X)](p) |_{t=0} \, . 
\]
Hence we have for any point $p$ the equality
\begin{equation}\label{commuting flow}
	D\phi^t_Y (X(p)) = X(\phi^t_Y(p)) + t [X,Y](\phi_Y^t(p)) + R_{X,Y}(\phi^t_Y(p),t) \, ,
\end{equation}
where the term $R_{X,Y}$ satisfies
\[
\lim_{t \to 0} \frac{R_{X,Y}(p,t)}{t} = 0 \, 
\]
and this convergence is uniform in $p$. Moreover, any derivatives of $R$ with respect to $p$ have the same limit as $t$ goes to $0$, since we can interchange the order of the derivatives. 

The proof of the lemma now goes by induction on $l$. The case $l=0$ is the statement of \cite[Lemma 3.11]{zentner}. We will prove the lemma for $l=1$, leaving the induction to higher $l$ to the interested reader. This case will itself proceed by induction on $m$, the number of summands in the expression $Z=W_0+\dots+W_{m-1}$, beginning with $m=2$.

Suppose we have $Z=W_0 + W_1$.  Differentiating the composition $\phi^{t}_{W_1} \circ \phi^{t}_{W_0}$ and using the above remark, we obtain
\begin{align} \label{diff flows}
			\frac{d}{dt}  \phi^t_{W_1} (\phi^t_{W_0}(p))  &= \frac{d \phi_{W_1}^t}{dt}(\phi^t_{W_0}(p)) + (\phi^t_{W_1})_* \left(\frac{d \phi_{W_0}^t(p)}{dt}\right) \nonumber \\
			 &= W_1( \phi_{W_1}^t(\phi^t_{W_0}(p))) + (\phi^t_{W_1})_* W_0(\phi^t_{W_0}(p)) \nonumber \\
			 &= (W_1 + W_0) ( \phi_{W_1}^t(\phi^t_{W_0}(p))) \\
			   &\qquad+ t [W_0,W_1](\phi^t_{W_1} (\phi^t_{W_0}(p))) + R_{W_0,W_1}(\phi^t_{W_1}(\phi^t_{W_0}(p)),t) \, . \nonumber
\end{align} 
Integrating this equation, and differentiating with respect to $p$, we obtain the integral equality
\begin{align}\label{integral_equality_composite}
 D (\phi^{t}_{W_1} \circ \phi^{t}_{W_0})(p) &= \id + \int_{0}^{t} D(W_1 + W_0) \circ D(\phi^s_{W_1} \circ \phi^s_{W_0})(p) \, ds \nonumber \\
 	&\qquad + \int_0^{t} s \, D([W_1,W_0]) \circ D(\phi^{s}_{W_1} \circ \phi^{s}_{W_0})(p) \, ds \\
	&\qquad + \int_0^{t} s \, D\left(\frac{R}{s}\right)(\phi^t_{W_1}(\phi^t_{W_0}(p))) \, ds \, .  \nonumber
\end{align}
We substract this from the corresponding integral equality which is satisfied by the derivative of the flow $\phi^{t}_{W_1+W_0}$, as in \eqref{eq:flow differentiated}, and this yields the estimate
\begin{align*}
	\left\|(D \phi^{t}_{W_1+W_0}) (p)\right. & - \left.D(\phi^{t}_{W_1} \circ \phi^{t}_{W_0})(p)\right\| \\*
		& \leq \int_0^{t} \left\| D(W_1 + W_0)_{\phi^s_{W_1+W_0}(p)} \circ D \phi^{s}_{W_1+W_0}(p) \right. \\*
		 & \qquad\qquad - \left.D(W_1 + W_0)_{(\phi^s_{W_1} \circ \phi^s_{W_0})(p)} \circ D(\phi^s_{W_1} \circ \phi^s_{W_0})(p)\right\| \, ds \\*
		& \qquad + \int_0^{t} s \norm{[W_1,W_0]}_{C^1} \, \norm{D (\phi^{s}_{W_1} \circ \phi^s_{W_0})} \, ds \\*
		& \qquad + \int_0^{t} s \norm{D \frac{R}{s}} \, ds \\
		&  \leq \int_0^{t} \norm{ D(W_1 + W_0)_{\phi^s_{W_1+W_0}(p)}}\norm{D \phi^{s}_{W_1+W_0}(p) 
			- D(\phi^s_{W_1} \circ \phi^s_{W_0})(p)} \, ds \\*
		& \qquad 
			+ \int_0^{t} \norm{D(W_1 + W_0)_{\phi^{s}_{W_1+W_0}(p)}
			- D(W_1 + W_0)_{(\phi^s_{W_1} \circ \phi^s_{W_0})(p)}} \norm{D(\phi^s_{W_1} \circ 					\phi^s_{W_0})(p)} \, ds \\*
		& \qquad + \frac{t^2}{2} \max_{s \in [0,t]} (\norm{[W_1,W_0]}_{C^1} \,\cdot \, \norm{D (\phi^{s}_{W_1} \circ \phi^s_{W_0})}) \\*
		&  \qquad + \frac{t^2}{2} \max_{s \in [0,t]} \norm{\frac{R}{s}}_{C^1} \\
		& \leq K_1 \int_0^{t} \norm{D \phi^{s}_{W_1+W_0}(p) 
			- D(\phi^s_{W_1} \circ \phi^s_{W_0})(p)} \, ds \\*
		& \qquad + K_2 \int_0^{t} \norm{\phi^s_{W_1+W_0}(p) - (\phi^s_{W_1} \circ \phi^s_{W_0})(p)} \, 					\norm{D (\phi^{s}_{W_1} \circ \phi^s_{W_0})(p)} \, ds \\*
		& \qquad + \frac{t^2}{2} \max_{s \in [0,t]} (\norm{[W_1,W_0]}_{C^1} \,\cdot \, \norm{D (\phi^{s}_{W_1} \circ \phi^s_{W_0})(p)}) \\*
		&  \qquad + \frac{t^2}{2} \max_{s \in [0,t]} \norm{\frac{R}{s}}_{C^1} \, .
		\stepcounter{equation}\tag{\theequation}\label{proof:bound_composite vs sum}
\end{align*}
By the induction hypothesis we have 
\[
\norm{\phi^s_{W_1+W_0}- \phi^s_{W_1} \circ \phi^s_{W_0}} \leq s^2 \, c_0^{(2)}(s,M_0,K_1) \, 
\]
for all $s$. Furthermore one easily shows that 
\[
\norm{D(\phi^{s}_{W_1} \circ \phi^s_{W_0})}_\infty \leq e^{K_1 s}
\]
holds for all $s$. If we plug these two inequalities into the estimate \eqref{proof:bound_composite vs sum}, Gronwall's inequality yields the desired result for $l=1$; here we define 
\[
	c_1^{(2)}(t,M,K):= \big( K \, c_0^{(2)}(t,M,K) e^{Kt} + \frac{1}{2} M e^{Kt} + \frac{1}{2} r_1(t,M,K) \big) \, e^{Kt} \, ,
\]
where $r_1(t,M,K)$ is a function that can be explicitly described in terms of $R_{W_1,W_0}$, which was defined in Equation~\eqref{commuting flow} above. We leave the induction in the cases $m \geq 3$ as an exercise for the interested reader. 
\end{proof}

We are now ready to start with the proof of Theorem \ref{C1 approximation}.
\subsection*{Strategy of the proof}
We observe that an isotopy $(\psi_t)_{t \in [0,1]}$ as in the statement of Theorem \ref{C1 approximation} is the flow of a time-dependent, divergence-free vector field $X_t$ satisfying
\begin{equation}\label{time dependent vector field}
	\frac{d \psi^t(p)}{dt} = X_t(\psi^t(p))
\end{equation}
for any $p \in T^2$ and any $t \in [0,1]$. Integrating the vector field from time $t_0$ to time $t$ sends a point $p \in T^2$ to $\psi^t((\psi^{t_0})^{-1}(p))$. 

Let $\epsilon > 0$ be given, and suppose that we want to approximate $\psi_t$ to within a distance of at most $\epsilon$ in the $C^r$ norm. 
The proof of Theorem \ref{C1 approximation} is an approximation in three steps, just as in the proof of \cite[Theorem~3.3]{zentner}.
\begin{enumerate}
	\item \label{i:strategy part 1}
	We approximate the flow $\psi_t$ of $X_t$ by a composition of flows of time-independent vector fields $X_i := X_{i/n}$, $i=0,\dots,n-1$.  More precisely, we define the isotopies $(\Theta^t_{(X_j)}): T^2 \to T^2$ on each interval $\frac{i}{n} \leq t \leq \frac{i+1}{n}$ by
\begin{equation}\label{broken flow}
	\Theta^t_{(X_j)} := \phi^{t-i/n}_{X_i} \circ \phi^{1/n}_{X_{i-1}} \circ \dots \circ \phi^{1/n}_{X_0} \, .
\end{equation}
We will show that by taking $n$ large enough, we get
	\[
	\norm{\Theta^t_{(X_j)} - \psi^t}_{C^{r}} < \frac{\epsilon}{3}
	\]
for all $t \in [0,1]$.

	\item \label{i:strategy part 2}
	We approximate each of the $X_i$ by a finite sum $Z_i$ of Fourier vector fields, using the $C^r$-version of \cite[Lemma~3.4]{zentner} mentioned above.  We define $\Theta^t_{(Z_j)}$ to be the analogue of \eqref{broken flow} above with $Z_i$ in place of $X_i$, and by keeping track of the accumulated error, we find that 
	\[
	\norm{\Theta^t_{(X_j)} - \Theta^t_{(Z_j)}}_{C^r} < \frac{\epsilon}{3}
	\]
for all $t \in [0,1]$, assuming each $Z_j$ was chosen sufficiently $C^r$-close to $X_j$.

	\item \label{i:strategy part 3}
We write each $Z_j$ as a finite Fourier series
\[
	Z_j = \sum_{r=0}^{m_j-1} W^{(j)}_r,
\]
where each $W^{(j)}_r$ is a shearing vector field in some direction in $\Z^2$.  We will approximate the flow of $Z_j$, which occurs over a time interval of length at most $\frac{1}{n}$ in $\Theta^t_{(Z_j)}$, by successive flows along each of the summands $W^{(j)}_{r}$.

For each $Z_j$, we fix some $k_j \in \N$ and define an isotopy $\Xi^t_{(W_r^{(j)})}$ as follows.  For $0 \leq t \leq \frac{1}{k_jm_jn}$, we flow along $W_0^{(j)}$ with speed $m_j$; then we flow along $W_1^{(j)}$ with speed $m_j$ for $\frac{1}{k_jm_jn} \leq t \leq \frac{2}{k_jm_jn}$; and so on, until we flow along $W_{m_j-1}^{(j)}$ during the $m_j$th interval of length $\frac{1}{k_jm_jn}$ and a total time of $\frac{1}{k_jn}$ has elapsed.  We then repeat this $k_j$ times to get the desired $\Xi^t_{(W_r^{(j)})}$, defined for $0 \leq t \leq \frac{1}{n}$.

We now approximate $\Theta^t_{(Z_j)}$ by an isotopy $\Omega^t_{(Z_j)}$, defined in terms of $\Xi^t_{(W_r^{(j)})}$ by
\begin{equation}\label{very broken flow}
	\Omega^t_{(Z_j)} := \Xi^{t-i/n}_{(W^{(i)}_r)} \circ \Xi^{1/n}_{(W^{(i-1)}_r)} \circ \dots \circ \Xi^{1/n}_{(W^{(0)}_r)} 
\end{equation}
for $\frac{i}{n} \leq t \leq \frac{i+1}{n}$, where $i = 0, \dots, n-1$. 
We will show that if for each sum $Z_j = \sum_{r=0}^{m_j-1} W^{(j)}_r$ we choose the corresponding $k_j$ large enough, then
	\[
	\norm{\Theta^t_{(Z_j)} - \Omega^t_{(Z_j)}}_{C^r} < \frac{\epsilon}{3}
	\]
for all $t \in [0,1]$.

\end{enumerate}

\subsection*{First step}
We start by approximating the isotopy $\psi_t$ by $\Theta^t_{(X_j)}$, as defined in \eqref{broken flow}.  Thus we need an estimate that bounds the distance from $\psi^t$ to $\Theta^{t}_{(X_j)}$ in the $C^r$-topology. 

\begin{lemma}\label{bound first step}
	Let $r \geq 0$. Then for any $l \leq r$ and $n \in \N$, there are functions 
	\[
		h_l^{(n)}: \R_+ \times \R_+ \times \R_+ \to \R_+
	\]
	which are continuous and monotonically increasing in all three variables, such that the following holds:
	\begin{enumerate}
		\item \label{i:bound first limit}The functions $h^{(n)}_l$ have the property that 
		\[
			\lim_{x \to 0^+} h^{(n)}_l(t,x,K) = 0 \, ,
		\]
		and the convergence is uniform in $n$, and also in $t$ and $K$ taken from compact subsets of 		$\R_+$. 
		\item \label{i:bound first estimate} Suppose that for any $l \leq r+1$ we have 
		\[
		\norm{D^{(i)}X_t}_\infty \leq K_l
		\] 
		for all $i \leq l$ and all $t \geq 0$.  Then we have the estimate
		\[
		\norm{D^{(l)}\psi^t - D^{(l)} \Theta^t_{(X_j)}}_\infty \, \leq \, h^{(n)}_l\left(t,\max_{j=0, \dots, n-1} \max_{\frac{j}{n} \leq t \leq \frac{j+1}{n}} \norm{X_j - X_t}_{C^{l}(T^2)},K_{l+1}\right) \, 
		\]
		for all $t \in [0,1]$. 
	\end{enumerate}
\end{lemma}
\begin{proof}
We prove the lemma by induction on $l$. The case $l=0$ is the conclusion of \cite[Lemma~3.9]{zentner}. In fact, we can define the function $h_0^{(n)}$ by the formula
\begin{equation*}
	h_0^{(n)}(t,x,K):= \left(t-\frac{j}{n}\right) x\, e^{K (t-\frac{j}{n})} + \frac{x}{n}\sum_{i=0}^{j-1} e^{K(t-\frac{i}{n})}
\end{equation*}
where $\frac{j}{n} \leq t \leq \frac{j+1}{n}$. This function clearly is bounded uniformly by $f_0(t,x,K) := t \, x \, e^{Kt}$, and hence the claim follows. 

For simplicity of notation and more clarity of the argument, we will treat the case $l=1$ only, and we will leave the induction argument for $l \geq 2$ to the interested reader.

The proof of Lemma \ref{le:bound_b} extends verbatim to the case of vector fields which depend on time. In our case, we will compare the flow of $X_t$ to that of $X_0$. The conclusion we get is the uniform bound
\begin{equation}\label{proof:comparing_flows}
\norm{D\psi^t - D\phi^t_{X_0}} \leq t \cdot b_1(t,\sup_{s \in [0,t]} \norm{X_0-X_s}_{C^1(T^2)},K_2)\, 
\end{equation}
which holds for all $t \geq 0$. 
It is also straightforward to verify the bounds
\begin{equation}\label{proof:bound_0}
	\norm{D \psi^{t}}_\infty \leq e^{K_1 \, t} \quad \text{and} \quad \norm{D \phi^{t}_{X_j}}_\infty \leq e^{K_1 \, t}\, 
\end{equation}
for any $j=0, \dots, n-1$, cf.\ Equation~\eqref{eq:flow_vs_its_vector_field}.

We claim that for $\frac{j}{n} \leq t \leq \frac{j+1}{n}$ we have
\begin{align}\label{proof:bound_1}
\left\lVert (D\psi^t)(p) \right. & \left.- (D\Theta^t_{(X_j)})(p)\right\rVert \nonumber \\
	& \leq \left(t-\frac{j}{n}\right) \, \left[ b_1\left(\frac{1}{n},\sup_{s \in [\frac{j}{n},\frac{j+1}{n}]} \norm{X_j - X_s}_{C^1(T^2)},K_2\right) \right. \\*
	& \qquad
	+ \left.a_1\left(\frac{1}{n},K_2\right) \, h_0^{(n)}\left(\frac{j}{n},\sup_{s \in [\frac{j}{n},\frac{j+1}{n}]} \norm{X_i - X_s}_{C^1(T^2)},K_1\right) \right] e^{K_1 \frac{j}{n}} \nonumber \\*
	& \quad + \sum_{i=0}^{j-1} \frac{1}{n} \left[ b_1\left(\frac{1}{n},\sup_{s \in [\frac{j}{n},\frac{j+1}{n}]} \norm{X_j - X_s}_{C^1(T^2)},K_2\right) \right.\nonumber \\*
	& \qquad+ \left.a_1\left(\frac{1}{n},K_2\right) \, h_0^{(n)}\left(\frac{i}{n},\max_{i=0, \dots, n-1} \sup_{s \in [\frac{i}{n},\frac{i+1}{n}]} \norm{X_i - X_s}_{C^1(T^2)},K_1\right) \right] e^{K_1 \frac{j}{n}}. \nonumber
\end{align}
Here $a_1$ and $b_1$ are some functions satisfying the statement of Lemmas \ref{le:bound_a} and \ref{le:bound_b} above. 

We prove this claim by induction on $j$.  The case $j=0$, and hence the claim for times $0 \leq t \leq \frac{1}{n}$ follows immediately from \eqref{proof:comparing_flows}. (Notice that $h_0^{(n)}(0,x,K) = 0$.)

Suppose the claim holds for $t \leq \frac{j}{n}$.
Using the triangle inequality, Lemmas \ref{le:bound_a} and \ref{le:bound_b}, and Equation~\eqref{proof:bound_0} above, we obtain for $\frac{j}{n} \leq t \leq \frac{j+1}{n}$ the bound
\begin{align*}
   \norm{(D\psi^t)(p) - (D\Theta^t_{(X_j)})(p)} &= \norm{D (\psi^t \circ (\psi^{\frac{j}{n}})^{-1})_{\psi^{\frac{j}{n}}(p)} \circ (D\psi^{\frac{j}{n}})(p) - (D\phi^{t-\frac{j}{n}}_{X_j})_{\Theta^{\frac{j}{n}}(p)} \circ D \Theta^{\frac{j}{n}}_{(X_j)}(p)} \\
   	 &\leq \norm{D (\psi^t \circ (\psi^{\frac{j}{n}})^{-1})_{\psi^{\frac{j}{n}}(p)} - 
	 		(D\phi^{t-\frac{j}{n}}_{X_j})_{\psi^{\frac{j}{n}}(p)}} \,\cdot\,  \norm{(D\psi^{\frac{j}{n}})(p)} \\*
	      &\qquad + \norm{(D\phi^{t-\frac{j}{n}}_{X_j})_{\psi^{\frac{j}{n}}(p)} - 
	        (D\phi^{t-\frac{j}{n}}_{X_j})_{\Theta^{\frac{j}{n}}(p)}} \,\cdot\, 
	         \norm{(D\psi^{\frac{j}{n}})(p)} \\*
	      &\qquad + \norm{(D\psi^{\frac{j}{n}})(p) - (D\Theta^{\frac{j}{n}})(p)} \,\cdot\, 
	      	  \norm{(D\phi^{t-\frac{j}{n}}_{X_j})_{\Theta^{\frac{j}{n}}(p)}} \\
	  &\leq \left(t-\frac{j}{n}\right)\,  b_1\left(t-\frac{j}{n}, \sup_{s \in [\frac{j}{n},t]} \norm{X_j-X_s}_{C^1(T^2)},K_2\right)\,  e^{K_1 \frac{j}{n}} \\*
	  &\qquad + \left(t-\frac{j}{n}\right) \, a_1\left(t-\frac{j}{n},K_2\right) \, \norm{\psi^{\frac{j}{n}}(p) - \Theta^{\frac{j}{n}}(p)}\, e^{K_1 \frac{j}{n}} \\*
	  &\qquad + \norm{(D\psi^{\frac{j}{n}})(p) - (D\Theta^{\frac{j}{n}})(p)} \, e^{K_1 (t-\frac{j}{n})}. 
\end{align*}
Using the induction hypothesis to bound the last term, and the case $l=0$ to bound the middle term, we clearly obtain the bound \eqref{proof:bound_1}. Hence we define the function $h_1^{(n)}$ by the formula
\begin{align*}
	h_1^{(n)}(t,x,K) &:= \left(t-\frac{j}{n}\right) \left(b_1\left(\frac{1}{n},x,K\right) + a_1\left(\frac{1}{n},K\right) h_0^{(n)}\left(\frac{j}{n},x,K\right) \right) e^{K t} \\
		&\qquad + \sum_{i=0}^{j-1} \frac{1}{n} \left(b_1\left(\frac{1}{n},x,K\right) + a_1\left(\frac{1}{n},K\right) h_0^{(n)}\left(\frac{i}{n},x,K\right) \right) e^{K\frac{i}{n}}
\end{align*}
for $\frac{j}{n} \leq t \leq \frac{j+1}{n}$. The bound \eqref{proof:bound_1} implies statement~\eqref{i:bound first estimate} of the lemma for the case $l=1$. 

The functions $h_1^{(n)}$ satisfy the bound 
\[
	h_1^{(n)}(t,x,K) \leq f_1(t,x,K) := \big( b_1(1,x,K) + a_1(1,K) f_0(t,x,K) \big) e^{K t}
\]
for all $t \in [0,1]$, where we recall that $f_0(t,x,K) = txe^{Kt}$. By the properties of the functions $f_0$, $a_1$ and $b_1$ we clearly have 
\[
	\lim_{x \to 0^+} f_1(t,x,K) = 0\, ,
\]
and this limit is uniform in $t$ and $K$ taken from compact subsets. This implies statement~\eqref{i:bound first limit} of the lemma for the case $l=1$. 
\end{proof}

The first statement of Lemma~\ref{bound first step} tells us that there are functions
\[
	f_l: \R_+ \times \R_+ \to \R_+
\]
such that 
\[
	h_l^{(n)}(t,x,K_{r+1})  \leq f_l(x,K_{r+1})
\]
for all $n \in \N$ and all $t \in [0,1]$. Furthermore, these functions satisfy 
\[
\lim_{x \to 0^+} f_l(x,K)=0 
\]
for all $K$, so we can choose $M>0$ so small that $f_l(M,K_{r+1}) < \frac{\epsilon}{3(r+1)}$ for all $l=0, \dots,r$.

We now notice that the time dependent vector field $(X_t)\colon T^2 \times [0,1] \to \R^2$ is equicontinuous, and all of its derivatives are also equicontinuous maps. In particular, given $M$ as above, there is some $\delta > 0$ such that 
we have 
\[
	\norm{X_s - X_t}_{C^r(T^2)} < M \quad\mathrm{whenever}\quad \abs{s-t} < \delta. 
\]
Choosing $n$ large enough so that $\frac{1}{n} < \delta$, the second conclusion of Lemma~\ref{bound first step} implies that
\[
	\norm{\Theta^t_{(X_j)} - \psi^{t}}_{C^{r}(T^2)} < \frac{\epsilon}{3} \, . 
\]
for all $t \in [0,1]$, as desired.

\subsection*{Second step} 
We now approximate $\Theta^t_{(X_j)}$ by $\Theta^t_{(Z_j)}$, where each $Z_i$ is a finite Fourier series which $C^r$-approximates $X_i$.  The following is nearly a duplicate of Lemma~\ref{bound first step}, with $\Theta^t_{(X_j)}$ and $\Theta^t_{(Z_j)}$ in place of $\psi^t$ and $\Theta^t_{(X_j)}$.

\begin{lemma}\label{bound second step}
	Let $r \geq 0$. Then for any $l \leq r$ and $n \in \N$, there are functions 
	\[
		g_l^{(n)}: \R_+ \times \R_+ \times \R_+ \to \R_+
	\]
	which are continuous and monotonically increasing in all three variables, such that:
	\begin{enumerate}
		\item The functions $g^{(n)}_l$ have the property that 
		\[
			\lim_{x \to 0^+} g^{(n)}_l(t,x,K) = 0 \, ,
		\]
		and the convergence is uniform in $n$, and also in $t$ and $K$ taken from compact subsets of 		$\R_+$. 
		\item Suppose that for any $l \leq r+1$ we have 
		\[
		\norm{D^{(i)}X_j}_\infty \leq K_l 
		\] 
		for all $i \leq l$.  Then we have the estimate
		\[
		\norm{D^{(l)}\Theta^t_{(X_j)} - D^{(l)} \Theta^t_{(Z_j)}}_\infty \, \leq \, g^{(n)}_l(t,\max_{j=0, \dots, n-1} \max_{\frac{j}{n} \leq t \leq \frac{j+1}{n}} \norm{X_j - Z_j}_{C^{l}(T^2)},K_{l+1}) \, 
		\]
		for all $t \in [0,1]$. 
	\end{enumerate}
\end{lemma}
\begin{proof}
	We repeat the proof of Lemma~\ref{bound first step} verbatim, using the same functions $h_l^{(n)}$ as before.  The only details which need to be established first are bounds for each $l \leq r+1$ of the form $\norm{D^{(i)}Z_j}_\infty \leq K_l$, for all $i \leq l$ and all $j=0, \dots, n-1$. But clearly we have 
	\[
		\norm{D^{(i)}Z_j}_\infty \leq \norm{D^{(i)} X_j}  + \norm{D^{(i)} (Z_j - X_j)} 
			\leq K_l + \norm{Z_j - X_j}_{C^{l}}\, , 
	\] 
and we can absorb the extra $\norm{Z_j-X_j}_{C^l}$ in the definition of the functions $g^{(n)}_l$. 
\end{proof} 

The first statement of the preceding lemma tells us that 
\[
\lim_{x \to 0^+} g_l^{(n)}(t,x,K_{r+1}) = 0
\]
uniformly in $t \in [0,1]$. (We don't need the statement about uniform convergence in $n$, as $n$ was already fixed in the first approximation step.)  We choose $M>0$ small enough that $g_l^{(n)}(1,M,K_{r+1}) \leq \frac{\epsilon}{3(r+1)}$ for all $l=0, \dots,r$. Having fixed $M$, we choose the finite Fourier sums $Z_j=W^{(j)}_0 + \dots + W^{(j)}_{m_j-1}$ so that we have
\[
	\norm{X_j - Z_j}_{C^{r}} \leq M
\]
for $j=0, \dots, n-1$. Now the second conclusion of Lemma~\ref{bound second step} implies that
\[
	\norm{\Theta^t_{(X_j)} - \Theta^{t}_{(Z_j)}}_{C^{r}(T^2)} < \frac{\epsilon}{3} 
\]
for all $t \in [0,1]$. 

\subsection*{Third step} 
We start with a lemma that describes the $C^l$-distance between the flow $\phi^{t}_{Z}$ of $Z=W_0+\dots+W_{m-1}$ and the isotopy $\Xi^{t}_{(W_r)}$, defined as in step~\eqref{i:strategy part 3}.  To be explicit, suppose that we have fixed positive integers $n$ and $k$, so that for any $t \in [0,\frac{1}{n}]$ we can find integers $i,r$ such that
\[ \frac{1}{n}\left(\frac{i}{k}+\frac{r}{km}\right) \leq t \leq \frac{1}{n}\left(\frac{i}{k}+\frac{r+1}{km}\right) \]
with $0 \leq i \leq k-1$ and $0 \leq r \leq m-1$.  Then we define $\Xi^t_{(W_r)}$ by the formula
\[ \Xi^t_{(W_r)} := \phi_{W_r}^{m(t-\frac{im+r}{kmn})} \circ \phi_{W_{r-1}}^{\frac{1}{kn}} \circ \dots \circ \phi_{W_0}^{\frac{1}{kn}} \circ \big(\phi_{W_{m-1}}^{\frac{1}{kn}} \circ \dots \circ \phi_{W_0}^{\frac{1}{kn}}\big)^i. \]

\begin{lemma}\label{bound third step}
	Let $r \geq 0$. Then for any $l \leq r$ and $k \in \N$, there are functions 
	\[
		d_l^{(k)}: \R_+ \times \R_+ \times \R_+ \to \R_+
	\]
	which are continuous and monotonically increasing in all three variables, such that:
	\begin{enumerate}
		\item The functions $d^{(k)}_l$ have the property that 
		\[
			\lim_{k \to \infty} d^{(k)}_l(t,M,K) = 0 \, 
		\]
		for any $M$, $K$, and $t$. 
		\item Let $Z$ be a vector field given as a sum of vector fields $Z= W_0 + \dots + W_{m-1}$. 
		Suppose that there are positive constants $K_0,\dots,K_{r+1}$ and $M_0,\dots,M_r$ such that
		for any $l \leq r+1$ we have
		\[
			\norm{D^{(i)} (W_{j_1} + \dots + W_{j_n})}_\infty \leq K_l
		\] 
		for all $i \leq l$ and all finite subsets $\{j_1, \dots, j_n\}$ of $\{0, \dots, m-1\}$, and such that
		\[
			\norm{D^{(i)}[W_n,W_0+ \dots + W_{n-1}]}_\infty \leq M_l
		\] 
		for all $i \leq l \leq r$ and all $n \geq 1$.  Then we have the estimate
		\[
		\norm{D^{(l)}\phi^t_{Z} - D^{(l)} \Xi^{t}_{(W_r)}}_\infty \, \leq \, 
			d^{(k)}_l(t, M_l,K_{l+1}) \, 
		\]
		for all $t\geq 0$. (We recall here that $\Xi^{t}_{(W_r)}$ depends on $k$.)
	\end{enumerate}
\end{lemma}
\begin{proof}
The case $l=0$ is the content of \cite[Lemma 3.12]{zentner}. Again, we will show the case $l=1$ and leave the induction for higher $l$ to the interested reader; we start by noting that
\begin{equation}\label{proof:thirdbound_bounds1}
	\norm{D\Xi^{t}_{(W_r)}}_\infty \leq e^{K_1 t}
\end{equation}
for all $t\geq 0$. 

We first claim that at times $t=\frac{j}{nk}$, where $j$ is an integer, we have
\begin{equation}\label{proof:thirdbound_bound2}
	\begin{split}
	\norm{D\Xi_{(W_r)}^{\frac{j}{nk}}-D\phi^{\frac{j}{nk}}_Z}_\infty 
	& \leq \sum_{i=0}^{j-1} \frac{1}{nk} \left[ \frac{1}{nk} c_1\left(\frac{1}{nk},M_1,K_2\right)\right. \\
	&\qquad\qquad+ 
		\left.a_1\left(\frac{1}{nk},K_2\right) \, d_0^{(k)}\left(\frac{j}{nk},M_1,K_2\right) \right] e^{K_1 \frac{i}{nk}}\, . 
	\end{split}
\end{equation}
Here $c_1$ denotes a function $c_1^{(m)}: \R_+ \times \R_+ \times \R_+ \to \R_+$ as provided by Lemma \ref{le:bound_c} above; we omit the superscript $(m)$ for convenience, as well as the $(W_r)$ subscript on $\Xi$ in the sequel.

The proof of \eqref{proof:thirdbound_bound2} is by induction, analogous to the proof of the bound \eqref{proof:bound_1} in Lemma~\ref{bound second step}; in this situation, we make use of Lemma~\ref{le:bound_c} above.  By using the triangle inequality, followed by Lemma~\ref{le:bound_a} and \eqref{eq:flow_vs_its_vector_field}, we obtain the estimate
\begin{align}\label{proof:bound_3}
   \norm{(D\Xi^t)(p) - (D\phi^t_{Z})(p)} &= \norm{D (\Xi^{t-\frac{j}{nk}} )_{\Xi^{\frac{j}{nk}}(p)} \circ (D\Xi^{\frac{j}{nk}})(p) - (D\phi^{t-\frac{j}{nk}}_{Z})_{\phi^{\frac{j}{nk}}_{Z}(p)} \circ D \phi^{\frac{j}{nk}}_{Z}(p)} \nonumber\\
   	 &\leq \norm{D (\Xi^{t-\frac{j}{nk}})_{\Xi^{\frac{j}{nk}}(p)} - 
	 		(D\phi^{t-\frac{j}{nk}}_{Z})_{\Xi^{\frac{j}{nk}}(p)}} \,\cdot\,  \norm{(D\Xi^{\frac{j}{nk}})(p)} \nonumber\\*
	      &\qquad + \norm{(D\phi^{t-\frac{j}{nk}}_{Z})_{\Xi^{\frac{j}{nk}}(p)} - 
	        (D\phi^{t-\frac{j}{nk}}_{Z})_{\phi^{\frac{j}{nk}}_{Z}(p)}} \,\cdot\, 
	         \norm{(D\Xi^{\frac{j}{nk}})(p)} \nonumber\\*
	      &\qquad + \norm{(D\Xi^{\frac{j}{nk}})(p) - (D\phi^{\frac{j}{nk}}_{Z})(p)} \,\cdot\, 
	      	  \norm{(D\phi^{t-\frac{j}{nk}}_{Z})_{\Theta^{\frac{j}{nk}}(p)}} \nonumber\\
	  &\leq \norm{D (\Xi^{t-\frac{j}{nk}})_{\Xi^{\frac{j}{nk}}(p)} - 
	 		(D\phi^{t-\frac{j}{nk}}_{Z})_{\Xi^{\frac{j}{nk}}(p)}}\,  e^{K_1 \frac{j}{nk}} \nonumber\\*
	  &\qquad + \left(t-\frac{j}{nk}\right) \, a_1\left(t-\frac{j}{nk},K_2\right) \, \norm{\Xi^{\frac{j}{nk}}(p) - \phi_{Z}^{\frac{j}{nk}}(p)}\, e^{K_1 \frac{j}{nk}} \nonumber\\*
	  &\qquad + \norm{(D\Xi^{\frac{j}{nk}})(p) - (D\phi^{\frac{j}{nk}}_{Z})(p)} \, e^{K_1 (t-\frac{j}{nk})}.
\end{align}

Supposing we are trying to prove \eqref{proof:thirdbound_bound2} for $t=\frac{j+1}{nk}$, the last two terms of \eqref{proof:bound_3} can be bounded using the case $l=0$ of the lemma and the case $t=\frac{j}{nk}$ of \eqref{proof:thirdbound_bound2} respectively.  At time $t=\frac{j+1}{nk}$ the map $\Xi^{t-\frac{j}{nk}}$ is equal to $\Xi^{\frac{1}{nk}} = \phi_{W_{m-1}}^{\frac{1}{nk}} \circ \dots \circ \phi_{W_0}^{\frac{1}{nk}}$, and hence Lemma~\ref{le:bound_c} gives us the bound
\begin{equation*}
\begin{split}
	\norm{D (\Xi^{t-\frac{j}{nk}})_{\Xi^{\frac{j}{nk}}(p)} - 
	 		(D\phi^{t-\frac{j}{nk}}_{Z})_{\Xi^{\frac{j}{nk}}(p)}}
	& = 	\norm{D (\Xi^{\frac{1}{nk}})_{\Xi^{\frac{j}{nk}}(p)} - 
	 		(D\phi^{\frac{1}{nk}}_{Z})_{\Xi^{\frac{j}{nk}}(p)}} \\
	& \leq \left(\frac{1}{nk}\right)^2 c_1\left(\frac{1}{nk}, M_1,K_2\right) \, .
\end{split}
\end{equation*}
This bound together with \eqref{proof:bound_3} establishes the bound \eqref{proof:thirdbound_bound2} by induction.

Now the intermediate times $\frac{j}{kn} \leq t \leq \frac{j+1}{kn}$ need a separate treatment, because for these times $\Xi^{t-\frac{j}{nk}}$ is not of the form where Lemma \ref{le:bound_c} applies. In fact, for these times we can only establish a coarser estimate:
\begin{align}\label{proof:estimate_coarse}
	\norm{D \Xi^{t-\frac{j}{nk}} - 
	 		D\phi^{t-\frac{j}{nk}}_{Z}}_\infty
		& \leq  
		\norm{D \Xi^{t-\frac{j}{nk}} - \id}_\infty + \norm{D\phi^{t-\frac{j}{nk}}_{Z} - \id}_\infty \nonumber\\
		& \leq \left(t-\frac{j}{nk}\right) K_1 e^{K_1 (t-\frac{j}{nk})} + \left(t-\frac{j}{nk}\right) K_1 e^{K_1 (t-\frac{j}{nk})}
		\nonumber\\
		& = 2 \left(t-\frac{j}{nk}\right) K_1 e^{K_1 (t-\frac{j}{nk})} \, .
\end{align}
The estimate $\norm{D\phi^{t}_{Z} - \id}_\infty \leq t \, K_1 e^{K_1 t}$ is an immediate application of Gronwall's inequality to the derivative of the flow equation of $Z$, i.e.\ equation \eqref{eq:flow differentiated}, and we leave the estimate 
\[
\norm{D\Xi^{t-\frac{j}{nk}}_{Z} - \id}_\infty \leq (t-\frac{j}{nk}) \, K_1 e^{K_1 (t-\frac{j}{nk})} 
\]
for $\frac{j}{kn} \leq t \leq \frac{j+1}{kn}$ as an exercise to the interested reader.

In total, we have now proved that for any $\frac{j}{nk} \leq t \leq \frac{j+1}{nk}$ we have
\begin{equation}\label{eq:pre f}
\begin{split}
	\norm{D\Xi_{(W_r)}^t - D\phi^t_{Z}}_\infty 
	& \leq 
	\sum_{i=0}^{j} \frac{1}{nk} \left[ \max\left\{\frac{1}{nk} \, c_1\left(\frac{1}{nk},M_1,K_2\right),\frac{2}{nk} K_1 e^{\frac{K_1}{nk}}\right\} \right. \\
	& \qquad\qquad\qquad + \left.
		a_1\left(\frac{1}{nk},K_2\right) d_0^{(k)}\left(\frac{j}{nk},M_1,K_2\right) \right] e^{K_1 \frac{i}{nk}}\, .
\end{split}
\end{equation}
Now notice that for $\frac{j}{nk} \leq t < \frac{j+1}{nk}$ the sum on the right hand side has $j+1 = \floor{nk t} + 1$ summands. Thus we clearly obtain the estimate
\begin{equation}\label{proof:f}
	\begin{split}
	\norm{D\Xi_{(W_r)}^t - D\phi^t_{Z}}_\infty 
	& \leq 2 \left(t + \frac{1}{k}\right)
	 \left( \max\left\{\frac{1}{k} \, c_1(1,M_1,K_2),\frac{2}{k} K_1 e^{\frac{K_1}{k}}\right\} \right. \\
	& \quad \quad \quad + \left.
		a_1\left(\frac{1}{k},K_2\right) d_0^{(k)}(1,M_1,K_2) \right) e^{K_1 t}\, .
\end{split}
\end{equation}
(The functions $a_1$, $c_1$, and $d_0^{(k)}$ are increasing in $t$, so we can thus obtain expressions independent of $n$.)
We define $d_1^{(k)}(t,M_1,K_2)$ to be the right hand side of \eqref{proof:f}, and then
\[
\lim_{k \to \infty} d_1^{(k)}(t,M,K) = 0
\] 
follows immediately from the fact that $\lim_{k \to \infty} d_0^{(k)}(t,M,K) = 0$. 
\end{proof}

We are now ready to compare $\Theta^{t}_{(Z_j)}$ to the map $\Omega^{t}_{(Z_j)}$ defined in equation~\eqref{very broken flow} above. Recall that the definition of $\Omega^{t}_{(Z_j)}$ depends on the summands $W_0^{(j)}, \dots, W_{m_j-1}^{(j)}$ in $Z_j = W_0^{(j)} + \dots + W_{m_j-1}^{(j)}$ and on the integers $k_j$ which appear in the definitions of the maps $\Xi^t_{(W_{r}^{(j)})}$, for $j=0,1,\dots,n-1$. Here $n$ has been fixed in the first approximation step, and the vector fields $Z_0, \dots, Z_{n-1}$ have been fixed in the second approximation step. 

\begin{lemma}\label{third step estimate}
Let $r \geq 0$. Then for any $l \leq r$ there are functions 
	\[
	f_l^{(k_0, \dots, k_{n-1})}:\R_+ \times \R_+^{n} \times \R_+^{n} \to \R,
	\]
which are continuous and monotonically increasing in all variables, such that:
\begin{enumerate}
	\item For any tuples ${\bf M}=(M_0, \dots, M_{n-1})$ and ${\bf K}=(K_0, \dots, K_{n-1})$ we have
	\begin{equation*}
		\lim_{k_{\text{min}} \to \infty} f_l^{(k_0, \dots, k_{n-1})}(t,{\bf M},{\bf K}) = 0	\, ,
	\end{equation*} 
	where $k_{\text{min}}$ is defined as the minimum of the numbers $k_0, \dots, k_{n-1}$. \\
	
	\item Let $Z_j$ be a vector field of the form $W_0^{(j)} + \dots + W^{(j)}_{m_j-1}$ for $0 \leq j \leq n-1$.
		Suppose that for each $j$ there are positive constants $K_0^{(j)},\dots,K_{r+1}^{(j)}$ and $M_0^{(j)},\dots,M_r^{(j)}$ such that for any $l \leq r+1$ one has 
		\[
			\norm{D^{(i)} (W^{(j)}_{i_1} + \dots + W^{(j)}_{i_q})}_\infty \leq K_l^{(j)}
		\] 
		for all $i \leq l$ and all subsets $\{i_1, \dots, i_q\}$ of $\{0, \dots, m_{j}-1\}$, and
		\[
			\norm{D^{(i)}[W_q^{(j)},W_0^{(j)} + \dots + W_{q-1}^{(j)}]}_\infty \leq M_l^{(j)}
		\] 
		for all $i \leq l \leq r$ and $1 \leq q \leq m_j-1$.  Then we have the estimate
		\begin{equation*}
		\norm{D^{(l)}\Theta^t_{(Z_j)} - D^{(l)} \Omega^t_{(W_r^{(j)})}}_\infty \leq f^{(k_0,\dots,k_{n-1})}_l(t, (M_l^{(0)}, \dots,M_l^{(n-1)}),(K_{l+1}^{(0)}, \dots, K_{l+1}^{(n-1)})) \, 
		\end{equation*}
		for all $t\geq 0$. 
\end{enumerate}
\end{lemma}
\begin{proof}
	The theme of the proof is by now familiar. The case $l=0$ is the statement of \cite[Lemma 3.13]{zentner}. 
	For the case $l=1$, and for times $\frac{j}{n} \leq t \leq \frac{j+1}{n}$, one starts with the estimate 
\begin{align}\label{estimate:very_broken_trajectory}
		\norm{D \Omega^t_{(Z_j)} - D\Theta^{t}_{(Z_j)}} &= \norm{D \Xi^{t-\frac{j}{n}}_{(W_r^{(j)})} \circ D\Omega^{\frac{j}{n}}_{(Z_j)} - D\phi^{t-\frac{j}{n}}_{Z_j} \circ D \Theta^{\frac{j}{n}}_{(Z_j)} } \nonumber \\
			& \leq \norm{D \Omega^{t-\frac{j}{n}}_{(Z_j)} - D\phi^{t-\frac{j}{n}}_{Z_j}} \, \norm{D \Omega^{\frac{j}{n}}_{(Z_j)}} + \norm{D\phi^{t-\frac{j}{n}}_{Z_j}} \, \norm{D\Omega^{\frac{j}{n}}_{(Z_j)} - D \Theta^{\frac{j}{n}}_{(Z_j)} }\, .
\end{align}
Now it is simple to verify that $\norm{D \Omega^{\frac{j}{n}}_{(Z_j)}} \leq e^{\frac{1}{n}(K_1^{(0)} + \dots + K_1^{(j-1)})}$, and together with Lemma \ref{bound third step} this gives inductively for $\frac{j}{n} \leq t \leq \frac{j+1}{n}$ the estimate
\begin{align*}
		\norm{D \Omega^t_{(Z_j)} - D\Theta^{t}_{(Z_j)}} & \leq \sum_{i=0}^{j-1} d_1^{(k_i)}\left(\frac{1}{n},M_1^{(j)},K_2^{(j)}\right) e^{\frac{1}{n} (K_1^{(0)} + \dots + K_1^{(i-1)})} \\
			& \qquad + d_1^{(k_j)}\left(t-\frac{j}{n},M_1^{(j)},K_2^{(j)}\right) e^{\frac{1}{n} (K_1^{(0)} + \dots + K_1^{(j-1)})} \, ,
\end{align*}
where the $d_1^{(k)}$ are functions satisfying the conclusion of Lemma~\ref{bound third step}. The right hand side defines the function $f_1^{(k_0,\dots,k_{n-1})}$, and the fact that
	\begin{equation*}
		\lim_{k_{\text{min}} \to \infty} f_l^{(k_0, \dots, k_{n-1})}(t,{\bf M},{\bf K}) = 0	\, 
	\end{equation*} 
follows from the first conclusion of Lemma \ref{bound third step}.

We leave the induction steps for $l \geq 2$ as an exercise for the interested reader.
\end{proof}

Once the vector fields $Z_j$ and their decomposition as sums $Z_j = W_{0}^{(j)} + \dots + W_{m_j-1}^{(j)}$ are fixed, there are constants $K_l^{(j)}$ and $M_l^{(j)}$ satisfying the hypotheses of Lemma \ref{third step estimate}.  Lemma~\ref{third step estimate} then allows us to choose the multi-index $(k_0, \dots, k_{n-1})$ large enough so that we have
\[
	f_l^{(k_0, \dots, k_{n-1})}(t,(M_l^{(0)}, \dots, M_l^{(n-1)}),(K_{l+1}^{(0)}, \dots, K_{l+1}^{(n-1)})) < \frac{\epsilon}{3(r+1)}
\]
for all $l=0, \dots, r$, and it follows that
\begin{equation*}
	\norm{\Theta^{t}_{(Z_j)} - \Omega^t_{(Z_j)}}_{C^r} < \frac{\epsilon}{3}
\end{equation*}
for all $t \in [0,1]$. This concludes the proof of Theorem \ref{C1 approximation}.

\section{Area-preserving isotopies of surfaces} \label{sec:symplectic-isotopy}

In this appendix we prove Lemma~\ref{lem:symplectic-isotopy}, which was used in the proofs of Theorems~\ref{thm:pillowcase-near-pi/2} and \ref{thm:det-limit-slope}.  We restate the lemma here for convenience.

{
\renewcommand{\thetheorem}{\ref*{lem:symplectic-isotopy}}
\begin{lemma}
Let $\Sigma$ be a compact surface with area form $\omega$, and let $K$ be a compact subset of $\Sigma$.
Let $C_t \subset \Sigma$ ($0 \leq t \leq 1$) be a smooth isotopy of smoothly embedded curves (not necesarily connected), and suppose that
\begin{enumerate}
\item the union $\bigcup_{t\in[0,1]} C_t$ is contained in the interior of a compact, connected subsurface $\Sigma'$ with smooth boundary, such that $\Sigma'$ is disjoint from $K$;
\item each component $\gamma \subset C_t$ is separating in $\Sigma \setminus (C_t \setminus \gamma)$; \label{i:separating-condition}
\item the areas of the respective components of $\Sigma \setminus C_t$ are independent of $t$.
\end{enumerate}
Then there is a smooth isotopy $\psi_t: \Sigma \to \Sigma$ with $\psi_0 = \mathrm{id}_\Sigma$ such that
\begin{enumerate}
\item $\psi_t(C_0) = C_t$ for all $t$;
\item $\psi_t$ is constant on $\overline{\Sigma\setminus\Sigma'}$, and hence on a neighborhood of $K$, for all $t \in [0,1]$;
\item each $\psi_t$ is a symplectomorphism, i.e., $\psi_t^*\omega = \omega$ for all $t$.
\end{enumerate}
\end{lemma}
\addtocounter{theorem}{-1}
}

\begin{proof}
We break the proof into two steps: first we arrange an isotopy which fixes the symplectic form near $C_t$ and on the closure of $\Sigma \setminus \Sigma'$, and then we use another isotopy to fix it away from these sets.
\vspace{1em}

\noindent{\bf Step 1}: Find disjoint open neighborhoods $V$ of $C_0$ and $W$ of $\overline{\Sigma\setminus\Sigma'}$, and an isotopy $\varphi_t: \Sigma \to \Sigma$ with $\varphi_t(C_0) = C_t$, such that $\left(\varphi_t^*\omega\right)|_V = \left.\omega\right|_V$ and $\varphi_t|_W = \mathrm{id}_W$ for all $t$.

We start by using the isotopy extension theorem to produce an isotopy
\[ \phi_t: \Sigma \to \Sigma, \quad 0 \leq t \leq 1, \]
supported on an arbitrarily small neighborhood of the compact set $\bigcup_t C_t \subset \Sigma'$, such that $\phi_t(C_0) = C_t$ for all $t$.  By taking the support small enough we can make it disjoint from an open neighborhood $W$ of $\overline{\Sigma\setminus\Sigma'}$, since the latter is compact and disjoint from $\bigcup_t C_t$.

Since $C_0$ is Lagrangian with respect to each of the forms $\phi_t^*\omega$, $0 \leq t \leq 1$, the parametrized Weinstein Lagrangian neighborhood theorem (see \cite[Exercise~3.3.4]{oh-book}) provides tubular neighborhoods $U_t$ of $C_0$ and a 1-parameter family of diffeomorphisms
\[ h_t: U_0 \xrightarrow{\sim} U_t, \]
with $h_0=\id_\Sigma$, which fix $C_0$ pointwise and satisfy $h_t^*(\phi_t^*\omega) = \omega$.  This $h_t$ is constructed as the time-1 flow of a vector field $v_t$, which is defined on a neighborhood $U$ of $C_0$ and satisfies $v_t|_{C_0} = 0$ for all $t$.  We take a smaller open neighborhood $V$ of $C_0$ whose closure satisfies
\[ \overline{V} \subset U \cap (\Sigma \setminus \overline{W}), \]
and a smooth cut-off function $\rho: \Sigma \to [0,1]$ which is supported on $U \setminus \overline{W}$ and satisfies $\rho|_V\equiv 1$.  We then let $\tilde{h}_t: \Sigma \to \Sigma$ be the time-$t$ flow of $\rho\cdot v_t$, and define
\[ \varphi_t = \phi_t \circ \tilde{h}_t. \]
By construction, it follows that $\varphi_t(C_0) = \phi_t (\tilde{h}_t(C_0)) = \phi_t(C_0) = C_t$; that both $\phi_t$ and $\tilde{h}_t$ fix $W$ pointwise, hence so does $\varphi_t$; and that
\[ \left.\varphi_t^*\omega\right|_V = \left.\tilde{h}_t^*(\phi_t^*\omega)\right|_V = \left.h_t^*(\phi_t^*\omega)\right|_V = \omega|_V \]
as desired.

\vspace{1em}

\noindent{\bf Step 2}: Construct the isotopy $\psi_t: \Sigma \to \Sigma$.

Let $\Sigma_i$ denote the closures of the various components of $\Sigma' \setminus C_0$.  Then by hypothesis, and by the fact that $\varphi_t$ fixes $\overline{\Sigma\setminus\Sigma'} \subset W$, each of the areas
\[ \int_{\Sigma_i} \varphi_t^*\omega \]
is constant as a function of $t$, so we apply Moser's theorem, as generalized by Banyaga \cite{banyaga} to compact manifolds with boundary, to the surfaces $\Sigma_i$.  This theorem asserts that since $\int_{\Sigma_i} \varphi_t^*\omega$ is constant, there is an isotopy
\[ \chi^i_t: \Sigma_i \to \Sigma_i, \qquad 0 \leq t \leq 1 \]
such that $(\chi^i_t)^*(\varphi_t^*\omega) = \varphi_0^*\omega = \omega$ and $\chi^i_t|_{\partial\Sigma_i} = 0$ for all $t\in[0,1]$.  Banyaga's proof begins by constructing a 1-form $\alpha^i_t$ on $\Sigma_i$ satisfying
\begin{align*}
\frac{d}{dt}\left(\left.\varphi_t^*\omega\right|_{\Sigma_i}\right) &= d\alpha^i_t, &
\left.\alpha^i_t\right|_{\partial\Sigma_i} &= 0,
\end{align*}
and then one takes $\chi^i_t$ to be the flow of a vector field $v^i_t$ satisfying $\iota_{v^i_t}(\varphi_t^*\omega) = -\alpha^i_t$.  We remark that by condition~\eqref{i:separating-condition}, each component of $C_0 \cap \Sigma_i$ belongs to the boundary of $\Sigma_i$ rather than the interior, so that $\alpha^i_t = 0$ along all of $C_0 \cap \Sigma_i$.

We now modify the above so that $\chi^i_t$ fixes not just $\partial \Sigma_i$ but a whole neighborhood thereof.  Let $N \subset \Sigma_i$ be a collar neighborhood of $\partial\Sigma_i$ contained in $V \cup W$, and identify $N \cong [0,1]\times\partial\Sigma_i$ so that $\partial\Sigma_i$ is identified with $\{0\}\times\partial\Sigma_i$.  We know that $\left.\alpha^i_t\right|_{\partial\Sigma_i}$ is zero, so it represents the zero class in $H^1_{\mathrm{dR}}(\partial\Sigma_i)$.  Since $N \subset V\cup W$, Step 1 guarantees that $\varphi_t^*\omega$ is constant on $N$, so then $\left.\alpha^i_t\right|_N$ is closed, and its cohomology class is sent to zero by the restriction map
\[ H^1_{\mathrm{dR}}(N) \xrightarrow{\sim} H^1_{\mathrm{dR}}(\partial\Sigma_i), \]
so it must be exact.   Thus we can write $\left.\alpha^i_t\right|_N = d\beta_t$, and we replace $\alpha^i_t$ with $\alpha^i_t - d(\rho\beta_t)$, where $\rho: [0,1]\to[0,1]$ is a smooth cutoff function satisfying $\rho(t)=1$ for $t \leq \frac{1}{3}$ and $\rho(t)=0$ for $t \geq \frac{2}{3}$.  This leaves $d\alpha^i_t$ unchanged, but now we have $\alpha^i_t \equiv 0$ on $[0,\frac{1}{3}] \times \partial\Sigma_i \subset N$.  It follows in the above construction that $v^i_t$ is zero on this neighborhood, so $\chi^i_t$ is stationary there as desired.

Having arranged for each $\chi^i_t$ to fix a neighborhood of $\partial\Sigma_i$, we can now glue the various $\chi^i_t$ together and extend by the identity on the rest of $\Sigma$ to get a single smooth isotopy
\[ \chi_t: \Sigma \to \Sigma \]
such that $\chi_t^*(\varphi_t^*\omega) = \omega$, and such that $\chi_t$ is the identity on a neighborhood $U$ of $C_0 \cup \overline{\Sigma\setminus\Sigma'}$.  We therefore define
\[ \psi_t = \varphi_t \circ \chi_t: \Sigma \to \Sigma, \]
which fixes the neighborhood $U \cap W$ of $\overline{\Sigma\setminus\Sigma'}$.  By construction we have
\[ \psi_t(C_0) = \varphi_t(\chi_t(C_0)) = \varphi_t(C_0) = C_t \]
and $\psi_t^* \omega = \chi_t^*(\varphi_t^*\omega) = \omega$, so $\psi_t$ is the desired isotopy.
\end{proof}

\bibliographystyle{myalpha}
\bibliography{References}

\end{document}